\UseRawInputEncoding
\documentclass[11pt]{article}

\usepackage{calrsfs,amsmath,mathrsfs,amsthm, amssymb, graphicx,amsfonts,color,stmaryrd,upgreek,authblk,dsfont} 

\newtheorem{theorem}{\bf Theorem}[section]
\newtheorem{lemma}[theorem]{Lemma}

\newtheorem{coro}{\bf Corollary}[section]
\newtheorem{assum}{\bf Assumption}[section]
\newtheorem{rem}{\bf Remark}[section]

\newlength{\kaka}

\newcommand{\ahref}[2]{}

\newcommand{\obs}{^{\text{obs}}}

\newcommand{\mD}{\mathcal{D}}
\newcommand{\mE}{\mathcal{E}}

\newcommand{\beq}{\begin{equation}}
\newcommand{\eeq}{\end{equation}}
\newcommand{\lb}{\label}

\newcommand{\bea}{\begin{eqnarray}}
\newcommand{\eea}{\end{eqnarray}}
\newcommand{\bxr}{\begin{array}}
\newcommand{\exr}{\end{array}}

\newcommand\exs{\hspace*{0.4mm}}
\newcommand\xxs{\hspace*{0.2mm}}

\newcommand\nxs{\hspace*{-0.2mm}}

\newcommand{\norms}[1]{\parallel\! #1 \!\parallel}

\newcommand{\bfG} {\boldsymbol{G}}

\newcommand{\bSig} {\boldsymbol{\Sigma}}

\newcommand{\bC} {\boldsymbol{C}}
\newcommand{\bK} {\boldsymbol{K}}

\newcommand{\bR} {\boldsymbol{\sf R}}

\newcommand{\btau} {\boldsymbol{\tau}}
\newcommand{\bn} {\boldsymbol{n}}
\newcommand{\ba} {\boldsymbol{a}}

\newcommand{\bx} {\boldsymbol{x}}
\newcommand{\by} {\boldsymbol{y}}

\newcommand{\bI} {\boldsymbol{I}}

\newcommand{\sip} {\!\cdot\!}

\newcommand{\R}{\mathbb{R}}

\newcommand{\OOd}{{\Omega}}

\newcommand{\bzero}{\boldsymbol{0}}

\newcommand{\bu} {\boldsymbol{u}}

\newcommand{\bt} {{\boldsymbol{t}}}

\newcommand{\bv} {\boldsymbol{v}}

\newcommand{\bxi} {\boldsymbol{\xi}}

\newcommand{\bw} {\boldsymbol{w}}

\newcommand{\bPsi}{\boldsymbol{\Psi}}

\begin{document}

\title{\bf Ultrasonic imaging in highly heterogeneous backgrounds}

\author[1,2]{Fatemeh Pourahmadian\thanks{Corresponding author email: fatemeh.pourahmadian@colorado.edu}}
\author[3]{Houssem Haddar}
\affil[1]{Department of Civil, Environmental \& Architectural Engineering, University of Colorado Boulder, USA}
\affil[2]{Department of Applied Mathematics, University of Colorado Boulder, USA}
\affil[3]{INRIA, Center of Saclay Ile de France and UMA, ENSTA Paris Tech, Palaiseau Cedex, FRANCE}
\date{} 
\renewcommand\Affilfont{\itshape\small}

\maketitle

\vspace{-10mm}
\begin{abstract}
 
 This work formally investigates the differential evolution indicators as a tool for ultrasonic tracking of elastic transformation and fracturing in randomly heterogeneous solids. Within the framework of periodic sensing, it is assumed that the background at time $t_\circ$ contains \emph{(i)}~a multiply connected set of viscoelastic, anisotropic, and piece-wise homogeneous inclusions, and \emph{(ii)}~a union of possibly disjoint fractures and pores. The support, material properties, and interfacial condition of scatterers in \emph{(i)} and \emph{(ii)}  are unknown, while elastic constants of the matrix are provided. The domain undergoes progressive variations of arbitrary chemo-mechanical origins such that its geometric configuration and elastic properties at future times are distinct. At every sensing step $t_\circ, t_1, \ldots$, multi-modal incidents are generated by a set of boundary excitations, and the resulting scattered fields are captured over the observation surface. The test data are then used to construct a sequence of wavefront densities by solving the spectral scattering equation. The incident fields affiliated with distinct pairs of obtained wavefronts are analyzed over the stationary and evolving scatterers for a suit of geometric and elastic evolution scenarios entailing both interfacial and volumetric transformations. The main theorem establishes the invariance of pertinent incident fields at the loci of static fractures and inclusions between a given pair of time steps, while certifying variation of the same fields over the modified regions. These results furnish a basis for theoretical justification of differential evolution indicators for imaging in complex composites which, in turn, enable the exclusive tomography of evolution in a background endowed with many unknown features.                          
 
\end{abstract}


\section{Introduction}

Many critical components in aerospace structures and energy systems are comprised of highly heterogeneous composites~\cite{pokh2017,bond2015}. Examples include (a) single- and polycrystalline superalloys deployed in aeroengine and gas turbine blades~\cite{bali2022},~(b) interpenetrating phase metamaterials such as SiC-SiC composites used in accident-tolerant nuclear fuel claddings~\cite{arre2022,terr2018}, and (c) multifunctional polymer matrix composites with a wide spectrum of applications thanks to their exceptional mechanical properties~\cite{del2021,van2018}. The topology and characteristics of such materials at micro- and meso- scales are often unknown, or only known to a limited extent because of variabilities in the manufacturing process~\cite{meier2019, netz2022} and/or aging~\cite{yan2018}. In addition, mechanisms of deterioration via corrosion, fatigue, irradiation, and thermal cycling are yet to be fully understood. These processes, however, are responsible for continuous microstructural evolution leading to inevitable development of micro/macro cracks and volumetric damage zones which may result in the loss of functional performance in key components~\cite{beau2017,roy2017}. 

Recent developments in sensing technology have resulted in a suit of imaging solutions germane to complex environments~\cite{amin2017,faria2022,zhao2018, garc2018,bych2018,zhan2022,bakre2022,zhan2020}. State-of-the-art examples include:~penetrating-radar techniques~\cite{amin2017}, infrared thermography~\cite{faria2022}, laser shearography~\cite{zhao2018}, X-ray computed tomography~\cite{garc2018}, acoustic tomography~\cite{bych2018}, ultrasonic surface wave methods~\cite{zhan2022}, nonlinear ultrasound~\cite{bakre2022}, and laser ultrasonic imaging~\cite{zhan2020}. Among which, ultrasonic sensing often emerges as the preferred (or the only feasible) imaging modality in many applications. Laser ultrasonics~\cite{wilc2020,zhan2019,davi1993}, in particular, has come under the spotlight for enabling non-contact actuation and measurement that is crucial for high-fidelity in-situ monitoring of fabrication processes and advanced manufacturing~\cite{zime2021,zhan2020, liu2022}. 

Existing approaches to ultrasonic waveform inversion mostly rely on~{(a)}~distinct patterns in the measured scattered field associated with certain modes of propagation,~{(b)}~specific sensing configurations, and~(c)~major postulates on the nature of wave motion in the background which, by and large, forgo the uncertain (yet important) scattering signatures affiliated with the specimen's microstructure. Such attributes expedite the data processing, yet entail the following impediments:~{(i)}~unstable reconstructions featuring many artifacts,~{(ii)}~significant errors in heterogeneous and anisotropic backgrounds where multiple scattering generates remarkable wave dispersion and attenuation,~{(iii)}~major restrictions on the geometry of incident and/or measurement grids,~{(iv)}~limited scalability beyond the controlled laboratory environment. Thus, there is a critical need for next-generation imaging solutions that carefully integrate state-of-the-art instrumentation and advanced data analytic solutions to enable fast (yet robust) ultrasonic tomography of complex processes in uncertain or unknown environments.  

Ongoing efforts in this vein are mainly focused on~(a) optimization-based full-waveform inversion, and~(b) machine learning (ML). Inverse algorithms in (a) have so far been associated with tardy reconstructions due to their high computational cost. Lately, a few studies showed that the latter may be addressed by leveraging deep learning solutions pertinent to partial differential equations such as physics-informed neural networks~\cite{chen2020,chen2022}. However, the majority of paradigms in (b) make use of ML principles within the framework of existing logics for ultrasonic imaging so that the above-mentioned barriers are not fundamentally resolved. Nonetheless, ML schemes are shown to facilitate the implementation of various imaging solutions, and may serve as effective post-processing tools for image enhancement~\cite{van2019,cant2022,arin2020}.

In applied mathematics, in parallel, over two decades of research in inverse scattering and transmission eigenvalues has given rise to a suit of rigorous algorithms for non-iterative waveform inversion~\cite{cako2020,cako2016,audi2017,guzi2015,bonn2019}. Recently, sampling-based approaches to inverse scattering have been applied to laser-ultrasonic test data~\cite{Venk2022}. The results demonstrate superior reconstructions in terms of quality and resolution compared to conventional methods. In addition, introduction of the differential evolution indicators~\cite{AudiDLSM} as a tool for waveform tomography in unknown media showcase a unique opportunity to achieve the above-mentioned goal for real-time, laser-based monitoring. The differential indicators have so far been theoretically established for (i) acoustic imaging of obstacles in periodic and random media~\cite{nguy2020,audi2021}, and (ii) elastic-wave imaging of fractures in monolithic solids~\cite{pour2020,pour2021(3)}. A rigorous justification of this imaging modality for its potentially dominant field of application i.e., ultrasonic tomography in complex composites -- where the background features a random distribution of heterogeneities and discontinuities of unknown support and material characteristics which are subject to both interfacial and volumetric evolution, is still lacking. The present study is an effort toward establishing the differential evolution indicators in the general case of solids. In this vein, a special attention is paid to pose the forward scattering problem in a broad sense by taking advantage of contributions on the nature of transmission eigenvalues in elastodynamics~\cite{char2002, char2002(2), char2008, char2006, bell2010, bell2013,cako2021}. 

More specifically, the direct problem formulates sequential experiments conducted on a randomly structured composite such that at every sensing step $t_\circ, t_1, \ldots$, the specimen features an arbitrary networks of pores and fractures along with a set of viscoelastic, anisotropic, and heterogeneous inclusions embedded in an elastic matrix. Properties of the binder are assumed to be known, while the support, material properties, and interfacial condition of scatterers are a-priori unknown and subject to spatiotemporal evolution. In this setting, boundary excitations at every time step give rise to distinct scattering footprints on the measurement surface. The idea is to use the sequence of scattered field measurements to design an imaging functional endowed with appropriate invariance with respect to stationary scatterers between any pairs of time steps such that the associated reconstructions uniquely expose the support of mechanical evolution in a given timeframe without the need to image the entire domain which may be insurmountable. To this end, the conditions for wellposedness of the forward problem are identified. The set of spectral scattering equations is then defined for waveform inversion. At every time step, the affiliated scattering operator and its related properties are carefully analyzed in order to construct a sequence of approximate solutions to the scattering equation with strong convergence characteristics under a certain condition. The obtained solutions i.e., wavefront densities are then used to specify a set of incident fields over a generic model of the background, which forms the basis for differential imaging indicators. Next, the incidents corresponding to distinct pairs of wavefronts are analyzed over the stationary and evolving scatterers for a suit of geometric and elastic evolution configurations. In the general case of solids, the latter involves a number of novel scenarios including (a) fracturing at bimaterial interfaces, (b) elastic transformation and/or expansion of fractured inclusions, and (c) elastic conversion of microcracked damage zones. The main theorem establishes the invariance of incident fields at the loci of stationary fractures and inclusions, while certifying variation of the same fields over the evolved regions. These results pave the way for differential tomography of evolution in unknown backgrounds.              
 
In the sequel, \textcolor{black}{Section}~\ref{S1} provides the geometric and mechanistic description of the direct scattering problem. \textcolor{black}{Section}~\ref{Sc1} specifies the fundamental properties of scattering operators. The main theorems in \textcolor{black}{Section}~\ref{DEI} establish the behavior of differential indicators in the case of composites. \textcolor{black}{Section}~\ref{IR} is dedicated to a numerical implementation and discussion of the results.

\section{Preliminaries}\label{S1}

Consider periodic illumination of mechanical evolution in a randomly structured composite shown in Fig.~(\ref{comp_ps}). At the first sensing step $t = t_{\circ}$,  the specimen $\mathcal{B} \subset \R^3$ is comprised of~(i) a linear, elastic, isotropic and homogeneous binder of mass density $\rho \in \R$ and Lam\'{e} parameters $\mu, \lambda \in \R$,~(ii) a union of bounded inclusions $\mD_\circ^\star \cup \mD_\circ^o = \mD_\circ \subset \mathcal{B}$ with Lipschitz boundaries composed of penetrable $\mD_\circ^\star$ and impenetrable $\mD_\circ^o$ components where in the former case, the inclusions may be viscoelastic, anisotropic, and multiply connected, and~(iii) a network of discontinuity surfaces $\Gamma_{\circ} \subset \mathcal{B}$ characterized by the complex-valued and heterogeneous interfacial stiffness matrix $\bK_{\circ}(\bxi)$ where $\bxi \in \Gamma_{\circ}$ is the position vector. The specimen may be exposed to irradiation or chemical reactions as common producers of interfacial damage~\cite{lach2022}, and/or subject to thermal cycling, fatigue, and shock-waves which are mostly responsible for volumetric degradation~\cite{lieo2021} so that at any future sensing steps $t = t_i \nxs> t_\circ$, $i \in \mathbb{N}$, the domain $\mathcal{B}$ features an evolved set of inclusions $\mD_i^\star \cup \mD_i^o = \mD_i \subset \mathcal{B}$ and interfaces $\Gamma_{i} \subset \mathcal{B}$ such that $\mD_{i-1} \subset \mD_{i}$ and $\Gamma_{i-1}\!\setminus\nxs \overline{\mD^o_{i}} \subset \Gamma_{i}$, $\forall i \in \mathbb{N}$. For further clarity, let $\rho_\kappa = \rho_\kappa(\bxi) >0$ and $\bC_\kappa = \bC_\kappa(\bxi)$, $\kappa = \circ, 1, \ldots $, designate the mass density and (complex-valued) viscoelasticity tensor associated with the penetrable obstacles $\mD_\kappa^\star$ at $t_{\kappa}$. Here, $(\bC_\kappa, \rho_\kappa)$ are understood in a piecewise-constant sense i.e.,~$\mD_\kappa^\star$ can be decomposed into $N_\kappa^\star$ open, simply connected, and non-overlapping subsets $\mD_\kappa^n \subset \mD_\kappa^\star$ (of Lipschitz boundaries) where both $\rho_\kappa$ and $\bC_\kappa$ are constants $\forall \bxi \in \mD_\kappa^n$, and $\overline{\mD_\kappa^\star} = {\textstyle \bigcup_{n \exs =1}^{N_\kappa^\star}} \overline{\mD_\kappa^n}$. It should be noted that $\forall t_\kappa$, $\mD_\kappa^\star$ and $\mD_\kappa^o$ are assumed to be disjoint i.e.,~$\overline{\mD_\kappa^\star} \cap \overline{\mD_\kappa^o} = \emptyset$. In addition, the support of $\Gamma_{\kappa}$ may be decomposed into $N_\kappa$ smooth open subsets $\Gamma_n \subset \Gamma_{\kappa}$, $n=1,\ldots, N_\kappa$, such that $\Gamma_{\kappa} = {\textstyle \bigcup_{n \exs=1}^{N_\kappa}} \Gamma_n$. The support of $\Gamma_n$ may be arbitrarily extended to a closed Lipschitz surface $\partial D_n$ of a bounded simply connected domain $D_n$, so that the unit normal vector $\bn$ to $\Gamma_n$ coincides with the outward normal vector to $\partial D_n$. We assume that $\Gamma_n$ is an open set (relative to $\partial D_n$) with a positive surface measure. 

In this setting, let ${\mE}^\star_{i} \cup {\mE}^o_{i}$ and $\hat\Gamma_i \cup \tilde\Gamma_i$ respectively specify the support of volumetric and interfacial evolution within $[t_{i-1},\,t_i]$ for $i \in \mathbb{N}$ such that 
\beq \lb{HT}
\begin{aligned}
&\overline{\hat{\Gamma}_{i}} \, \colon\!\!\!= \overline{{\Gamma}_{i}\nxs\setminus\nxs\overline{{\Gamma}_{i-1}}}, \quad\,\,\, \overline{\tilde{\Gamma}_{i}} \, \colon\!\!\!= \overline{\big\lbrace \bxi \in {\Gamma}_{i-1}\!\setminus\nxs \overline{\mD^o_{i}}: \,\,\, \bK_{i-1}(\bxi) \exs \neq\, \bK_{i}(\bxi) \big\rbrace}, \\
& \overline{{\mE}^o_{i}} \,\xxs \colon\!\!\!=\, \overline{\mD_i^o\nxs\setminus\nxs\overline{\mD_{i-1}^o}}, \quad \overline{{\mE}^\star_{i}} \,\xxs \colon\!\!\!=\, \overline{\text{supp}(\bC_{i} - \bC_{i-1}) \exs\cup\, \text{supp}(\rho_{i} - \rho_{i-1})},  \quad i = 1,2,\ldots.
\end{aligned}
\eeq   

 At $t = t_{\circ}$, the specimen's external boundary $\partial \mathcal{B}$ and the binder's constants $(\lambda, \mu, \rho)$ are known which may be used to define a \emph{baseline model} of the background. On the other hand, the support of pre-existing scatterers $\mD_\circ \cup \Gamma_{\circ}$ and their designated properties $(\bC_\circ, \rho_\circ, \bK_\circ)(\bxi)$ are unknown. Given sequential sensory data at $t_{i-1}$ and $t_i$ for $i \in \mathbb{N}$, the objective is to reconstruct the support of volumetric and interfacial evolution $\mE^\star_i \cup \mE^o_i \cup \hat\Gamma_i \cup \tilde\Gamma_i$ defined by~\eqref{HT}.

\begin{figure}[tp!]
\vspace*{-0mm}
\begin{center} 
\includegraphics[width=0.8\linewidth]{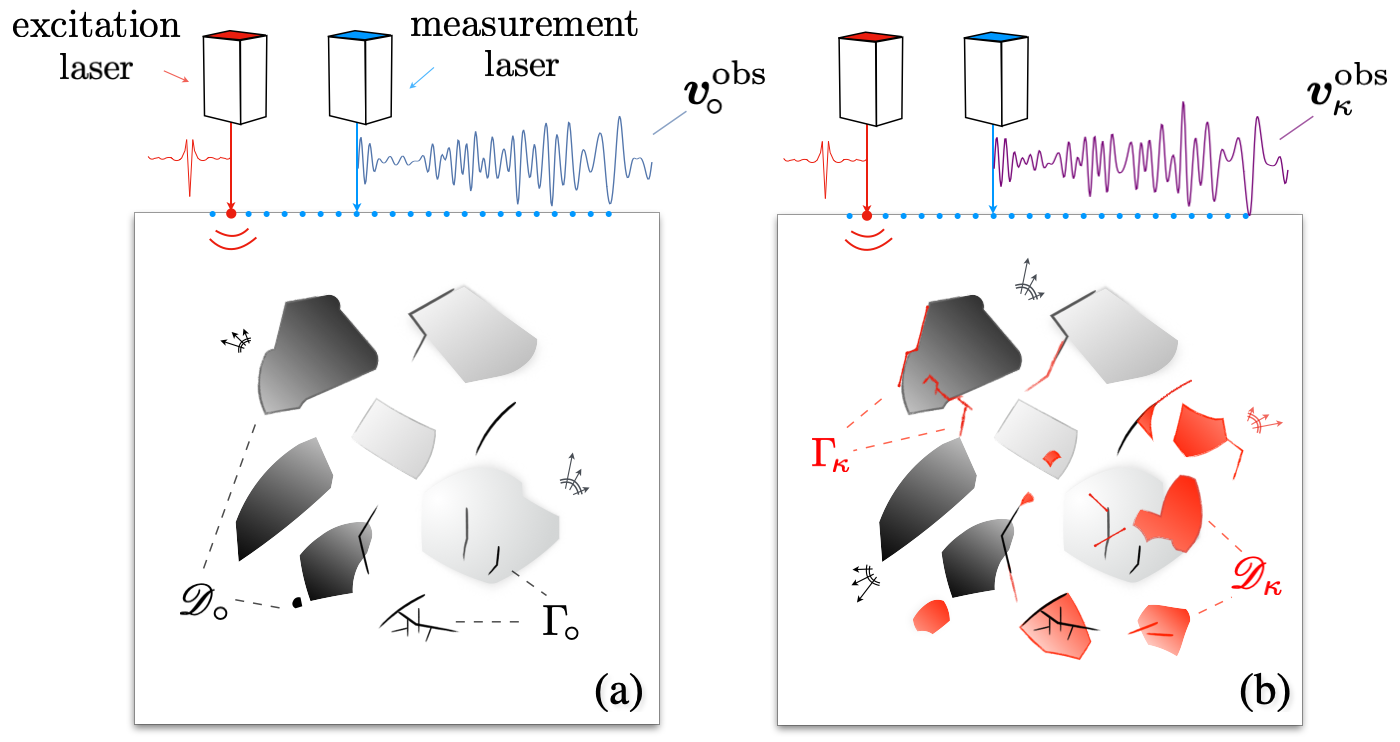}
\end{center} \vspace*{-6.5mm}
\caption{Ultrasonic sensing of evolution in a heterogeneous composite featuring a network of arbitrary inclusions, fractures, and pores of unknown distribution, elastic properties, and interfacial conditions: {\bf(a)}~{\emph{primary experiments}} conducted at $t \nxs=\nxs t_\circ$ when a set of boundary excitations, interacting with the unknown pre-existing scatterers $\mD_\circ \cup \Gamma_{\circ}$, induce the scattered field $\bv_{\nxs\circ}{\!\!\nxs\obs}\!$ on the observation surface, and~{\bf(b)}~{\emph{secondary experiments}} performed sequentially at later times $t_i \!=\! \lbrace t_1, t_2, ... \rbrace$ when new and evolved scatterers $\mD_\kappa \cup \Gamma_{\kappa}$ lead to distinct waveform measurements $\bv_{\kappa}{\!\!\!\nxs\obs}$.}
\label{comp_ps}\vspace*{-5.0mm}
\end{figure}


\begin{assum}\label{tfc}
Let $\Gamma^{o}_{\kappa} \subset \Gamma_{\kappa}$ denote the union of traction-free cracks at $t_\kappa$ such that 
\beq \nonumber
\overline{{\Gamma}_{\kappa}^{o}} \, \colon\!\!\!= \overline{\big\lbrace \bxi \in {\Gamma}_{\kappa}: \,\,\, \bK_{\kappa}(\bxi) \exs =\, \bzero \big\rbrace}, \qquad \kappa = \circ,1,\ldots,
\eeq
then $\mathcal{B} \nxs\setminus\!  \overline{\mD^{o}_{\kappa} \cup \Gamma^{o}_{\kappa}} $ remains connected $\forall t_\kappa$.
\end{assum}
 
\begin{assum}\label{CP}
\vspace*{-1mm}
Let $\Re(\cdot)$ and $\Im(\cdot)$ respectively denote the real and imaginary parts of a complex-valued quantity, and recall that the fourth-order tensor $\bC_\kappa(\bxi)$ represents the viscoelastic and anisotropic behavior of inclusions $\mD_\kappa^\star$ at $t_\kappa$. Then, the real part of $\bC_\kappa$ is bounded by piecewise-constant and strictly positive functions $\text{\sf c}_\kappa$ and $\text{\sf C}_\kappa$, while the magnitude of its imaginary part is constrained by piecewise-constant and non-negative functions $\text{\sf v}_\kappa$ and $\text{\sf V}_\kappa$ such that $\forall \boldsymbol{\Upphi} \in \mathbb{C}(\mD_\kappa^\star)^{3\times3\nxs}$,
\beq \lb{CkP}
\!\left\{\begin{array}{l}
\!\!\text{\sf c}_\kappa {|}\boldsymbol{\Upphi} {|}^2 \,\leqslant\, \Re(\boldsymbol{\Upphi} \exs\colon\nxs \bC_\kappa \colon \overline{\boldsymbol{\Upphi}}) \,\leqslant\, \text{\sf C}_\kappa {|} \boldsymbol{\Upphi} {|}^2 \qquad \text{in  $\mD_\kappa^\star$}  \\*[1.5mm]
\!\!\text{\sf v}_\kappa {|}\boldsymbol{\Upphi} {|}^2 \,\leqslant\, -\Im(\boldsymbol{\Upphi} \exs\colon\nxs \bC_\kappa \colon \overline{\boldsymbol{\Upphi}}) \,\,\leqslant\, \text{\sf V}_\kappa {|} \boldsymbol{\Upphi} {|}^2 \quad\, \text{in  $\mD_\kappa^\star$}
\end{array}\right.,
 \qquad \kappa = \circ, 1, \ldots.
\eeq   

Also, the interfacial stiffness matrix $\bK_{\nxs\kappa} \in L^\infty(\Gamma_\kappa)^{3\times 3}\!$ is symmetric $\forall \kappa$, while satisfying \mbox{$\overline{\boldsymbol{\varphi}} \cdot \Im \bK_{\nxs\kappa}(\bxi)\cdot \boldsymbol{\varphi} \leqslant 0$, $\forall \boldsymbol{\varphi}\in\mathbb{C}(\Gamma_\kappa)^3$}. 
\end{assum}

\begin{assum}\label{ITP_A1}
\vspace*{-1mm}
Let $\bC = \lambda\bI_2\!\otimes\!\bI_2 + 2\mu\bI_4$ denote the binder's fourth-order elasticity tensor, wherein $\bI_m \,(m\!=\!2,4)$ designates the $m$th-order symmetric identity tensor. Then, given the Poisson's ratio $\nu = \frac{\lambda}{2(\lambda+\mu)}$ and ${\mathcal{B}}^-_\kappa \colon \!\!\! = {\mathcal{B}} \exs \backslash \overline{\mD_\kappa \nxs\cup \Gamma_{\kappa}}$, observe that $\forall \xxs \boldsymbol{\Uppsi} \in \mathbb{C}({\mathcal{B}}^-_\kappa)^{3\times3\nxs}$, 
\beq\label{Clim}\nonumber
\!\left\{\begin{array}{l}
\!\!2\mu {|}\boldsymbol{\Uppsi} {|}^2 \,\leqslant\, \Re(\boldsymbol{\Uppsi} \colon\! \bC \colon \overline{\boldsymbol{\Uppsi}}) \,\leqslant\, (3\lambda+2\mu) {|} \boldsymbol{\Uppsi} {|}^2 \qquad \text{\,\,\, for \, $ 0 \,<\, \nu \,<\, \frac{1}{2}$} \\*[1.5mm]
\!\!(3\lambda+2\mu) {|}\boldsymbol{\Uppsi} {|}^2 \,\leqslant\, \Re(\boldsymbol{\Uppsi} \colon\! \bC \colon \overline{\boldsymbol{\Uppsi}}) \,\leqslant\, 2\mu  {|} \boldsymbol{\Uppsi} {|}^2 \qquad \text{for \, $ -1 \,<\, \nu \,<\, 0$}
\end{array}\right.,
\quad \Im(\boldsymbol{\Uppsi} \colon\! \bC \colon \overline{\boldsymbol{\Uppsi}}) ~=~ 0.
\vspace*{0.5mm}
\eeq
Next, invoke $\mD_\kappa^n \subset \mD_\kappa^\star$ with constant $(\bC_\kappa, \rho_\kappa)$, and let $\text{\sf c}_\kappa^n$, $\text{\sf C}_\kappa^n$, $\text{\sf v}_\kappa^n$, and $\text{\sf V}_\kappa^n$ represent the respective values of $\text{\sf c}_\kappa$, $\text{\sf C}_\kappa$, $\text{\sf v}_\kappa$, and $\text{\sf V}_\kappa$ in each $\mD_\kappa^n$, $n \in \lbrace 1,\ldots, N^\star_\kappa \rbrace$ and $\kappa \in \lbrace \circ, 1, \ldots \rbrace$. Then, in light of~\cite{bell2010}, $\forall\kappa$ 
\beq\label{ITP_A2}\nonumber
\big( \, \rho \,<\, \rho_\kappa \,\,\land\,\, \max\lbrace 3\lambda+2\mu, 2\mu \rbrace  \,<\, \min\lbrace \text{\sf c}_\kappa^n \rbrace \exs \big)  \,\lor\, \big( \,  \rho \,>\, \rho_\kappa \,\,\land\,\, \min\lbrace 3\lambda+2\mu, 2\mu \rbrace  \,>\, \max\lbrace \text{\sf c}_\kappa^n \rbrace \exs \big).
\eeq
\end{assum}

\paragraph*{Experiments.}
The domain $\mathcal{B}$ is subject to periodic inspections at time steps $t_\kappa = \lbrace t_\circ, t_1, \ldots \rbrace$. At every $t_\kappa$, the specimen is excited by a combination of ultrasonic sources on its external boundary $\partial \mathcal{B}_t$ so that the corresponding incident field $\bu^{\textrm{f}} \in H^1(\mathcal{B})^3$ in the baseline model is governed by 
\beq\lb{uf}
\begin{aligned}
&\nabla \exs\sip\exs \bC \exs \colon \! \nabla \bu^{\textrm{f}}(\bxi) \,+\, \rho \exs \omega^2 {\bu}^{\textrm{f}}(\bxi) ~=~ \bzero, \quad & \bxi \in {\mathcal{B}}, \\*[0.0mm]
&\bn \exs\sip\exs \bC \exs \colon \!  \nabla  \bu^{\textrm{f}}(\bxi)~=~\btau(\bxi),  \quad & \bxi \in \partial {\mathcal{B}}_t, \\*[0.0mm]
& \bu^{\textrm{f}}(\bxi)~=~\bzero,   \quad &\bxi \in {\partial\mathcal{B} \nxs \setminus \!{\partial {\mathcal{B}}_t}}.
\end{aligned}   
\eeq  

Here, the illumination frequency $\omega>0$ is selected such that the shear wavelength $\lambda_s = 2 \pi \sqrt{\mu\exs /(\rho \omega^2)}$ is sufficiently smaller than the characteristic length scale of the sought-for objects; $\bn$ is the unit outward normal to the sample's boundary $\partial \mathcal{B}$; $\btau(\bxi)$ represents the external traction on the Neumann part of the boundary $\partial {\mathcal{B}}_t \subset \partial {\mathcal{B}}$ which includes the source input. It is assumed that $\text{supp}(\partial\mathcal{B} \setminus \overline{\partial {\mathcal{B}}_t}) = \emptyset$, i.e.,~a set of fixed boundary points of zero surface measure prevent the rigid body motion. Henceforth, the homogeneous Dirichlet part of the boundary will be implicitly indicated. At every sensing step $t_\kappa$, the interaction of $\bu^{\textrm{f}}$ with the hidden scatterers $\Gamma_{\kappa} \cup \mD_\kappa$ gives rise to the total field $(\bu^\kappa, \bw^\kappa) \in H^1({\mathcal{B}}^-_\kappa)^3 \nxs\times\nxs H^1(\mD^\star_\kappa \! \setminus \! \overline{\Gamma_{\kappa}})^3$ satisfying      
\beq\lb{uk}
\begin{aligned}
&\nabla \exs\sip\exs \bC \exs \colon \! \nabla \bu^\kappa(\bxi) \,+\, \rho \exs \omega^2 {\bu}^\kappa(\bxi) ~=~ \bzero, \quad &  \bxi \exs \in {\mathcal{B}}^-_\kappa, \\*[0.0mm]
&\nabla \exs\sip\exs [\exs \bC_\kappa(\bxi) \exs \colon \! \nabla \bw^\kappa\exs](\bxi) \,+\, \rho_\kappa(\bxi) \exs \omega^2 {\bw}^\kappa(\bxi) ~=~ \bzero, \quad &  \bxi \exs \in \mD^\star_\kappa  \setminus \overline{\Gamma_{\kappa}},   \\*[0.2mm]
&\bt[ \bu^\kappa](\bxi)~=~\bt[ \bw^\kappa](\bxi),  \,\,\, \bu^\kappa(\bxi)~=~\bw^\kappa(\bxi), \quad  & \bxi \in \partial \mD^\star_\kappa  \setminus \overline{\Gamma_\kappa}, \\*[0.2mm]
&\bt[\text{\bf{u}}](\bxi)~=~\bK_{\kappa}(\bxi) \llbracket \text{\bf{u}} \rrbracket(\bxi),  \,\,\, \llbracket \bt[\text{\bf{u}}] \rrbracket(\bxi)~=~\bzero,  \,\, & \bxi \in \Gamma_\kappa,  \\*[0.2mm] 
&\bt[ \bu^\kappa](\bxi)~=~\bzero,  \quad  & \bxi \in \partial \mD^o_\kappa, \\*[0.2mm]
&\bt[  \bu^\kappa](\bxi)~=~\btau(\bxi),  \quad & \bxi \in \partial {\mathcal{B}}_t.
\end{aligned}   
\eeq  
Here, 
\beq\label{nK0}\nonumber
\!\left\{\begin{array}{l}
\begin{aligned}
&\!\!  \bt[\bu^\kappa](\bxi)~=~\bn(\bxi) \exs\sip\exs \bC \exs\xxs\colon\nxs \nabla\bu^\kappa(\bxi), \quad   & \bxi \exs \in   \partial {\mathcal{B}}_t \cup \partial {\mD_\kappa}   \!\!\! \\*[0.2mm] 
& \!\! \bt[\bw^\kappa](\bxi)~=~\bn(\bxi) \exs\sip\exs \bC_\kappa(\bxi) \exs\colon\nxs \nabla\bw^\kappa(\bxi), \quad   & \bxi \exs \in   \partial {\mD^\star_\kappa}  \!\!\! 
\end{aligned}
\end{array}\right.,
\eeq
wherein $\bn$ is the unit \emph{outward} normal to $\partial {\mathcal{B}}$ and $\partial {\mD_\kappa}$. In addition, $\llbracket \text{\bf{u}} \rrbracket$ (\emph{resp}.~$ \llbracket \bt[\text{\bf{u}}] \rrbracket$) denotes the jump in displacement $\text{\bf{u}}$ (\emph{resp}.~traction $\bt[\text{\bf{u}}]$) across~$\Gamma_\kappa$ such that 
\beq\label{nK}\nonumber
 \llbracket \text{\bf{u}} \rrbracket~=~ \!\left\{\begin{array}{l}
\begin{aligned}
&\!\!\!\! \llbracket \bu^\kappa \rrbracket   & \text{on  ${\Gamma_{\kappa}} \!\setminus\nxs \overline{\mD^\star_\kappa}$}    \!\!\! \\*[0.1mm]
& \!\!\!\! \llbracket \bw^\kappa \rrbracket    & \text{on   ${\Gamma_{\kappa}} \nxs\cap\nxs {\mD^\star_\kappa}$} \!\!\! \\*[0.1mm]
& \!\!\!\!  \bu^\kappa\nxs - \bw^\kappa   & \text{on   ${\Gamma_{\kappa}} \nxs\cap\nxs \partial {\mD^\star_\kappa}$} \!\!\! 
\end{aligned}
\end{array}\right., \quad 
\llbracket\bt[\text{\bf{u}}]\rrbracket~=~ \!\left\{\begin{array}{l}
\begin{aligned}
&\!\!\!\! \llbracket\bt[\bu^\kappa]\rrbracket \,\colon\!\!\!=\, \llbracket \xxs \bn \xxs\sip\xxs \bC \exs\colon\!\nxs \nabla\bu^\kappa\rrbracket  & \text{on  ${\Gamma_{\kappa}} \!\setminus\nxs \overline{\mD^\star_\kappa}$} \!\!\! \\*[0.1mm]
&\!\!\!\!  \llbracket\bt[\bw^\kappa]\rrbracket \,\colon\!\!\!=\, \llbracket \xxs \bn \xxs\sip\xxs \bC_\kappa \colon\!\nxs \nabla\bw^\kappa\rrbracket   & \text{on   ${\Gamma_{\kappa}} \nxs\cap\nxs {\mD^\star_\kappa}$} \!\!\! \\*[0.1mm]
&\!\!\!\!  \bn \xxs\sip\xxs (\bC \xxs\colon\!\nxs \nabla\bu^\kappa - \bC_\kappa \colon\!\nxs \nabla\bw^\kappa)   & \text{on   ${\Gamma_{\kappa}} \nxs\cap\nxs \partial {\mD^\star_\kappa}$}  \!\!\! 
\end{aligned}
\end{array}\right.,
\eeq
where 
\beq\nonumber
\llbracket \boldsymbol{f} \rrbracket ~=~ \boldsymbol{f}^+ -\, \boldsymbol{f}^-, \qquad \boldsymbol{f}^{\pm}(\bxi) ~= \lim_{{\, h\to 0^+}}  \boldsymbol{f}(\bxi \pm h \bn(\bxi)), \qquad \bxi \in {\Gamma_{\kappa}}.
\eeq
Keep in mind that the unit normal vector $\bn$ to $\Gamma_\kappa$ is specified earlier. Also, the second of~\eqref{uk} should be understood as a shorthand for the set of $N^\star_\kappa$ governing equations over the respective homogeneous regions $\mD_\kappa^n (n = 1,\ldots,N^\star_\kappa)$, supplemented by the continuity conditions for displacement and traction across $\partial \mD_\kappa^n$ as applicable. 

\begin{assum}\label{EigFT}
$\omega>0$ is not an eigenvalue of the homogeneous form of~\eqref{uf} and~\eqref{uk}.
\end{assum}

Given~\eqref{uf} and~\eqref{uk}, the scattered field $\bv^\kappa \in H^1({\mathcal{B}}_\kappa^{-}\nxs \cup \mD^\star_\kappa \!\setminus\!  \overline{\Gamma_{\kappa}})^3$ is governed by 
\beq\lb{vk}
\begin{aligned}
&\nabla \exs\sip\exs \bC \exs \colon \! \nabla \bv^\kappa(\bxi) \,+\, \rho \exs \omega^2 {\bv}^\kappa(\bxi) \,+\, \mathds{1}({\mD^\star_\kappa  \!\setminus\!  \overline{\Gamma_{\kappa}}}) [\xxs \boldsymbol{f}_\kappa \nxs+ \nabla \xxs\sip\xxs \boldsymbol{\sigma}_\kappa](\bxi)~=~ \bzero, \quad &  \bxi \exs \in {\mathcal{B}}_\kappa^{-}\nxs \cup \mD^\star_\kappa \!\setminus\!  \overline{\Gamma_{\kappa}},   \\*[0.2mm]
&\bt[\bv^\kappa](\bxi) + \mathds{1}({\overline{\mD^\star_\kappa}  \nxs\cap\nxs  {\Gamma_{\kappa}}})  \bn \exs\sip\exs \boldsymbol{\sigma}_\kappa(\bxi)~=~\bK_{\kappa}(\bxi) \llbracket \bv^\kappa \rrbracket(\bxi) -  \bt^{\textrm{f}}(\bxi),  \hspace{-10mm}  & \bxi \in \Gamma_\kappa,  \\*[0.2mm] 
& \llbracket \bt[\bv^\kappa] + \mathds{1}({\mD^\star_\kappa  \nxs\cap\nxs  {\Gamma_{\kappa}}}) \bn \exs\sip\exs \boldsymbol{\sigma}_\kappa(\bxi) \rrbracket(\bxi)~=~ \mathds{1}({\partial\mD^\star_\kappa  \nxs\cap\nxs  {\Gamma_{\kappa}}}) \exs \bn \exs\sip\exs \boldsymbol{\sigma}_\kappa(\bxi),  \,\, & \bxi \in \Gamma_\kappa,  \\*[0.2mm] 
&\llbracket \bt[ \bv^\kappa] \rrbracket(\bxi)~=~\bn \exs\sip\exs \boldsymbol{\sigma}_\kappa(\bxi),  \,\,\, \llbracket \bv^\kappa \rrbracket(\bxi) ~=~\bzero, \qquad  & \bxi \in \partial \mD^\star_\kappa  \setminus \overline{\Gamma_\kappa}, \\*[0.2mm]
&\bt[ \bv^\kappa](\bxi)~=\,- \bt^{\textrm{f}}(\bxi),  \quad  & \bxi \in \partial \mD^o_\kappa, \\*[0.2mm]
&\bt[  \bv^\kappa](\bxi)~=~\bzero,  \quad & \bxi \in \partial {\mathcal{B}}_t.
\end{aligned}   
\eeq 
where $ \bt^{\textrm{f}} \exs\colon\!\!\!= \bn \exs\sip\exs \bC \exs \colon \!  \nabla  \bu^{\textrm{f}}$; $\bt[\bv^\kappa] = \bt^-[\bv^\kappa] \exs\colon\!\!\!= \bn \exs\sip\exs \bC \exs \colon \!  \nabla  \bv^{\kappa^-}\nxs$; and, 
$$ \boldsymbol{f}_\kappa(\bxi) \colon \!\!\!= (\rho_\kappa(\bxi) - \rho) \exs \omega^2 {\bw}^\kappa(\bxi), \qquad \boldsymbol{\sigma}_\kappa(\bxi) \colon \!\!\!= (\bC_\kappa(\bxi) - \bC) \exs \colon \! \nabla \bw^\kappa(\bxi), \qquad \bxi \in {\mD^\star_\kappa  \nxs\setminus\nxs  \overline{\Gamma_{\kappa}}}.$$

\paragraph*{Dimensional platform.}
 In what follows, all quantities are rendered dimensionless by taking $\rho$, $\mu$, and ${\ell_\circ}$ -- denoting the minimum length scale attributed to the hidden scatterers, as the respective reference scales for mass density, elastic modulus, and length -- which amounts to setting $\rho = \mu = \ell_\circ = 1$~\cite{Scaling2003}.

\paragraph*{Function spaces.}

It is known that stress singularities at the branch points of multiple intersecting fractures in an isotropic and homogeneous background is weaker than the classical crack-tip singularity~\cite{Lo1978,daux2000}. The latter is also the case for delamination cracks propagating along bi-material interfaces~\cite{muka1990}. 
High-order singularities may occur when a crack tip meets a bi-material interface in an angle~\cite{gosw2012, mend1984}, in which case it is shown that the contact laws in the vicinity of the crack tip may be modified such that the usual asymptotic forms for stress still applies~\cite{atkin1977}. In light of this, it will be assumed that $\llbracket \text{\bf{u}} \rrbracket \in \tilde{H}^{\frac{1}{2}}(\Gamma_\kappa)^3$ where
\beq\lb{funS}
\begin{aligned}
&H^{\pm \frac{1}{2}}(\Gamma_\kappa)^3 ~:=~\big\lbrace \boldsymbol{f} \big|_{\Gamma_\kappa}  \colon \,\,\, \boldsymbol{f} \in H^{\pm \frac{1}{2}}(\partial D)^3 \big\rbrace, \\*[0.0 mm]
& \tilde{H}^{\pm \frac{1}{2}}(\Gamma_\kappa)^3 ~:=~\big\lbrace  \boldsymbol{f} \in H^{\pm\frac{1}{2}}(\partial {D})^3 \exs \colon  \,\,\, \text{supp}(\boldsymbol{f}) \subset \overline{\Gamma}_\kappa  \big\rbrace.
\end{aligned}
\eeq

Here, ${D} = {\textstyle \bigcup_{n =1}^{N_\kappa}} {D}_n$ is a multiply connected Lipschitz domain of bounded support such that $\Gamma_\kappa \subset \partial {D}$. Recall that ${D}_n$ is an arbitrary extension of $\Gamma_n$ defined in the above. On invoking $H^{-1/2}(\Gamma_\kappa)^3$ and $\tilde{H}^{-1/2}(\Gamma_\kappa)^3$ as the respective dual spaces of $\tilde{H}^{1/2}(\Gamma_\kappa)^3$ and $H^{1/2}(\Gamma_\kappa)^3$, it follows that
\beq\lb{embb}
\tilde{H}^{\frac{1}{2}}(\Gamma_\kappa)^3 \,\subset\, H^{\frac{1}{2}}(\Gamma_\kappa)^3 \,\subset\, L^2(\Gamma_\kappa)^3 \,\subset\, \tilde{H}^{-\frac{1}{2}}(\Gamma_\kappa)^3 \,\subset\, H^{-\frac{1}{2}}(\Gamma_\kappa)^3.
\eeq

In this setting, $\bt[\text{\bf{u}}] \in H^{-1/2}(\Gamma_\kappa)^3$. For future reference, let us also define
\beq\lb{compd}
\begin{aligned}
& S(\mD^\star_\kappa \cup \Gamma_\kappa  \nxs \cup \partial \mD^o_\kappa) ~\colon\!\!\!=~L^2({\mD^\star_\kappa  \!\setminus\!  \overline{\Gamma_{\kappa}}})^3 \nxs\times L^{2}({\mD^\star_\kappa  \!\setminus\!  \overline{\Gamma_{\kappa}}})^{3\times 3} \nxs\times H^{-\frac{1}{2}}({\Gamma_\kappa})^3 \nxs\times H^{-\frac{1}{2}}(\partial \mD^o_\kappa)^3, \\*[0.2mm]
& \tilde{S}(\mD^\star_\kappa \cup \Gamma_\kappa \nxs \cup \partial \mD^o_\kappa) ~\colon\!\!\!=~L^2({\mD^\star_\kappa  \!\setminus\!  \overline{\Gamma_{\kappa}}})^3 \nxs\times L^{2}({\mD^\star_\kappa  \!\setminus\!  \overline{\Gamma_{\kappa}}})^{3\times 3} \nxs\times \tilde{H}^{\frac{1}{2}}({\Gamma_\kappa})^3 \nxs\times \tilde{H}^{\frac{1}{2}}(\partial \mD^o_\kappa)^3.
\end{aligned}
\eeq

\paragraph*{Wellposedness.}

Under Assumptions~\ref{tfc} and~\ref{CP}, observe that the direct scattering problem~\eqref{vk} is of Fredholm type, and thus, its wellposedness may be established by drawing from the unique continuation principles. See~Appendix A for details.

\paragraph*{Scattering signatures.}
By deploying Betti's reciprocal theorem, one obtains the following integral representation for the scattered fields $\bv^\kappa \xxs\colon\!\!\!=\xxs \bu^\kappa -\bu^{\textrm{f}}$, $\kappa = \lbrace \circ,1,2,\ldots \rbrace$, on the specimen's boundary. 
 
\beq\lb{IE}
\begin{aligned}
&\frac{1}{2} \bv^\kappa(\bxi) \,=\,\int_{\mD^\star_\kappa  \setminus  \overline{\Gamma_{\kappa}}} \big[\xxs (\rho_\kappa(\by) -\rho) \xxs \omega^2 \bw^\kappa(\by) \cdot \bfG(\bxi,\by)  -  \nabla \exs \bfG(\bxi,\by) \exs\colon \! (\bC_{\kappa\nxs}(\by)-\bC) \xxs\colon\! \nabla \bw^\kappa(\by) \xxs \big]  \textrm{d}V_{\by} ~+\,  \\*[0.2mm]   
&  \int_{\partial \mD^o_\kappa} \bu^\kappa(\by) \cdot \boldsymbol{T}(\bxi,\by) \, \textrm{d}S_{\by} \,+\, \int_{\Gamma_{\kappa}} \llbracket \text{\bf{u}} \rrbracket(\by) \cdot \boldsymbol{T}(\bxi,\by) \, \textrm{d}S_{\by}, \quad \boldsymbol{T}(\bxi,\by) \colon\!\!\!= \bn(\by) \cdot \bSig(\bxi,\by), \quad \bxi \in  \partial {\mathcal{B}_t}. 
\end{aligned}
\eeq

Here, $\boldsymbol{G}(\bxi,\by)$ is the Green's displacement tensor solving 
\beq\lb{Grf}
\begin{aligned}
&\nabla_{\nxs\by} \xxs\sip\exs \bC \exs \colon \! \nabla_{\nxs\by} \xxs \bfG(\bxi,\by) \,+\, \rho \exs \omega^2 \bfG(\bxi,\by) \,+\, \delta(\by-\bxi) \bI_{3\times 3}  ~=~ \bzero, \quad & \by \in {\mathcal{B}}\nxs\setminus\! \lbrace \bxi \rbrace, \\*[0.2mm]
&\bn \exs\sip\exs \bC \exs \colon \!  \nabla_{\nxs\by} \xxs \bfG(\bxi,\by)~=~\bzero,  \quad & \by \in \partial {\mathcal{B}}_t, 
\end{aligned}   
\eeq  
and $\bSig(\bxi,\by) \colon\!\!\! = \bC \exs \colon \! \nabla_{\nxs\by} \xxs \bfG(\bxi,\by)$ is the associated Green's stress tensor.

\section{Scattering operators\!\!}\lb{Sc1}

Let us define the scattering operator $\Lambda_\kappa: L^2(\partial {\mathcal{B}_t})^3 \to\exs L^2(\partial {\mathcal{B}_t})^3$ by
\beq\lb{ffo0}
\Lambda_\kappa(\btau) ~=~ \bv^\kappa \big{|}_{\partial {\mathcal{B}_t}}, \quad \kappa \,=\, \circ,1,2,\ldots,
\eeq
where $\bv^\kappa$ solves~\eqref{vk}. Then, there exists the factorization
\beq\lb{fact} 
\Lambda_\kappa ~=~ {\mathcal{S}}^*_\kappa \, T_\kappa \exs \mathcal{S}_\kappa
\eeq  
such that $\mathcal{S}_\kappa \colon L^2({\partial {\mathcal{B}_t}})^3 \,\rightarrow\, S(\mD^\star_\kappa \cup \Gamma_\kappa \nxs \cup \partial \mD^o_\kappa)$ is defined by 
\beq\lb{oH}
\mathcal{S}_\kappa(\btau) ~:=~ \big(\bu^{\textrm{f}}\big{|}_{\mD^\star_\kappa  \setminus  \overline{\Gamma_{\kappa}}}, \nabla \bu^{\textrm{f}}\big{|}_{\mD^\star_\kappa  \setminus  \overline{\Gamma_{\kappa}}}, \exs\bt[\bu^{\textrm{f}}]\big{|}_{\Gamma_\kappa}, \exs\bt[\bu^{\textrm{f}}]\big{|}_{\partial \mD^o_\kappa}\big),
\eeq  
whose adjoint operator $\mathcal{S}_\kappa^* \colon \tilde{S}(\mD^\star_\kappa \cup \Gamma_\kappa \nxs \cup \partial \mD^o_\kappa)  \rightarrow L^2({\partial {\mathcal{B}_t}})^3$ takes the form
\beq\lb{Hstar}
\mathcal{S}_\kappa^*(\boldsymbol{\upphi},\boldsymbol{\Upphi},\boldsymbol{\varphi},\boldsymbol{\phi}) ~:=~ \bv^*(\bxi), \qquad \bxi \in {\partial {\mathcal{B}_t}}, 
\eeq  
where $\bv^* \in H^1({\mathcal{B}}\setminus\nxs\overline{\Gamma_\kappa \cup \partial \mD_\kappa})^3$ solves
\beq\lb{up_a}
\begin{aligned}
&\nabla \exs\sip\exs \bC \exs \colon \! \nabla \bv^*(\bxi) \,+\, \rho \exs \omega^2 {\bv}^*(\bxi) \,+\, \mathds{1}({\mD^\star_\kappa  \nxs\setminus\nxs  \overline{\Gamma_{\kappa}}}) (\boldsymbol{\upphi} - \nabla \exs\sip\exs \boldsymbol{\Upphi}) ~=~ \bzero, \quad &  \bxi \exs \in {\mathcal{B}} \exs \backslash \overline{\Gamma_{\kappa} \cup \partial \mD_\kappa }, \\*[0.0mm]
&\llbracket \bt[ \bv^*] \rrbracket(\bxi)~=\, - \bn \exs\sip\exs \boldsymbol{\Upphi}(\bxi),  \,\,\, \llbracket \bv^*\rrbracket(\bxi)~=~\bzero, \,\,  & \bxi \in \partial \mD^\star_\kappa \nxs\setminus\nxs  \overline{\Gamma_{\kappa}},  \\*[0.2mm]
&\llbracket \bv^*\rrbracket(\bxi)~=~\boldsymbol{\varphi}, \,\, \llbracket \bt[ \bv^*] - \mathds{1}({\mD^\star_\kappa  \nxs\cap\nxs  {\Gamma_{\kappa}}}) \bn \exs\sip\exs \boldsymbol{\Upphi} \rrbracket(\bxi)~=\, - \mathds{1}({\partial\mD^\star_\kappa  \nxs\cap\nxs  {\Gamma_{\kappa}}}) \exs \bn \exs\sip\exs \boldsymbol{\Upphi}(\bxi) , \hspace{-12mm}  & \bxi \in \Gamma_\kappa,  \\*[0.2mm] 
&\llbracket \bv^*\rrbracket(\bxi)~=~\boldsymbol{\phi},  \,\, \llbracket \bt[ \bv^*] \rrbracket(\bxi)~=~\bzero, \quad  & \bxi \in \partial \mD^o_\kappa, \\*[0.2mm]
&\bt[  \bv^*](\bxi)~=~\bzero,  \quad & \bxi \in \partial {\mathcal{B}}_t,
\end{aligned}   
\eeq 
wherein $\bt[ \bv^*] \colon\!\!\!= \bn \nxs\cdot\nxs \bC \xxs\colon\!\nxs \nabla  \bv^*$. This may be observed by~(i)~premultiplying the first of~\eqref{up_a} by $\bar{\bu}^{\textrm{f}}$, and~(ii)~postmultiplying the conjugated first of~\eqref{uf} by $\bv^*$. Integration by parts over ${\mathcal{B}}\nxs\setminus\nxs\overline{\Gamma_\kappa \cup \partial \mD_\kappa}$ followed by application of the contact condition over $\Gamma_\kappa \cup \partial \mD_\kappa$ and summation of the results yield  
\[ 
 \int_{\partial {\mathcal{B}_t}\!}  \bar\btau \xxs\exs\sip\exs\xxs {\bv}^* \, \textrm{d}S_{\bxi} ~=\, \int_{\mD^\star_\kappa  \nxs\setminus  \overline{\Gamma_{\kappa}}} \big( \bar{\bu}^{\textrm{f}} \xxs\sip\exs \boldsymbol{\upphi}   +  \nabla \bar{\bu}^{\textrm{f}}  \exs\colon\nxs \boldsymbol{\Upphi} \big) \, \textrm{d}V_{\bxi}   ~+\,  \int_{\Gamma_\kappa} \bar{\bt}[\bu^{\textrm{f}}]\exs\sip\exs\xxs \boldsymbol{\varphi}  \, \textrm{d}S_{\bxi}  ~+\,  \int_{\partial \mD^o_\kappa\!} \nxs \bar{\bt}[\bu^{\textrm{f}}] \exs\sip\exs\xxs \boldsymbol{\phi} \, \textrm{d}S_{\bxi},
\]
which substantiates~\eqref{Hstar} via~$ \langle (\boldsymbol{\upphi},\boldsymbol{\Upphi},\boldsymbol{\varphi},\boldsymbol{\phi}), \mathcal{S}_\kappa(\btau) \rangle_{\mD^\star_\kappa \cup \exs \Gamma_\kappa \cup \exs \partial \mD^o_\kappa\nxs} =\, \langle {\bv}^*, \btau \rangle_{\partial {\mathcal{B}_t}\nxs}$. Here, 
\beq\lb{DuP}
\begin{aligned}
\left< \xxs \cdot \xxs, \xxs  \cdot \xxs \right>_{{\mD^\star_\kappa \cup \exs \Gamma_\kappa \cup \exs \partial \mD^o_\kappa\nxs}} \,\colon \!\!=~ \big\langle & \tilde{S}(\mD^\star_\kappa \cup \Gamma_\kappa \nxs \cup \partial \mD^o_\kappa), S(\mD^\star_\kappa \cup \Gamma_\kappa \nxs \cup \partial \mD^o_\kappa)  \big\rangle, \\*[0.1mm] 
 \left< \xxs \cdot \xxs, \xxs  \cdot \xxs \right>_{\partial {\mathcal{B}_t}} \,\colon \!\!=~ \big\langle & \tilde{H}^{\frac{1}{2}}(\partial {\mathcal{B}_t})^3, {H}^{-\frac{1}{2}}(\partial {\mathcal{B}_t})^3 \big\rangle,
\end{aligned}
\eeq
extend $L^2$ inner products. In this setting, the middle operator $T_\kappa \colon S(\mD^\star_\kappa \cup \Gamma_\kappa \nxs \cup \partial \mD^o_\kappa) \rightarrow \, \tilde{S}(\mD^\star_\kappa \cup \Gamma_\kappa \nxs \cup \partial \mD^o_\kappa)$ is given by
\beq\lb{T}
\begin{aligned}
T_\kappa\big(\bu^{\textrm{f}}\big{|}_{\mD^\star_\kappa  \setminus  \overline{\Gamma_{\kappa}}},& \nabla \bu^{\textrm{f}}\big{|}_{\mD^\star_\kappa  \setminus  \overline{\Gamma_{\kappa}}}, \exs\bt[\bu^{\textrm{f}}]\big{|}_{\Gamma_\kappa}, \exs\bt[\bu^{\textrm{f}}]\big{|}_{\partial \mD^o_\kappa}\big) ~:=~\\*[0.2mm] 
&\big((\rho_\kappa -\rho) \xxs \omega^2 \bw^\kappa\big{|}_{\mD^\star_\kappa  \setminus  \overline{\Gamma_{\kappa}}}, - (\bC_{\kappa\nxs}-\bC) \xxs\colon\! \nabla \bw^\kappa\big{|}_{\mD^\star_\kappa  \setminus  \overline{\Gamma_{\kappa}}}, \exs \llbracket \text{\bf u} \rrbracket\big{|}_{\Gamma_\kappa}, \exs\bu^\kappa\big{|}_{\partial \mD^o_\kappa} \big).
\end{aligned}   
\eeq

\section{Properties of operators}\label{SFS}

\begin{assum}\lb{Frac}
Given $\kappa \in \lbrace \circ, 1, \ldots \rbrace$ and $j \in \lbrace \circ, 1, \ldots, N^o_\kappa \rbrace$, suppose that (a) $\mD^o_{\kappa}$ may be decomposed into simply-connected components $\mD^o_{\kappa,j} \subset \mD^o_{\kappa}$, and (b) $\Gamma_\kappa$ consists of $M_\kappa\geqslant1$ (possibly disjoint) analytic surfaces~${\sf{S}}_m\subset\Gamma_\kappa$, $m=1,\ldots M_\kappa$, with the unique continuation $\partial {\sf{D}}_m$ identifying the ``interior'' domain~${\sf{D}}_m\!\subset\mathcal{B}$. Then for any $j$ (\emph{resp}.~$m$), it is assumed that $\omega\nxs>\nxs0$ is not a ``Neumann'' eigenfrequency of the Navier equation in $\mD^o_{\kappa,j}$ (\emph{resp}.~${\sf{D}}_m$), i.e.,~as long as ${\boldsymbol{\sf{u}}}_j \in H^1(\mD^o_{\kappa,j})^3$ and ${\boldsymbol{\sf{u}}}_m \in H^1({\sf{D}}_m)^3$ satisfying 
\beq\label{uiH}
\!\left\{\begin{array}{l}
\begin{aligned}
&\!\! \nabla \xxs \sip \xxs (\bC \exs\colon \! \nabla {\boldsymbol{\sf{u}}}_j) \,+\, \rho \exs \omega^2 {\boldsymbol{\sf{u}}}_j ~=~ \bzero  \!\! & \textrm{in}~ \mD^o_{\kappa,j}   \!\!\! \\*[0.1mm]
& \!\!\bn \xxs\sip\xxs \bC \exs\colon \!  \nabla  {\boldsymbol{\sf{u}}}_j ~=~ \bzero  \!\!   & \textrm{on}~ \partial \mD^o_{\kappa,j} \!\!\! 
\end{aligned}
\end{array}\right. \,\,\, \wedge \,\,\,
\!\left\{\begin{array}{l}
\begin{aligned}
&\!\! \nabla \xxs \sip \xxs (\bC_\kappa \exs\colon \! \nabla {\boldsymbol{\sf{u}}}_m) \,+\, \rho_\kappa \exs \omega^2 {\boldsymbol{\sf{u}}}_m ~=~ \bzero  \!\! & \textrm{in}~ {\sf{D}}_m   \!\!\! \\*[0.1mm]
& \!\!\bn \xxs\sip\xxs \bC_\kappa \exs\colon \!  \nabla  {\boldsymbol{\sf{u}}}_m ~=~ \bzero  \!\!   & \textrm{on}~ \partial {\sf{D}}_m \!\!\! 
\end{aligned}
\end{array}\right.,
\eeq
vanish identically in $\mD^o_{\kappa,j}$ and ${\sf{D}}_m$ respectively. If~${\sf{D}}_m$ is bounded, the real eigenfrequencies of~\eqref{uiH} form a discrete set~\cite{bell2013,bell2010}.
\end{assum}

\begin{lemma}\lb{InjH}
In light of the unique continuation principle, the operator $\mathcal{S}_\kappa \colon L^2({\partial {\mathcal{B}_t}})^3 \to S(\mD^\star_\kappa \cup \Gamma_\kappa \nxs \cup \partial \mD^o_\kappa)$ is injective at all $t_\kappa$.
\end{lemma}

\begin{lemma}\lb{C(R(T))}
Let 
\beq\label{bar(R(T))}\nonumber
H_{\Delta}~=~ \big{\lbrace} \hat{\boldsymbol{\sf u}}  \in H^1(\mD^\star_\kappa)^3 \exs \big{|} \, \nabla \exs\sip\exs \bC \exs \colon \! \nabla \hat{\boldsymbol{\sf u}} \,+ \rho \exs \omega^2 \hat{\boldsymbol{\sf u}} ~=~ \bzero \,\,\text{in} \,\, \mD^\star_\kappa \big{\rbrace}, 
\eeq
and define the map $({\hat{S}}_1, {\hat{S}}_2, {\hat{S}}_3) \colon H_{\Delta} \exs\to\, L^2({\mD^\star_\kappa  \nxs\setminus\nxs  \overline{\Gamma_{\kappa}}})^3 \times\xxs L^2({\mD^\star_\kappa  \nxs\setminus\nxs  \overline{\Gamma_{\kappa}}})^{3\times3} \nxs\times L^2(\Gamma_\kappa \cap \overline{\mD^\star_\kappa})^3$ such that 
\[
({\hat{S}}_1, {\hat{S}}_2, {\hat{S}}_3)(\hat{\boldsymbol{\sf u}})~=~ \big(\hat{\boldsymbol{\sf u}}\big{|}_{\mD^\star_\kappa  \setminus  \overline{\Gamma_{\kappa}}},\, \nabla \hat{\boldsymbol{\sf u}}\big{|}_{\mD^\star_\kappa  \setminus  \overline{\Gamma_{\kappa}}}, \,\bt[ \hat{\boldsymbol{\sf u}}] \big{|}_{\Gamma_\kappa \cap\exs \overline{\mD^\star_\kappa}} \big), \quad \bt[ \hat{\boldsymbol{\sf u}}]~=~\bn \exs\sip\exs \bC \exs\colon\nxs\nxs \nabla \hat{\boldsymbol{\sf u}},
\] 
then ${\mathfrak{S}}(H_{\Delta}) \colon\!\!\! = {\hat{S}}_1(H_{\Delta}) \nxs\times\nxs {\hat{S}}_2(H_{\Delta}) \nxs\times\nxs ({\hat{S}}_3(H_{\Delta}) \oplus H^{-1/2}(\Gamma_\kappa \nxs\!\setminus\! \overline{\mD^\star_\kappa})^3) \nxs\times\nxs {H^{-1/2}(\partial \mD^o_\kappa)^3} = \exs \overline{\mathcal{R}(\mathcal{S}_\kappa)}$  \\where the latter denotes the closure of the range of $\mathcal{S}_\kappa$.
\end{lemma}

\begin{proof}
Observe that the restriction of $\bu^{\textrm{f}} \in H^1(\mathcal{B})^3$, satisfying~\eqref{uf}, to $\mD^\star_\kappa$ belongs to $H_{\Delta}$, hence $\mathcal{R}(\mathcal{S}_\kappa) \subset {\mathfrak{S}}(H_{\Delta})$. To prove the claim, it is then sufficient to establish that $\mathcal{S}^*_\kappa:{\mathfrak{S}}(H_{\Delta}) \to L^2(\partial \mathcal{B}_t)^3$ given by 
\beq\lb{BIE*}
\begin{aligned}
& \dfrac{1}{2} \xxs \mathcal{S}_\kappa^*(\boldsymbol{\upphi},\boldsymbol{\Upphi},\boldsymbol{\varphi} \oplus \boldsymbol{\psi},\boldsymbol{\phi}) \,=\, \int_{\mD^\star_\kappa  \setminus  \overline{\Gamma_{\kappa}}} \big[\xxs \boldsymbol{\upphi}(\by) \cdot \bfG(\bxi,\by) \xxs + \exs  \boldsymbol{\Upphi}(\by)\exs\colon \! \nabla \exs \bfG(\bxi,\by) \xxs \big]  \textrm{d}V_{\by} ~+\,  \\*[0.2mm]   
&   \int_{\Gamma_\kappa \cap \exs\overline{\mD^\star_\kappa}} \exs \boldsymbol{\varphi}(\by) \cdot \boldsymbol{T}(\bxi,\by) \, \textrm{d}S_{\by}  \,+\, \int_{\Gamma_\kappa \nxs\!\setminus \overline{\mD^\star_\kappa}} \boldsymbol{\psi}(\by) \cdot \boldsymbol{T}(\bxi,\by) \, \textrm{d}S_{\by} \,+\, \int_{\partial \mD^o_\kappa} \boldsymbol{\phi}(\by) \cdot \boldsymbol{T}(\bxi,\by) \, \textrm{d}S_{\by},
\end{aligned}
\eeq
is injective on ${\mathfrak{S}}(H_{\Delta})$. Suppose that there exists $(\boldsymbol{\upphi},\boldsymbol{\Upphi},\boldsymbol{\varphi}\oplus\boldsymbol{\psi},\boldsymbol{\phi}) = ({\boldsymbol{\sf u}}^{i}, \nabla {\boldsymbol{\sf u}}^{i}, \bt[{\boldsymbol{\sf u}}^i] \oplus \boldsymbol{\psi}, \boldsymbol{\phi})$ -- with ${\boldsymbol{\sf u}}^i  \in H^1(\mD^\star_\kappa)^3$ satisfying $\nabla \exs\sip\exs \bC \exs \colon \! \nabla {\boldsymbol{\sf u}}^i + \rho \exs \omega^2 {\boldsymbol{\sf u}}^i = \bzero$ in $\mD^\star_\kappa$ while $(\boldsymbol{\psi}, \boldsymbol{\phi}) \in \tilde{H}^{1/2}(\Gamma_\kappa \nxs\!\setminus\! \overline{\mD^\star_\kappa})^3 \nxs\times\nxs {\tilde{H}^{1/2}(\partial \mD^o_\kappa)^3}$ -- such that $\mathcal{S}_\kappa^*({\boldsymbol{\sf u}}^{i}, \nabla {\boldsymbol{\sf u}}^{i}, \bt[{\boldsymbol{\sf u}}^i] \oplus \boldsymbol{\psi}, \boldsymbol{\phi}) = \bzero$. Since by construction $\bv^\kappa(\bxi) = \mathcal{S}_\kappa^*(\cdot)$ on $\bxi \in \partial \mathcal{B}_t$, it is evident from~\eqref{vk} that $\bv^\kappa$ has trivial Dirichlet and Neumann traces on $\partial \mathcal{B}_t$, and thus, the unique continuation principle reads $\bv^\kappa = \bzero$ in ${\mathcal{B}} \exs \backslash \overline{\Gamma_{\kappa} \cup \mD_\kappa}$. From (a)~properties of the layer potentials i.e.,~$\boldsymbol{\phi} = \bu^\kappa$ and $\boldsymbol{\psi} = \llbracket \bv^\kappa \rrbracket$, (b)~fifth of~\eqref{vk} which reads $\bt^{\textrm{f}} = \bzero$ on $\partial \mD^o_\kappa$, and (c)~Assumption~\ref{Frac} indicating that $\omega>0$ is not a Neumann eigenfrequency of the Navier equation affiliated with any simply-connected subset of $\mD^o_\kappa$, one may conclude that $\boldsymbol{\psi} = \boldsymbol{\phi} = \bzero$. Now, on denoting $\Gamma_\kappa^\iota = \Gamma_\kappa \cap \overline{\mD^\star_\kappa}$ and $\mathcal{B}^\star = \mathcal{B} \setminus \overline{\Gamma_\kappa^\iota \cup \partial \mD^\star_\kappa}$, let $\forall \bxi \in \mathcal{B} \setminus\! \overline{\Gamma_\kappa^\iota}$,
\beq\lb{tv}
\textrm{\bf v}(\bxi) \,=\, \int_{\mD^\star_\kappa  \setminus  \overline{\Gamma_{\kappa}}} \big[\xxs \boldsymbol{\sf u}^{i}(\by) \cdot \bfG(\bxi,\by) \xxs + \exs  \nabla {\boldsymbol{\sf u}}^{i}(\by)\exs\colon \! \nabla \exs \bfG(\bxi,\by) \xxs \big]  \textrm{d}V_{\by} \,+\,  \int_{\Gamma_\kappa^\iota} \bt[{\boldsymbol{\sf u}}^i](\by) \cdot \boldsymbol{T}(\bxi,\by) \, \textrm{d}S_{\by}. 
\eeq

From the regularity of volume potentials, one may infer that $\forall \xxs \textrm{\bf v}' \in H^1(\mathcal{B}^\star)$, $\textrm{\bf v} \in H^1(\mathcal{B}^\star)$ satisfies  
\beq\lb{Wik-v}
\begin{aligned}
  \int_{\mathcal{B}^\star} \big[ \xxs \rho \exs  \omega^2 \exs \textrm{\bf v}' \xxs\sip\exs \textrm{\bf v}  \,-\,   \nabla \exs \textrm{\bf v}' \colon\nxs \bC \exs\colon\nxs \nabla \textrm{\bf v}  \big] \, \textrm{d}V_{\bxi}  ~-\, \int_{\Gamma_\kappa^\iota} \llbracket \textrm{\bf v}' \rrbracket \xxs\sip\xxs (&\bn  \exs\sip\exs [\bC \exs\colon\! \nabla \textrm{\bf v} \exs- \nabla {\boldsymbol{\sf u}}^i] ) \, \textrm{d}S_{\bxi} \,~=\, \\*[0.2mm]  
&   ~-\, \int_{\mD^\star_\kappa \setminus\! \overline{\Gamma_\kappa}} \big[ \xxs \textrm{\bf v}' \xxs\sip\exs {\boldsymbol{\sf u}}^i \,+\,   \nabla \exs \textrm{\bf v}' \colon\! \nabla{\boldsymbol{\sf u}}^i  \big] \, \textrm{d}V_{\bxi}.   
\end{aligned}
\eeq

Note that $\textrm{\bf v} = \mathcal{S}_\kappa^*({\boldsymbol{\sf u}}^{i}, \nabla {\boldsymbol{\sf u}}^{i}, \bt[{\boldsymbol{\sf u}}^i] \oplus \bzero, \bzero)$ on $\partial \mathcal{B}_t$, then $\mathcal{S}_\kappa^*(\cdot) = \bzero$ implies that $\textrm{\bf v} = \bzero$ in $\mathcal{B}^\star \!\setminus \! \overline{\mD^\star_\kappa}$. Then by setting $\textrm{\bf v}' = \bar{\boldsymbol{\sf u}}^i$,~\eqref{Wik-v} may be recast as 
\beq\lb{Wik-v2}
  \int_{\mD^\star_\kappa \setminus \overline{\Gamma_\kappa}} \big[ \xxs \rho \exs  \omega^2 \exs \bar{\boldsymbol{\sf u}}^i \xxs\sip\exs \textrm{\bf v}  \,-\,   \nabla \exs \bar{\boldsymbol{\sf u}}^i \colon\nxs \bC \exs\colon\nxs \nabla \textrm{\bf v}  \big] \, \textrm{d}V_{\bxi} ~=\, -\, \norms{\xxs{\boldsymbol{\sf u}}^i}^2_{H^1({\mD^\star_\kappa \setminus \overline{\Gamma_\kappa}})^3}.   
\eeq

Also, on recalling $\nabla \exs\sip\exs \bC \exs \colon \! \nabla {\boldsymbol{\sf u}}^i + \rho \exs \omega^2 {\boldsymbol{\sf u}}^i = \bzero$ in $\mD^\star_\kappa$, it follows that       
\beq\lb{Wik-v3}
  \int_{\mD^\star_\kappa \setminus \overline{\Gamma_\kappa}} \big[ \xxs \rho \exs  \omega^2 \exs \bar{\boldsymbol{\sf u}}^i \xxs\sip\exs \textrm{\bf v}  \,-\,   \nabla \exs \bar{\boldsymbol{\sf u}}^i \colon\nxs \bC \exs\colon\nxs \nabla \textrm{\bf v}  \big] \, \textrm{d}V_{\bxi} ~=\,  \norms{\xxs\bt[{\boldsymbol{\sf u}}^i]}^2_{L^2(\Gamma_\kappa^\iota)^3}.   
\eeq

Combining \eqref{Wik-v2} and~\eqref{Wik-v3} reads $(\boldsymbol{\upphi},\boldsymbol{\Upphi},\boldsymbol{\varphi}) =\bzero$ which completes the proof. 
\end{proof}

\begin{lemma}\lb{H*p}
Under Assumption~\ref{CP}, the operator $\mathcal{S}_\kappa^* \colon  \tilde{S}(\mD^\star_\kappa \cup \Gamma_\kappa \nxs \cup \partial \mD^o_\kappa) \,\rightarrow L^2({\partial {\mathcal{B}_t}})^3$ is compact and has a dense range. 
\end{lemma}
\begin{proof} 
The compactness of $\mathcal{S}_\kappa^*$ is established by the smooth kernels in its integral form~\eqref{BIE*}, and its dense range results from the injectivity of $\mathcal{S}_\kappa$ per Lemma~\ref{InjH}.
\end{proof}

\begin{assum}\label{ITPA}
Under Assumptions~\ref{tfc}-\ref{ITP_A1}, given $(\boldsymbol{\sf f}^\kappa, \boldsymbol{\sf g}^\kappa) \in H^{-1/2}(\partial\mD^\star_\kappa)^3 \nxs\times\nxs H^{1/2}(\partial\mD^\star_\kappa)^3$, consider the solution $(\boldsymbol{\sf u}^\kappa, \boldsymbol{\sf w}^\kappa) \in H^1(\mD^\star_\kappa)^3 \nxs\times\nxs H^1(\mD^\star_\kappa \! \setminus \! \overline{\Gamma_{\kappa}})^3$ to the interior transmission problem (ITP)   
\beq\nonumber
\hspace{-67mm}\text{ITP}(\mD^\star_\kappa, \Gamma_{\kappa}; \lbrace\bC,\rho\rbrace, \lbrace\bC_\kappa,\rho_\kappa\rbrace, \bK_{\kappa}; \boldsymbol{\sf f}^\kappa\nxs, \boldsymbol{\sf g}^\kappa) \colon
\vspace{-1.5mm}
\eeq
\beq\lb{ITPE}
\!\left\{\begin{array}{l} \!\!\!
\begin{aligned}
&\nabla \exs\sip\exs [\exs \bC_\kappa(\bxi) \exs \colon \! \nabla \boldsymbol{\sf w}^\kappa\exs](\bxi) \,+\, \rho_\kappa(\bxi) \exs \omega^2 \boldsymbol{\sf w}^\kappa(\bxi) ~=~ \bzero, \quad &  \bxi \exs \in \mD^\star_\kappa  \nxs\setminus\nxs \overline{\Gamma_{\kappa}},   \\*[0.0mm]
&\nabla \exs\sip\exs \bC \exs \colon \! \nabla \boldsymbol{\sf u}^\kappa(\bxi) \,+\, \rho \exs \omega^2 \boldsymbol{\sf u}^\kappa(\bxi) ~=~ \bzero, \quad &  \bxi \exs \in \mD^\star_\kappa, \\*[0.2mm]
&\bt[\boldsymbol{\xxs \sf w}^\kappa](\bxi) - \xxs\bt[\boldsymbol{\xxs \sf u}^\kappa](\bxi) ~=~ \boldsymbol{\sf f}^\kappa(\bxi),  \,\,\, [\boldsymbol{\sf w}^\kappa-\boldsymbol{\sf u}^\kappa](\bxi) ~=~ \boldsymbol{\sf g}^\kappa(\bxi), \quad  & \bxi \in \partial \mD^\star_\kappa  \setminus \overline{\Gamma_\kappa}, \\*[0.2mm]
&\bt[\boldsymbol{\xxs \sf w}^\kappa](\bxi) ~=~\bK_{\kappa}(\bxi) \llbracket \boldsymbol{\sf w}^\kappa \rrbracket(\bxi),  \,\,\, \llbracket \bt[\boldsymbol{\xxs \sf w}^\kappa] \rrbracket(\bxi)~=~\bzero,  \,\, & \bxi \in \Gamma_\kappa \cap  \mD^\star_\kappa,  \\*[0.2mm] 
&\bt[\boldsymbol{\xxs \sf w}^\kappa](\bxi) - \xxs\bt[\boldsymbol{\xxs \sf u}^\kappa](\bxi) ~=~ \boldsymbol{\sf f}^\kappa(\bxi),  \,\,\, \bt[\boldsymbol{\xxs \sf w}^\kappa](\bxi)~=~\bK_{\kappa}(\bxi) [\xxs\boldsymbol{\sf g}^\kappa\nxs+\boldsymbol{\sf u}^\kappa\nxs-\boldsymbol{\sf w}^\kappa](\bxi),  & \bxi \in \Gamma_\kappa \cap  \partial\mD^\star_\kappa,
\end{aligned} 
\end{array}\right.  
\eeq  
wherein
\beq\nonumber
\bt[\boldsymbol{\xxs \sf u}^\kappa](\bxi)~=~\bn(\bxi) \exs\sip\exs \bC \exs\xxs\colon\nxs \nabla\boldsymbol{\sf u}^\kappa(\bxi), \qquad 
\bt[\boldsymbol{\xxs \sf w}^\kappa](\bxi)~=~\bn(\bxi) \exs\sip\exs \bC_\kappa(\bxi) \exs\colon\nxs \nabla\boldsymbol{\sf w}^\kappa(\bxi), \qquad \bxi \in \Gamma_\kappa \cup  \partial\mD^\star_\kappa.
\eeq

It is assumed that $\forall t_\kappa$, $\omega$ is such that~\eqref{ITPE} remains wellposed i.e.,~for $(\boldsymbol{\sf f}^\kappa, \boldsymbol{\sf g}^\kappa) = \bzero$, the ITP does not admit a nontrivial solution $(\boldsymbol{\sf u}^\kappa, \boldsymbol{\sf w}^\kappa)$. For $\Gamma_\kappa = \emptyset$, the elastodynamics ITP, with varying restrictions on $\bC$ and $\bC_\kappa(\bxi)$, is analyzed by~\cite{char2002,bell2010,bell2013,cako2021} following the variational method introduced by~\cite{hahn2000}. In the most general case, under Assumptions~\ref{CP} and~\ref{ITP_A1},~\eqref{ITPE} with $\Gamma_\kappa = \emptyset$ is well-posed when $\omega$ does not belong to (at most) a countable set of transmission eigenvalues~\cite{bell2010}. The latter is also concluded in a recent study of acoustic ITP for penetrable obstacles with sound-hard cracks~\cite{audi2021}. Similar analysis could be applied to~\eqref{ITPE} which is beyond the scope of the present study.   
\end{assum}

\begin{lemma}\lb{I{T}>0}
Operator~$T_\kappa \colon {\mathfrak{S}}(H_{\Delta})  \rightarrow \, \tilde{S}(\mD^\star_\kappa \cup \Gamma_\kappa \nxs \cup \partial \mD^o_\kappa)$ in~(\ref{T}) is bounded and satisfies
 \beq\lb{pos-IT}
 \Im \langle{{T_\kappa \boldsymbol{\Xi}}, {\boldsymbol{\Xi}} \xxs \rangle}_{\mD^\star_\kappa \cup \exs \Gamma_\kappa \cup \exs \partial \mD^o_\kappa\nxs}  > \exs 0,
 \eeq
 $\forall \exs \boldsymbol{\Xi} \in {\mathfrak{S}}(H_{\Delta}) \! :~ \boldsymbol{\Xi} \neq \bzero$. Consequently, $T_\kappa$ is also injective provided that $\omega$ is not a transmission eigenvalue of~\eqref{ITPE} per Assumption~\ref{ITPA}. 
\end{lemma}

\begin{proof}
The well-posedness of~\eqref{vk} establish the boundedness of $T_\kappa$. Now, consider $\bv^\kappa$ satisfying~\eqref{Wik-GE} with $\big(\bu^{\textrm{f}}\big{|}_{\mD^\star_\kappa  \setminus  \overline{\Gamma_{\kappa}}}, \, \nabla \bu^{\textrm{f}}\big{|}_{\mD^\star_\kappa  \setminus  \overline{\Gamma_{\kappa}}}, \,\bt[\bu^{\textrm{f}}]\big{|}_{\Gamma_\kappa}, \, \bt[\bu^{\textrm{f}}]\big{|}_{\partial \mD^o_\kappa}\big)= \boldsymbol{\Xi}$. Taking $\bv' = \bv^\kappa$ in $\mathcal{B}^-_\kappa$ and $\bv' = \bw^\kappa - \bu^{\textrm{f}}$ in $\mD^\star_\kappa  \setminus  \overline{\Gamma_{\kappa}}$, observe from~\eqref{Wik-GE} that  
\beq\lb{T-bound}
 \Im \langle{{T_\kappa\boldsymbol{\Xi}}, {\boldsymbol{\Xi}} \rangle}_{\mD^\star_\kappa \cup \exs \Gamma_\kappa \cup \exs \partial \mD^o_\kappa\nxs} ~=\, -\Im(\int_{\mD^\star_\kappa  \setminus  \overline{\Gamma_{\kappa}}\exs}   \nabla \exs \bar{\bw}^\kappa \colon\nxs \bC_\kappa \exs\colon\! \nabla \bw^\kappa \, \textrm{d}V_{\bxi} ~+\, \int_{\Gamma_\kappa} \llbracket \bar{\bv}^\kappa \rrbracket \xxs\sip\xxs \bK_{\nxs\kappa} \llbracket \bv^\kappa \rrbracket  \, \textrm{d}S_{\bxi}),
\eeq
whereby~\eqref{pos-IT} follows immediately from Assumption~\ref{CP}. Now, let $T_\kappa{\boldsymbol{\Xi}} = \bzero$, then~\eqref{T-bound}, second of~\eqref{uk}, the unique continuation principle, and~\eqref{IE} imply that ${\mathcal{S}}_\kappa^* T_\kappa{\boldsymbol{\Xi}} = \bzero$. Then, Lemma~\ref{comp_G} reads ${\boldsymbol{\Xi}} = \bzero$ which proves the injectivity of $T_\kappa$. 
\end{proof}

\begin{lemma}\lb{comp_G}
Under Assumption~\ref{ITPA}, the operator $\mathcal{V}_\kappa = {\mathcal{S}}_\kappa^* T_\kappa \colon\nxs {\mathfrak{S}}(H_{\Delta}) \rightarrow L^2({\partial {\mathcal{B}_t}})^3$ is compact and injective with dense range. 
\end{lemma}
\begin{proof}
Compactness of $\mathcal{V}_\kappa$ follows immediately from Lemmas~\ref{H*p} and~\ref{I{T}>0} which respectively establish the compactness of $\mathcal{S}_\kappa^*$ and the boundedness of $T_\kappa$. To demonstrate the injectivity of $\mathcal{V}_\kappa$, let 
\[
\big(\hat{\boldsymbol{\sf u}}\big{|}_{\mD^\star_\kappa  \setminus  \overline{\Gamma_{\kappa}}},\, \nabla \hat{\boldsymbol{\sf u}}\big{|}_{\mD^\star_\kappa  \setminus  \overline{\Gamma_{\kappa}}}, \,\bt[ \hat{\boldsymbol{\sf u}}] \big{|}_{\Gamma_\kappa \cap\exs \overline{\mD^\star_\kappa}} \oplus\boldsymbol{\psi}, \,\boldsymbol{\phi} \big) \in \mathfrak{S}(H_{\Delta}), \,\, (\boldsymbol{\psi}, \boldsymbol{\phi})\in H^{-1/2}(\Gamma_\kappa \nxs\!\setminus\! \overline{\mD^\star_\kappa})^3 \times {H^{-1/2}(\partial \mD^o_\kappa)^3},
\] 
such that $\mathcal{V}_\kappa(\hat{\boldsymbol{\sf u}}{|}_{\mD^\star_\kappa  \setminus  \overline{\Gamma_{\kappa}}},\, \nabla \hat{\boldsymbol{\sf u}}{|}_{\mD^\star_\kappa  \setminus  \overline{\Gamma_{\kappa}}}, \,\bt[ \hat{\boldsymbol{\sf u}}] {|}_{\Gamma_\kappa \cap\exs \overline{\mD^\star_\kappa}} \oplus\boldsymbol{\psi}, \,\boldsymbol{\phi}) = \bzero$. Since by definition $\bv^\kappa = \mathcal{V}_\kappa(\cdot)$ on $\partial \mathcal{B}_t$, one may observe that $\bv^\kappa$ satisfying~\eqref{vk} has trivial Dirichlet and Neumann traces on $\partial \mathcal{B}_t$ so that by the unique continuation principle $\bv^\kappa = \bzero$ in ${\mathcal{B}}^-_\kappa$. Now, from (a)~second and fifth of~\eqref{vk}, and (b)~Assumption~\ref{Frac} indicating that $\omega>0$ is not a Neumann eigenfrequency of the Navier equation affiliated with any simply-connected subset of $\mD^o_\kappa$, one may conclude that $\boldsymbol{\psi} = \boldsymbol{\phi} = \bzero$. Next, let us define ${\boldsymbol{\sf w}}^\kappa$ such that ${\boldsymbol{\sf w}}^\kappa = \hat{\boldsymbol{\sf u}} + \bv^\kappa$ in ${\mD^\star_\kappa}$, then the pair $(\hat{\boldsymbol{\sf u}}, {\boldsymbol{\sf w}}^\kappa)$ satisfies the interior transmission problem~\eqref{ITPE} with trivial boundary potentials i.e.,~$(\boldsymbol{\sf f}^\kappa, \boldsymbol{\sf g}^\kappa) = \bzero$. Since $\omega$ is not a transmission eigenvalue as per Assumption~\ref{ITPA}, on may deduce that $\hat{\boldsymbol{\sf u}} = \bzero$ in ${\mD^\star_\kappa  \setminus  \overline{\Gamma_{\kappa}}}$. Subsequently,~$\nabla \hat{\boldsymbol{\sf u}}=\bzero$ in ${\mD^\star_\kappa  \setminus  \overline{\Gamma_{\kappa}}}$, and $\bt[ \hat{\boldsymbol{\sf u}}] =\bzero$ on ${\Gamma_\kappa \cap\exs \overline{\mD^\star_\kappa}}$ which proves the injectivity of $\mathcal{V}_\kappa$.  

Denseness of the range of $\mathcal{V}_\kappa$ may be established by showing that the adjoint operator $\mathcal{V}^*_\kappa$ is injective. In this vein, from definition, $\mathcal{V}^*_\kappa \colon\nxs L^2({\partial {\mathcal{B}_t}})^3 \rightarrow \exs\tilde{S}(\mD^\star_\kappa \cup \Gamma_\kappa \nxs \cup \partial \mD^o_\kappa)$ is given by 
\beq\lb{uta2}
\mathcal{V}^*_\kappa(\boldsymbol{\uptau})~=~\big((\rho_\kappa -\rho) \xxs \omega^2 \bu^{\!\uptau}\big{|}_{\mD^\star_\kappa  \setminus  \overline{\Gamma_{\kappa}}}, - (\overline{\bC_{\kappa\nxs}}-\bC) \xxs\colon\! \nabla \bu^{\!\uptau}\big{|}_{\mD^\star_\kappa  \setminus  \overline{\Gamma_{\kappa}}}, \exs \llbracket \bu^{\!\uptau}\xxs \rrbracket\big{|}_{\Gamma_\kappa}, \exs\bu^{\!\uptau}\big{|}_{\partial \mD^o_\kappa} \big),
\eeq 
where $\bu^{\!\uptau} \in H^1({\mathcal{B}}_\kappa^{-}\nxs \cup \mD^\star_\kappa \!\setminus\!  \overline{\Gamma_{\kappa}})^3$ solves 
\beq\lb{uta}
\begin{aligned}
&\nabla \exs\sip\exs \bC \exs \colon \! \nabla \bu^{\!\uptau}(\bxi) \,+\, \rho \exs \omega^2 \bu^{\!\uptau}(\bxi) \,+\, \mathds{1}({\mD^\star_\kappa  \!\setminus\!  \overline{\Gamma_{\kappa}}}) [\xxs \text{\bf{f}\xxs}^{\uptau} \nxs+ \nabla \xxs\sip\xxs \boldsymbol{\upsigma\xxs}^{\!\uptau}](\bxi)~=~ \bzero, \quad &  \bxi \exs \in {\mathcal{B}}_\kappa^{-}\nxs \cup \mD^\star_\kappa \!\setminus\!  \overline{\Gamma_{\kappa}},   \\*[0.2mm]
&\bt[\bu^{\!\uptau}\xxs](\bxi) + \mathds{1}({\overline{\mD^\star_\kappa}  \nxs\cap\nxs  {\Gamma_{\kappa}}}) \exs \bn \exs\sip\exs \boldsymbol{\upsigma}^{\!\uptau}(\bxi)~=~\overline{\bK_{\kappa}}(\bxi) \llbracket \bu^{\!\uptau}\xxs \rrbracket(\bxi),  \hspace{-10mm}  & \bxi \in \Gamma_\kappa,  \\*[0.2mm] 
& \llbracket \bt[\bu^{\!\uptau}\xxs] + \mathds{1}({\mD^\star_\kappa  \nxs\cap\nxs  {\Gamma_{\kappa}}}) \xxs \bn \exs\sip\exs \boldsymbol{\upsigma}^{\nxs\uptau}(\bxi) \rrbracket(\bxi)~=~ \mathds{1}({\partial\mD^\star_\kappa  \nxs\cap\nxs  {\Gamma_{\kappa}}}) \exs \bn \exs\sip\exs \boldsymbol{\upsigma}^{\nxs\uptau}(\bxi),  \,\, & \bxi \in \Gamma_\kappa,  \\*[0.2mm] 
&\llbracket \bt[\bu^{\!\uptau}\xxs] \rrbracket(\bxi)~=~\bn \exs\sip\exs \boldsymbol{\upsigma}^{\nxs\uptau}(\bxi),  \,\,\, \llbracket \bu^{\!\uptau}\xxs \rrbracket(\bxi) ~=~\bzero, \qquad  & \bxi \in \partial \mD^\star_\kappa  \setminus \overline{\Gamma_\kappa}, \\*[0.2mm]
&\bt[\bu^{\!\uptau}\xxs](\bxi)~=\,\bzero,  \quad  & \bxi \in \partial \mD^o_\kappa, \\*[0.2mm]
&\bt[\bu^{\!\uptau}\xxs](\bxi)~=~\boldsymbol{\uptau},  \quad & \bxi \in \partial {\mathcal{B}}_t.
\end{aligned}   
\eeq 
where $\bt[\bu^{\!\uptau}\xxs] \exs\colon\!\!\!= \bn \exs\sip\exs \bC \exs \colon \!  \nabla  \bu^{\!\uptau}$, and 
$$ \text{\bf{f}\xxs}^{\uptau}(\bxi) \colon \!\!\!= (\rho_\kappa(\bxi) - \rho) \exs \omega^2 \bu^{\!\uptau}(\bxi), \qquad \boldsymbol{\upsigma\xxs}^{\!\uptau}(\bxi) \colon \!\!\!= (\overline{\bC_\kappa}(\bxi) - \bC) \exs \colon \! \nabla \bu^{\!\uptau}(\bxi), \qquad \bxi \in {\mD^\star_\kappa  \nxs\setminus\nxs  \overline{\Gamma_{\kappa}}}.$$
This may be observed by~(i)~premultiplying the first of~\eqref{vk} by $\overline{\bu^{\!\uptau}}$, and~(ii)~postmultiplying the conjugated first of~\eqref{uta} by $\bv^\kappa$. Integration by parts over ${\mathcal{B}}_\kappa^{-}\nxs \cup \mD^\star_\kappa \!\setminus\!  \overline{\Gamma_{\kappa}}$ followed by application of the contact condition over $\Gamma_\kappa \cup \partial \mD_\kappa$ and summation of the results yield  
\[ 
\begin{aligned}
 \overline{\big{\langle}\xxs {\bv}^\kappa, \boldsymbol{\uptau} \xxs\big{\rangle}}_{\partial {\mathcal{B}_t}\nxs} & ~=~\big{\langle} \big((\rho_\kappa -\rho) \xxs \omega^2 \bu^{\!\uptau}\big{|}_{\mD^\star_\kappa  \setminus  \overline{\Gamma_{\kappa}}}, - (\overline{\bC_{\kappa\nxs}}-\bC) \xxs\colon\! \nabla \bu^{\!\uptau}\big{|}_{\mD^\star_\kappa  \setminus  \overline{\Gamma_{\kappa}}}, \exs \llbracket \bu^{\!\uptau}\xxs \rrbracket\big{|}_{\Gamma_\kappa}, \exs\bu^{\!\uptau}\big{|}_{\partial \mD^o_\kappa} \big), \\*[0.2mm] 
  &\qquad\, \big(\bu^{\textrm{f}}\big{|}_{\mD^\star_\kappa  \setminus  \overline{\Gamma_{\kappa}}}, \nabla \bu^{\textrm{f}}\big{|}_{\mD^\star_\kappa  \setminus  \overline{\Gamma_{\kappa}}}, \exs\bt[\bu^{\textrm{f}}]\big{|}_{\Gamma_\kappa}, \exs\bt[\bu^{\textrm{f}}]\big{|}_{\partial \mD^o_\kappa}\big)  \big{\rangle}_{\mD^\star_\kappa \cup \exs \Gamma_\kappa \cup \exs \partial \mD^o_\kappa\nxs}, 
 \end{aligned}
\]
substantiating~\eqref{uta} as the adjoint operator. Next, let $\boldsymbol{\uptau} \in L^2({\partial {\mathcal{B}_t}})^3$ and assume that $\mathcal{V}^*_\kappa(\boldsymbol{\uptau}) = \bzero$, i.e.,
\[
\big(\bu^{\!\uptau}\big{|}_{\mD^\star_\kappa  \setminus  \overline{\Gamma_{\kappa}}}, \nabla \bu^{\!\uptau}\big{|}_{\mD^\star_\kappa  \setminus  \overline{\Gamma_{\kappa}}}, \exs \llbracket \bu^{\!\uptau}\xxs \rrbracket\big{|}_{\Gamma_\kappa}, \exs\bu^{\!\uptau}\big{|}_{\partial \mD^o_\kappa} \big) ~=~ \bzero.
\]
In this setting, the unique continuation reads $\bu^{\!\uptau} = \bzero$, and thus, $\boldsymbol{\uptau} = \bzero$ which concludes the proof.
 \end{proof}
 
\begin{assum}\label{Tass}
Under Assumptions~\ref{CP} and~\ref{ITP_A1}, one of the following applies:
\begin{description}
\item{$\bullet$\,\,\,}~$\mathfrak{R}(\bC_{\kappa\nxs}-\bC) - \alpha \xxs \mathfrak{I}(\bC_{\kappa\nxs})$ is positive definite on ${\mD^\star_\kappa  \nxs\setminus\nxs  \overline{\Gamma_{\kappa}}}$ for some constant $\alpha \geqslant 0$. \vspace{0.5mm}
\item{$\bullet$\,\,\,}~For some constants $\alpha,\eta > 0$,
\[
\overline{\text{\bf{X}}} \xxs \colon \! \mathfrak{R}(\bC-\bC_{\kappa\nxs})  \colon \! \text{\bf{X}} \exs\,\geqslant\, \alpha |\text{\bf{X}}|^2, \quad   \overline{\text{\bf{X}}}\xxs \colon \! \mathfrak{R}(\bC_{\kappa\nxs}) \colon \!  \text{\bf{X}} \,\geqslant\, \eta |\text{\bf{X}}|^2, \quad  \norms{\mathfrak{I}(\bC_{\kappa\nxs})}_{L^{\infty}} \,<\, \sqrt{\alpha\eta},
\]
on ${\mD^\star_\kappa  \!\setminus\!  \overline{\Gamma_{\kappa}}}$ for all $\text{\bf{X}}$ in $\mathbb{C}^{3\times\nxs 3}$.
\end{description}
\end{assum}

\begin{lemma}\lb{T-invs0}
Under Assumptions~\ref{tfc},~\ref{CP},~\ref{ITP_A1},~\ref{ITPA}, and~\ref{Tass}, the operator $T_\kappa \colon {\mathfrak{S}}(H_{\Delta}) \rightarrow \, \tilde{S}(\mD^\star_\kappa \cup \Gamma_\kappa \nxs \cup \partial \mD^o_\kappa)$ is coercive, i.e., there exists a constant $c\!>\!0$ independent of~$\boldsymbol{\Xi}$ such that
\beq\lb{co-T0}
|\langle {T_\kappa{\boldsymbol{\Xi}},{\boldsymbol{\Xi}}} \rangle| \,\,\geqslant\,\, \textrm{c} \nxs \norms{\boldsymbol{\Xi}}_{S(\mD^\star_\kappa \cup \Gamma_\kappa \nxs \cup \partial \mD^o_\kappa)}^2, \qquad \forall \boldsymbol{\Xi} \in {\mathfrak{S}}(H_{\Delta}).
\eeq 
\end{lemma} 

\begin{proof}
We adopt a contradiction argument as follows. Suppose~\eqref{co-T0} does not hold, then one may find a sequence $({\boldsymbol{\Xi}_n})_{n \in \mathbb{N}} \subset {\mathfrak{S}}(H_{\Delta})$ such that
\beq\lb{SeQ}
\norms{{\boldsymbol{\Xi}_n}}_{{S}(\mD^\star_\kappa \cup \Gamma_\kappa \nxs \cup \partial \mD^o_\kappa)}\,\exs=~1, \quad |\langle {T_\kappa{\boldsymbol{\Xi}_n},{\boldsymbol{\Xi}_n}} \rangle| \,\to\, 0 \quad \text{as} \quad n \,\to\, \infty.
\eeq
Denote by $\bv^n \in H^1({\mathcal{B}}_\kappa^{-}\nxs \cup \mD^\star_\kappa \!\setminus\!  \overline{\Gamma_{\kappa}})^3$ the solution to~\eqref{vk} with 
\[
\begin{aligned}
&(\bu^{\textrm{f}} {|}_{\mD^\star_\kappa  \setminus  \overline{\Gamma_{\kappa}}},\nxs \nabla \bu^{\textrm{f}} {|}_{\mD^\star_\kappa  \setminus  \overline{\Gamma_{\kappa}}}, \bt[\bu^{\textrm{f}}]{|}_{\Gamma_\kappa}, \bt[\bu^{\textrm{f}}]{|}_{\partial \mD^o_\kappa})~=~\boldsymbol{\Xi}_n ~\colon\!\!\!=~ \\*[0.2mm]
& \hspace{2.5 cm}(\bu^n {|}_{\mD^\star_\kappa  \setminus\xxs  \overline{\Gamma_{\kappa}}}\, ,\nabla \bu^n {|}_{\mD^\star_\kappa  \setminus\xxs  \overline{\Gamma_{\kappa}}} \,, \bt[\bu^n]{|}_{\Gamma_\kappa \cap \exs\overline{\mD^\star_\kappa}} \xxs\oplus\xxs \boldsymbol{\psi}^n, \boldsymbol{\phi}^n),\\*[0.2mm]
& \nabla \exs\sip\exs \bC \exs \colon \! \nabla \bu^n \,+ \rho \exs \omega^2 \bu^n ~=~ \bzero \,\,\, \text{in $\mD^\star_\kappa$}, \quad (\boldsymbol{\psi}^n, \boldsymbol{\phi}^n) \in  H^{-1/2}(\Gamma_\kappa \nxs\!\setminus\! \overline{\mD^\star_\kappa})^3) \nxs\times\nxs {H^{-1/2}(\partial \mD^o_\kappa)^3}. 
\end{aligned}
\]
Elliptic regularity implies that $\norms{\bv^n}_{H^2({\mathcal{B}}_\kappa^{-}\nxs)^3}$ is bounded uniformly with respect to $n$. Then, up to changing the initial sequence, one may assume that $\boldsymbol{\Xi}_n$ weakly converges to some $\boldsymbol{\Xi}$ in ${{S}(\mD^\star_\kappa \cup \Gamma_\kappa \nxs \cup \partial \mD^o_\kappa)}$ and $\bv^n$ converges weakly in $H^2({\mathcal{B}}_\kappa^{-}\nxs)^3 \cap H^1(\mD^\star_\kappa \!\setminus\!  \overline{\Gamma_{\kappa}})^3$ to some $\bv \in H^2({\mathcal{B}}_\kappa^{-}\nxs)^3 \cap H^1(\mD^\star_\kappa \!\setminus\!  \overline{\Gamma_{\kappa}})^3$. Now, observe that $\bv$ satisfies~\eqref{vk} for 
\[
\boldsymbol{\Xi} ~=~ (\bu {|}_{\mD^\star_\kappa  \setminus  \overline{\Gamma_{\kappa}}},\nxs \nabla \bu {|}_{\mD^\star_\kappa  \setminus  \overline{\Gamma_{\kappa}}}, \bt[\bu]{|}_{\Gamma_\kappa \cap \exs\overline{\mD^\star_\kappa}} \xxs\oplus\xxs \boldsymbol{\psi}, \boldsymbol{\phi}), \quad (\boldsymbol{\psi}, \boldsymbol{\phi}) \in  H^{-1/2}(\Gamma_\kappa \nxs\!\setminus\! \overline{\mD^\star_\kappa})^3) \nxs\times\nxs {H^{-1/2}(\partial \mD^o_\kappa)^3}, 
\]
wherein $\nabla \exs\sip\exs \bC \exs \colon \! \nabla \bu \,+ \rho \exs \omega^2 \bu ~=~ \bzero$ in $\mD^\star_\kappa$. In this setting,~\eqref{T-bound} along with $|\langle {{T_\kappa\boldsymbol{\Xi}_n}, {\boldsymbol{\Xi}_n}} \rangle| \,\to\, 0$ imply that $\llbracket \bv^n \rrbracket \to \bzero$ on $\Gamma_\kappa$ and $\nabla(\bu^n\nxs+\bv^n) \to \bzero$ in $\mD^\star_\kappa \nxs\setminus\nxs  \overline{\Gamma_{\kappa}}$, whereby the first of~\eqref{vk} reads that $\bu^n\nxs+\bv^n \to \bzero$ in $\mD^\star_\kappa \nxs\setminus\nxs  \overline{\Gamma_{\kappa}}$. One may then deduce from the unique continuation principle that the total field $\bu^n\nxs+\bv^n$ also vanishes in $\mathcal{B}^-_\kappa$. In which case, \eqref{IE} indicates that $\bv^n \to \bzero$, and thus $\bv= \bzero$, on $\partial \mathcal{B}_t$. This implies that~(a)~$\boldsymbol{\psi} = \boldsymbol{\phi} = \bzero$ by virtue of the second and fifth of~\eqref{vk} and the unique continuation, and~(b)~$\bu = \bv = \bzero$ which follows from Assumption~\ref{ITPA}.  

Next, on recalling~\eqref{DuP} and~\eqref{T},~\eqref{SeQ} may be recast as  
\beq\lb{SeQ2}
\begin{aligned}
&\langle {T_\kappa{\boldsymbol{\Xi}_n},{\boldsymbol{\Xi}_n}} \rangle \,=\, -\int_{\mD^\star_\kappa  \setminus  \overline{\Gamma_{\kappa}}} \big[\xxs \nabla \bar{\bu}^n \colon \! (\bC_{\kappa\nxs}-\bC) \xxs\colon \! \nabla (\bu^n\nxs+\bv^n) \xxs + \exs \omega^2(\rho -\rho_\kappa) \xxs \bar{\bu}^n \!\cdot\nxs (\bu^n\nxs+\bv^n)  \xxs \big]  \textrm{d}V ~+\,  \\*[0.2mm]   
&   \int_{\Gamma_\kappa \cap \exs\overline{\mD^\star_\kappa}} \bar{\bt}[\bu^n] \cdot \llbracket \bv^n \rrbracket  \, \textrm{d}S  \,+\, \int_{\Gamma_\kappa \nxs\!\setminus \overline{\mD^\star_\kappa}} \bar{\boldsymbol{\psi}}^n \nxs\cdot \llbracket \bv^n \rrbracket \, \textrm{d}S \,+\, \int_{\partial \mD^o_\kappa} \bar{\boldsymbol{\phi}}^n \! \cdot  \nxs (\bu_{\boldsymbol{\phi}}^n \nxs+\bv^n) \, \textrm{d}S.
\end{aligned}
\eeq   
where $\bu_{\boldsymbol{\phi}}^n$ solves 
\beq\lb{uphi}
\begin{aligned}
&\nabla \exs\sip\exs \bC \exs \colon \! \nabla \bu_{\boldsymbol{\phi}}^n(\bxi) \,+\, \rho \exs \omega^2 \bu_{\boldsymbol{\phi}}^n(\bxi) ~=~ \bzero, \quad & \bxi \in \mD^o_\kappa, \\*[0.0mm]
&\bn \exs\sip\exs \bC \exs \colon \!  \nabla  \bu_{\boldsymbol{\phi}}^n(\bxi)~=~{\boldsymbol{\phi}}^n(\bxi),  \quad & \bxi \in {\partial \mD^o_\kappa}.
\end{aligned}   
\eeq  
In addition, the variational form~\eqref{Wik-GE} with $\bv^\kappa=\bv' = \bv^n$ reads
\beq\lb{Wik-GE2}
\begin{aligned}
& \int_{\mD^\star_\kappa  \setminus  \overline{\Gamma_{\kappa}}} \big[\xxs \nabla \bar{\bv}^n \colon \! (\bC_{\kappa\nxs}-\bC) \xxs\colon \! \nabla (\bu^n\nxs+\bv^n) \xxs + \exs \omega^2(\rho -\rho_\kappa) \bar{\bv}^n \!\cdot\nxs (\bu^n\nxs+\bv^n)  \xxs \big]  \textrm{d}V   ~-\, \\*[0.2mm]  
-&   \int_{\Gamma_\kappa \cap \exs\overline{\mD^\star_\kappa}} {\bt}[\bu^n] \cdot \llbracket \bar{\bv}^n \rrbracket  \, \textrm{d}S  \,-\, \int_{\Gamma_\kappa \nxs\!\setminus \overline{\mD^\star_\kappa}} \boldsymbol{\psi}^n \nxs\cdot \llbracket \bar{\bv}^n \rrbracket \, \textrm{d}S \,-\, \int_{\partial \mD^o_\kappa} \boldsymbol{\phi}^n \! \cdot \bar{\bv}^n \, \textrm{d}S ~=\, \\*[0.2mm]  
 -& \int_{{\mathcal{B}}_\kappa^- \cup\xxs  \mD^\star_\kappa\nxs  \setminus  \overline{\Gamma_{\kappa}}}  \big[   \nabla \exs \bar{\bv}^n \colon\!\xxs \bC \exs\colon\! \nabla \bv^n \exs-\,  \rho \xxs  \omega^2 \exs \bar{\bv}^n \xxs\sip\exs \bv^n  \big] \, \textrm{d}V ~-\, \int_{\Gamma_\kappa} \llbracket \bar{\bv}^n \rrbracket \xxs\sip\xxs \bK_{\nxs\kappa} \llbracket \bv^n \rrbracket  \, \textrm{d}S. 
\end{aligned}
\eeq

Since $|\langle {T_\kappa{\boldsymbol{\Xi}_n},{\boldsymbol{\Xi}_n}} \rangle| \,\to\, 0$ as $n \,\to\, \infty$,~\eqref{SeQ2} in light of the Rellich compact embedding theorem along with the regularity of the trace operator implies that  
\beq\lb{SeQ3}
\int_{\mD^\star_\kappa  \setminus  \overline{\Gamma_{\kappa}}} \nabla \bar{\bu}^n \colon \! (\bC_{\kappa\nxs}-\bC) \xxs\colon \! \nabla (\bu^n\nxs+\bv^n) \xxs \textrm{d}V \,\to\,\, 0 \quad \text{as} \quad n \,\to\, \infty.
\eeq

Similarly, the compact embedding and trace theorems applied to~\eqref{Wik-GE2} reads 
\beq\lb{Wik-GE3}
\begin{aligned}
& \int_{\mD^\star_\kappa  \setminus  \overline{\Gamma_{\kappa}}} \nabla \bar{\bv}^n \colon \! (\bC_{\kappa\nxs}-\bC) \xxs\colon \! \nabla (\bu^n\nxs+\bv^n) \xxs \textrm{d}V   ~+~\nxs  \int_{{\mathcal{B}}_\kappa^- \cup\xxs  \mD^\star_\kappa\nxs  \setminus  \overline{\Gamma_{\kappa}}}  \nabla \exs \bar{\bv}^n \colon\!\xxs \bC \exs\colon\! \nabla \bv^n \xxs \textrm{d}V  \,\to\,\, 0
\end{aligned}
\eeq
as $n \,\to\, \infty$. On superimposing~\eqref{SeQ3} and~\eqref{Wik-GE3}, one finds 
\beq\nonumber
\begin{aligned}
& \int_{\mD^\star_\kappa  \setminus  \overline{\Gamma_{\kappa}}} \nabla (\bar{\bu}^n\nxs+\bar{\bv}^n) \colon \! (\bC_{\kappa\nxs}-\bC) \xxs\colon \! \nabla (\bu^n\nxs+\bv^n) \xxs \textrm{d}V   ~+~\nxs  \int_{{\mathcal{B}}_\kappa^- \cup\xxs  \mD^\star_\kappa\nxs  \setminus  \overline{\Gamma_{\kappa}}}  \nabla \exs \bar{\bv}^n \colon\!\xxs \bC \exs\colon\! \nabla \bv^n \xxs \textrm{d}V  \,\to\,\, 0
\end{aligned}
\eeq
as $n \,\to\, \infty$. Now, following the first of Assumption~\ref{Tass}, where $\mathfrak{R}(\bC_{\kappa\nxs}-\bC) - \alpha \xxs \mathfrak{I}(\bC_{\kappa\nxs})$ is positive definite on ${\mD^\star_\kappa  \nxs\setminus\nxs  \overline{\Gamma_{\kappa}}}$ for some constant $\alpha \geqslant 0$, observe that 
\beq\lb{Wik-GE5}
\begin{aligned}
& \Big|\int_{\mD^\star_\kappa  \setminus  \overline{\Gamma_{\kappa}}} \nabla (\bar{\bu}^n\nxs+\bar{\bv}^n) \colon \! (\bC_{\kappa\nxs}-\bC) \xxs\colon \! \nabla (\bu^n\nxs+\bv^n) \xxs \textrm{d}V   \,+\,\nxs  \int_{{\mathcal{B}}_\kappa^- \cup\xxs  \mD^\star_\kappa\nxs  \setminus  \overline{\Gamma_{\kappa}}}  \nabla \exs \bar{\bv}^n \colon\!\xxs \bC \exs\colon\! \nabla \bv^n \xxs \textrm{d}V\Big| \,\geqslant  \\*[0.2mm]  
& \hspace{2.5cm} \theta \Big(\int_{\mD^\star_\kappa  \setminus  \overline{\Gamma_{\kappa}}} \big|\nabla (\bu^n\nxs+\bv^n)\big|^2 \xxs \textrm{d}V   \,+\,\nxs  \int_{{\mathcal{B}}_\kappa^- \cup\xxs  \mD^\star_\kappa\nxs  \setminus  \overline{\Gamma_{\kappa}}}  \big|\nabla \bv^n\big|^2 \xxs \textrm{d}V\Big)
\end{aligned}
\eeq
for some $\theta>0$ independent of $n$. This implies that $\bv^n \,\to\, \bzero$ strongly in $H^1({{\mathcal{B}}_\kappa^- \cup\xxs  \mD^\star_\kappa\nxs  \setminus  \overline{\Gamma_{\kappa}}})^3$ which is a contradiction. Given~\eqref{SeQ3} and~\eqref{Wik-GE3}, the argument for establishing~\eqref{co-T0} for the second case of Assumption~\ref{Tass} directly follows the proof of Theorem 2.42 in~\cite{cako2016}.
\end{proof}

\begin{lemma}\lb{T-FI0}
Under Assumptions~\ref{ITPA} and~\ref{Tass}, the real part of operator $T_\kappa \colon {\mathfrak{S}}(H_{\Delta}) \rightarrow \, \tilde{S}(\mD^\star_\kappa \cup \Gamma_\kappa \nxs \cup \partial \mD^o_\kappa)$ may be decomposed on ${\mathfrak{S}}(H_{\Delta})$ into a coercive part $T^\pm_\circ$ and a compact part $T_\text{c}$. 
\end{lemma} 

\begin{proof}
See~Appendix B.
\end{proof}

\begin{lemma}\lb{FF_op}
The scattering operator $\Lambda_\kappa: L^2(\partial {\mathcal{B}_t})^3 \to\exs L^2(\partial {\mathcal{B}_t})^3$ is injective, compact and has a dense range.  
\end{lemma}
\begin{proof}
The injectivity (\emph{resp.}~compactness) of $\Lambda_\kappa = {\mathcal{V}_\kappa} \exs \mathcal{S}_\kappa$ results from the injectivity (\emph{resp.}~compactness) of ${\mathcal{V}_\kappa}$ and $\mathcal{S}_\kappa$ as per Lemmas~\ref{T-invs0} and~\ref{C(R(T))}. Now, according to Lemmas~\ref{H*p} and~\ref{T-invs0}, the adjoint operators $\mathcal{S}^*_\kappa$, ${\mathcal{V}^*_\kappa}$, and thus $\Lambda^*_\kappa = \mathcal{S}^*_\kappa \exs {\mathcal{V}^*_\kappa}$ are injective which establish the denseness of the range of $\Lambda_\kappa$.   
\end{proof}

\section{Design of imaging functionals}\label{DEI}
Let us generate a set of sampling points ${\bx_\circ} \in \mathcal{B}$ in the baseline model designating the loci of (monopole and dipole) trial scatterers. Monopole signatures are created via point sources applied along a set of trial directions ${\bn}$, while dipole patterns are constructed by nucleating dislocations $L \colon\!\!=\bx_\circ\!+\bR{\sf L} \subset \mathcal{B}$ in the baseline model wherein~$\sf L$ is a smooth arbitrary-shaped discontinuity whose orientation is given by the unitary rotation matrix $\bR\nxs\in\nxs U(3)$. In this setting, the scattering pattern $\bPsi^{o} \colon \tilde{H}^{1/2}(L)^3 \rightarrow L^2(\partial \mathcal{B}_t)^3$ is defined by
\beq\lb{RHS}
\bPsi^{o}(\bxi)~\nxs\colon\!\!\!=~(1-o) \exs \bn \exs\sip\exs \boldsymbol{G}(\bxi,\bx_\circ) \exs+\exs o \int_{L} \ba(\by) \cdot \boldsymbol{T}(\bxi,\by) \, \textrm{d}S_{\by}, \quad o \in \lbrace 0,1 \rbrace, \,\, \bxi \in \partial \mathcal{B}_t,
\eeq
for any admissible density $\ba \in \tilde{H}^{1/2}(L)^3$. Keep in mind that the Green's dyadic~$\boldsymbol{G}$ satisfies~\eqref{Grf} and its affiliated traction $\boldsymbol{T}$ on $L$ is specified in~\eqref{IE}. 

To construct a sampling-based imaging functional, we deploy $\bPsi^{o}$ to explore the range of $\Lambda_\kappa$ by minimizing the below sequence of cost functions
\beq\lb{GCf}
\mathfrak{J}_\kappa^\gamma(\bPsi^{o};\,\btau) ~:=~ \norms{\nxs \Lambda_\kappa\exs \btau\,-\,\bPsi^{o} \nxs}_{L^2(\partial \mathcal{B}_t)}^2 +~ \gamma\, \big(\btau, \Lambda_{\kappa_\sharp}\btau\big)_{L^2(\partial \mathcal{B}_t)}, \quad \btau\in L^2(\partial \mathcal{B}_t)^3, \,\,\, \gamma>0,
\eeq
where $\Lambda_{\kappa_\sharp} \colon  L^2(\partial \mathcal{B}_t)^3  \rightarrow L^2(\partial \mathcal{B}_t)^3$ is given by 
\beq\lb{Fs}
\Lambda_{\kappa_\sharp}\,\colon \!\!\!=\, \frac{1}{2} |\Lambda_{\kappa}+\Lambda_{\kappa}^*| \:+\: \frac{1}{2 \textrm{\emph{i}}} (\Lambda_{\kappa}-\Lambda_{\kappa}^*).
\eeq

\begin{rem}
Lemmas~\ref{H*p},~\ref{I{T}>0},~\ref{comp_G},~\ref{T-FI0} establish the premises of~\cite[Theorem 2.15]{kirs008} which concludes that operator $\Lambda_{\kappa_\sharp}$ is positive and has the following factorization
\vspace{-1 mm}
\beq\lb{facts2}
\Lambda_{\kappa_\sharp} ~=~ \mathcal{S}_{\kappa}^* \exs T_{\kappa_\sharp} \exs \mathcal{S}_{\kappa}, 
\vspace{-1 mm}
\eeq
where the middle operator $T_{\kappa_\sharp}$ is selfadjoint and positively coercive, i.e., there exists a constant~$\text{\emph{c}}\!>\!0$ independent of~${\boldsymbol{\Xi}}$ so that 
\beq\lb{co-T}
\langle {T_{\kappa_\sharp}{\boldsymbol{\Xi}}, {\boldsymbol{\Xi}}} \xxs \rangle \,\,\geqslant\,\, \text{\emph{c}} \nxs \norms{\boldsymbol{\Xi}}_{S(\mD^\star_\kappa \cup \Gamma_\kappa \nxs \cup \partial \mD^o_\kappa)}^2, \qquad \forall \boldsymbol{\Xi} \in {\mathfrak{S}}(H_{\Delta}).
\eeq 

 Moreover, the range of $\mathcal{S}_\kappa^*$ coincides with that of $\Lambda_{\kappa_\sharp}^{1/2}$.
\end{rem}

Assumptions and lemmas of Section~\ref{SFS} furnish all the necessary conditions for the fundamental theorems of GLSM~\cite{AudiDLSM,pour2020} to apply. These results are required for the differential imaging indicators which for future reference are included in the following.    

\begin{theorem}[\!\cite{AudiDLSM,pour2020}] \lb{GLSM1}
Consider the minimizing sequence $\btau^\gamma \in L^2(\partial \mathcal{B}_t)^3$ for $\mathfrak{J}_\kappa^\gamma$ such that   
\vspace{-1mm}
\beq\lb{mseq1}
\mathfrak{J}_\kappa^\gamma(\bPsi^{o};\btau^\gamma) \,\leqslant \, \mathfrak{j}_\kappa^\gamma(\bPsi^{o}) + \eta(\gamma), \qquad \gamma > 0, \vspace{-1mm}
\eeq
where $\eta(\gamma)/\gamma \rightarrow 0$ as $\gamma \rightarrow 0$ and
\vspace{-1mm}
\[
\mathfrak{j}_\kappa^\gamma(\bPsi^{o}) ~\colon\!\!\!=~\!\! \inf\limits_{\btau \in L^2(\partial \mathcal{B}_t)^3} \! \mathfrak{J}_\kappa^\gamma(\bPsi^{o};\btau). 
\vspace{-4mm}
\]

Then,  
\vspace{-1mm}
\beq\lb{statG}
\begin{aligned}
& \bPsi^{o} \in Range(\mathcal{V}_\kappa) ~~\Rightarrow~~ \lim\limits_{\gamma \rightarrow 0} \big(\exs \btau^\gamma, \Lambda_{\kappa_\sharp}\btau^\gamma \exs \big) \, < \, \infty, \\
&\bPsi^{o} \notin Range(\mathcal{V}_\kappa) ~~\Rightarrow~~ \liminf\limits_{\gamma \rightarrow 0} \big(\exs \btau^\gamma, \Lambda_{\kappa_\sharp}\btau^\gamma \exs \big) \, = \, \infty.
\end{aligned}
\vspace{-2mm}
\eeq

Moreover, when $\mathcal{V}_\kappa \exs \boldsymbol{\Xi} =  \bPsi^{o}$, the sequence $\mathcal{S}_\kappa \btau^\gamma$ strongly converges to $\boldsymbol{\Xi} \in {\mathfrak{S}}(H_{\Delta})$ as $\gamma \rightarrow 0$. 
\end{theorem}

\begin{coro}\lb{GLSM-ITP} 
Under Assumptions~\ref{EigFT} and~\ref{ITPA}, 
\[
\bPsi^{o} \in Range(\mathcal{V}_\kappa) \, \iff\, \bx_\circ \in \mD_\kappa \cup \Gamma_\kappa.
\] 
In addition, if 
\begin{itemize}
\item~$\bx_\circ \in \mD_\kappa^\star$ then there exists a unique solution $(\bu^\star_\kappa, \bw^\star_\kappa)$ to 
\beq\lb{ITP*}
\text{ITP}_\kappa ~\colon\!\!\!=~ \text{ITP}(\mD^\star_\kappa, \Gamma_{\kappa}; \lbrace\bC,\rho\rbrace, \lbrace\bC_\kappa,\rho_\kappa\rbrace, \bK_{\kappa}; \bn \xxs\sip\xxs \bC \xxs\colon \!\nxs  \nabla  \bPsi^{o}, \bPsi^{o}).\eeq
\item~$\bx_\circ \in \mD_\kappa^o$ then there exists a unique field $\boldsymbol{\sf u}^o_\kappa$ satisfying
\beq\lb{ITPo}
\begin{aligned}
&\nabla \xxs \sip \xxs (\bC \exs\colon \! \nabla \boldsymbol{\sf u}^o_\kappa) \,+\, \rho \exs \omega^2 \boldsymbol{\sf u}^o_\kappa ~=~ \bzero     \quad  &\textrm{in}~ \mD_\kappa^o, \\*[1 mm]
&\bn \xxs\sip\xxs \bC \xxs\colon \!\nxs  \nabla  (\boldsymbol{\sf u}^o_\kappa +\xxs \bPsi^{o}) ~=~ \bzero   \quad &\textrm{on}~ \partial \mD_\kappa^o.
\end{aligned}
\eeq
\item~$L\subset\Gamma_{\kappa} \!\setminus\! \overline{\mD^\star_\kappa}$ then there exists a unique $\llbracket \boldsymbol{\sf v}_{\kappa} \rrbracket \in\tilde{H}^{1/2}(\Gamma_{\kappa} \!\setminus\! \overline{\mD^\star_\kappa})^3$ such that $\mathcal{S}_{\kappa}^* \llbracket \boldsymbol{\sf v}_{\kappa} \rrbracket = \bPsi^{1}$. In this setting, the affiliated free-field traction $\bt[\bu^{\sf{f}}_\kappa]$ may be obtained from the second of~\eqref{vk},
\beq\lb{tufs}
\bt[\bu^{\sf{f}}_\kappa](\bxi)  ~=~  \bK_{\kappa}(\bxi) \llbracket \boldsymbol{\sf v}_{\kappa} \rrbracket(\bxi)  \,-\, \bn \xxs\sip\xxs \bC \xxs\colon \!\nxs  \nabla \bPsi^{1}(\bxi),  \qquad \bxi \in \Gamma_{\kappa} \!\setminus\! \overline{\mD^\star_\kappa}.
\eeq
\end{itemize}
\end{coro}
Now, let us recall from~\eqref{HT} that~(a)~$\mE^\star_i \subset \mathcal{B}$, $i \in \lbrace 1,2, \ldots \rbrace$, is the support of (volumetric) elastic transformation where by Assumption~\ref{ITP_A1} $\overline{\text{supp}(\bC_{i} - \bC_{i-1})} \exs=\, \overline{\text{supp}(\rho_{i} - \rho_{i-1})}$, and $\mD^\star_{i} = \mE^\star_i \cup \mD^\star_{i-1}$ since $\mD^\star_{i-1} \subset \mD^\star_{i}$,~(b)~$\mE^o_i$ designates the evolution of pore volume which is disjoint from $\mE^\star_i$ since $\overline{\mD^\star_i} \cap \overline{\mD^\circ_i} = \emptyset$, and~(c)~$\hat\Gamma_i \cup \tilde\Gamma_i$ represents the support of (geometric $\hat\Gamma_i$ and elastic $\tilde\Gamma_i$) interfacial evolution. On denoting by $\mD^\star_{i-1,j}$, $j = 1,2,\ldots, N_{i-1}$~(\emph{resp.}~$\mD^\star_{i,j}$, $j = 1,2,\ldots, N_{i}$) the simply connected components of $\mD^\star_{i-1}$ (\emph{resp.}~$\mD^\star_{i}$), one may define the set of stationary inclusions 
\beq\lb{stsc*}
\check{\mD}_{i-1}^\star \exs=\exs \bigcup_{j \exs\in\exs \iota} \mD^\star_{i-1,j}, \quad \iota \,=\, \big\lbrace \exs  j \exs {|} \exs \exists  \varkappa \,\, \mD^\star_{i-1,j} \exs =\exs \mD^\star_{i,\varkappa} ~\land~ \overline{\textcolor{white}{\tilde{1}\!\!}\mD_{i-1,j}^\star} \cap \overline{\mE^\star_i \cup \hat\Gamma_i \cup \tilde\Gamma_i} \exs =\exs \emptyset \big\rbrace,
\eeq   
which remain unchanged between $[t_{i-1} \,\, t_{i}]$. By adopting a similar notation, the stationary pores are identified by
\beq\lb{stsco}
\check{\mD}_{i-1}^o \exs=\exs \bigcup_{j \exs\in\exs \iota} \mD^o_{i-1,j}, \quad \iota \,=\, \big\lbrace \exs  j \exs {|} \exs \exists  \varkappa \,\, \mD^o_{i-1,j} \exs=\exs \mD^o_{i,\varkappa} \big\rbrace.
\eeq          

In this setting, $\tilde{\mD}_{i-1}^\tau \colon\!\!\! = {\mD}_{i-1}^\tau \!\setminus\nxs \overline{\check{\mD}_{i-1}^\tau}$, $\tau = \lbrace \star, o \rbrace$, signifies the evolved subset of ${\mD}_{i-1}^\tau$ within the same timeframe. Further, one may introduce 
\beq\lb{evsc}
\tilde{\mD}_{i}^\tau = \bigcup_{j \exs\in\exs \tilde{\iota}} \mD^\tau_{i,j}, \quad \tilde{\iota} \,=\, \big\lbrace \exs  j \exs {|} \exs \overline{\textcolor{white}{\tilde{1}\!\!}\mD^\tau_{i,j}} \cap \overline{\tilde{\mD}_{i-1}^\tau} \exs\neq\exs \emptyset \big\rbrace, \quad \tau = \lbrace \star, o \rbrace,
\eeq
so that $\mE^\tau_i$ may be decomposed into disjoint subsets $\overline{\tilde{\mE}^\tau_i} = \overline{\textcolor{white}{\tilde{1}\!\!}\mE^\tau_i} \cap \overline{\tilde{\mD}_{i}^\tau}$ and $\hat{\mE}^\tau_i = {\mD}_{i}^\tau \!\setminus\! \lbrace \check{\mD}_{i-1}^\tau \cup \tilde{\mD}_{i}^\tau \rbrace$. Based in this, let us in addition define
\beq\lb{chgam}
\check{\Gamma}_{i-1}  ~=~ \Gamma_{i-1} \!\setminus \overline{\tilde{\Gamma}_{i} \cup \tilde{\mD}^\star_{i} \cup \mE^\star_{i} \cup \mE^o_{i}}.
\eeq  

While our objective is to design imaging functionals to reconstruct $\mE^\star_i \cup \mE^o_i \cup \hat\Gamma_i \cup \tilde\Gamma_i$ given sequential sensory data at $t_{i-1}$ and $t_i$, one may observe in what follows that the proposed indicator is capable of recovering either $\tilde{\mD}_{i-1}^o \cup \tilde{\mD}_{i-1}^\star \cup \tilde{\Gamma}_i$ or $\mE^\star_i \cup \mE^o_i \cup \hat\Gamma_i \cup \tilde\Gamma_i \cup \tilde{\mD}_{i-1}^\star \cup \tilde{\mD}_{i-1}^o$. 

\begin{assum}\lb{Inc*}
Let us define ${(}\tilde{\bC}, \tilde{\rho}\xxs{)}$ in $ \tilde{\mD}_{i}^\star$ by
\beq\lb{TCTr}
\begin{aligned}
{(}\tilde{\bC}, \tilde{\rho}\xxs{)}(\bxi) ~\colon \!\!\! =~
\!\left\{\begin{array}{l}
\begin{aligned}
&\!\! {(}\bC_{i-1},\rho_{i-1}{)}(\bxi), \!\!\!\! & \!\! \bxi \exs \in  \tilde{\mD}_{i-1}^\star   \!\!\! \\*[0.75mm] 
& \!\! {(}\bC,\rho{)}, \!\!\!\!  & \!\! \bxi \exs \in  \tilde{\mD}_{i}^\star \!\setminus\! \tilde{\mD}_{i-1}^\star  \!\!\! 
\end{aligned}
\end{array}\right.,
\end{aligned}
\eeq
then $\omega>0$ is not a transmission eigenvalue solving

\[
\text{ITP}_\circ^{\xxs\text{\sf f}}(\tilde{\mE}_{i}^\star, \Gamma_{i-1}, \Gamma_i; \lbrace\tilde{\bC},\tilde{\rho}\rbrace, \lbrace\bC_i,\rho_i\rbrace, \bK_{i-1}, \bK_{i})~\colon\!\!\!=~
\]
\beq\lb{ITPE2}
\!\left\{\begin{array}{l} \!\!\!
\begin{aligned}
&\nabla \exs\sip\exs [\exs \bC_i(\bxi) \exs \colon \! \nabla \bw^\star_i\exs](\bxi) \,+\, \rho_i(\bxi) \exs \omega^2 \bw^\star_i(\bxi) ~=~ \bzero, \quad &  \bxi \exs \in \tilde{\mE}_{i}^\star  \!\setminus\! \overline{\Gamma_{i}},   \\*[0.0mm]
&\nabla \exs\sip\exs [\exs \tilde{\bC}(\bxi) \exs \colon \! \nabla \tilde{\bw}^\star_i\exs](\bxi) \,+\, \tilde{\rho}(\bxi) \exs \omega^2 \tilde{\bw}^\star_i(\bxi) ~=~ \bzero, \quad &  \bxi \exs \in \tilde{\mE}_{i}^\star  \!\setminus\! \overline{\Gamma_{i-1}},  \\*[0.2mm]
&\bt[\bw^\star_i](\bxi) - \xxs\bt[\tilde{\bw}^\star_i](\bxi) ~=~ \bzero,  \,\,\, [\bw^\star_i-\tilde{\bw}^\star_i](\bxi) ~=~ \bzero, \quad  & \bxi \in \partial \tilde{\mE}^\star_i  \!\setminus\! \overline{\hat{\Gamma}_i \cup \tilde{\Gamma}_i}, \\*[0.2mm]
&\bt[\bw^\star_i](\bxi) ~=~\bK_{i}(\bxi) \llbracket \bw^\star_i \rrbracket(\bxi),  \,\,\, \llbracket \bt[\bw^\star_i] \rrbracket(\bxi)~=~\bzero,  \,\, & \bxi \in \Gamma_i \cap  \tilde{\mE}^\star_i,  \\*[0.2mm] 
&\bt[\tilde{\bw}^\star_i](\bxi) ~=~\bK_{i-1}(\bxi) \llbracket \tilde{\bw}^\star_i \rrbracket(\bxi),  \,\,\, \llbracket \bt[\tilde{\bw}^\star_i] \rrbracket(\bxi)~=~\bzero,  \,\, & \bxi \in \Gamma_{i-1} \cap  \tilde{\mE}^\star_i,  \\*[0.2mm] 
&\bt[\bw^\star_i](\bxi) - \xxs\bt[\tilde{\bw}^\star_i](\bxi) ~=~ \bzero,  \,\,\, \bt[\bw^\star_i](\bxi)~=~\bK_{i}(\bxi) [\xxs \tilde{\bw}^\star_i\nxs-\bw^\star_i](\bxi),  & \bxi \in \partial\tilde{\mE}^\star_i \cap \hat{\Gamma}_i, \\*[0.2mm] 
&\bt[\bw^\star_i](\bxi) - \xxs\bt[\tilde{\bw}^\star_i](\bxi) ~=~ \bzero,  \,\,\, (\bI \nxs-\nxs \bK_i \bK_{i-1}^{-1})\bt[\bw^\star_i](\bxi)~=~\bK_{i}(\bxi) [\xxs \tilde{\bw}^\star_i\nxs-\bw^\star_i](\bxi),  & \bxi \in \partial\tilde{\mE}^\star_i \cap \tilde{\Gamma}_i,
\end{aligned} 
\end{array}\right.  
\eeq 

Note that special cases such as elastic transformation or growth of intact inclusions are also included in~\eqref{ITPE2} and may be obtained by setting $\Gamma_{i-1} = \emptyset$ and/or $\tilde{\Gamma}_{i} = \hat{\Gamma}_i = \emptyset$ in~\eqref{ITPE2}. Further, in the case of $\Gamma_{i-1} \cap \hat{\mE}_{i}^\star \neq \emptyset$, pertinent to the transformation of microcracked zones, it is further assumed that $\omega>0$ does not satisfy 
$
\text{ITP}_\circ^{\xxs\text{\sf f}}(\hat{\mE}_{i}^\star, \Gamma_{i-1}, \Gamma_i; \lbrace {\bC},{\rho}\rbrace, \lbrace\bC_i,\rho_i\rbrace, \bK_{i-1}, \bK_{i}).
$ 
\end{assum}

\begin{theorem}\label{Lxo} 
Given Assumptions~\ref{EigFT},~\ref{Frac},~\ref{ITPA},~\ref{Tass}, and~\ref{Inc*},
\begin{itemize}
\item{}~Let $\bx_\circ \in {\mD}_{i-1}^\star$ (or $L \subset {\mD}_{i-1}^\star$), then denote by $\bt[\bu^o_{s}]$ (\emph{resp.}~$\bt[\bu^{\sf{f}}_{s}]$) the free-field traction on $\partial \mD_s^o$ (\emph{resp.}~$\Gamma_s$) affiliated with the scattering pattern $\bPsi^{o}\nxs$ (or $\btau^\gamma$), while $(\bu^\star_{s}, \bw^\star_{s})$ uniquely solves $\text{ITP}_{s}$ at times $t_s$, $s \in \lbrace i-1, i \rbrace$. In this setting, 
\beq\lb{statsc}
\begin{aligned}
&\text{If}~~ \bx_\circ \in \check{\mD}_{i-1}^\star \text{\,\,(or\,} L \subset \check{\mD}_{i-1}^\star \text{)} ~~\text{then}~~ \\*[0.2mm]
&\hspace{0.3cm} (\bu^\star_{i} \exs {|}_{\mD^\star_{i-1}  \nxs\setminus  \overline{\Gamma_{i-1}}}, \nabla \bu^\star_{i} {|}_{\mD^\star_{i-1}  \setminus  \overline{\Gamma_{i-1}}}, \,\bt[\bu_{i}^{\emph{\sf{f}}}]{|}_{\Gamma_{i-1}}, \, \bt[\bu_{i}^{o}]{|}_{\partial \mD^o_{i-1}})~=~ \\*[0.2mm]
&\hspace{2cm} (\bu^\star_{i-1} \exs {|}_{{\mD}^\star_{i-1}  \!\setminus  \overline{\Gamma_{i-1}}}, \nabla \bu^\star_{i-1} {|}_{{\mD}^\star_{i-1}  \!\setminus  \overline{\Gamma_{i-1}}}, \, \bt[\bu_{i-1}^{\emph{\sf{f}}}]{|}_{{\Gamma_{i-1}}}, \, \bt[\bu_{i-1}^{o}]{|}_{\partial \mD^o_{i-1}}). \\*[0.2mm]
\end{aligned}
\eeq
\beq\lb{Vevsc}
\begin{aligned}
&\hspace{-0.3cm}\text{If}~~ \bx_\circ \in \tilde{\mD}_{i-1}^\star \text{\,\,(or\,} L \subset \tilde{\mD}_{i-1}^\star \text{)} ~~\text{then}~~ \\*[0.2mm]
&(\bu^\star_{i} \exs {|}_{\mD^\star_{i-1}  \nxs\setminus  \overline{\Gamma_{i-1}}}, \nabla \bu^\star_{i} {|}_{\mD^\star_{i-1}  \setminus  \overline{\Gamma_{i-1}}}, \,\bt[\bu_{i}^{\emph{\sf{f}}}]{|}_{\Gamma_{i-1} \cap \exs \overline{\mD^\star_{i-1}}})~\neq~ \\*[0.2mm]
&\hspace{2.85cm} (\bu^\star_{i-1} \exs {|}_{{\mD}^\star_{i-1}  \!\setminus  \overline{\Gamma_{i-1}}}, \nabla \bu^\star_{i-1} {|}_{{\mD}^\star_{i-1}  \!\setminus  \overline{\Gamma_{i-1}}}, \, \bt[\bu_{i-1}^{\emph{\sf{f}}}]{|}_{{\Gamma_{i-1}} \cap\xxs \overline{{\mD}^\star_{i-1}}}). \\*[0.6mm]
\end{aligned}
\eeq
\item{}~Provided that $\omega$ is not a Neumann eigenvalue of~\eqref{ITPo} per Assumption~\ref{Frac}, then 
\[
\begin{aligned}
&\text{If}~~ \bx_\circ \in \check{\mD}_{i-1}^o \text{\,\,(or\,} L \subset \check{\mD}_{i-1}^o \text{)} ~~\text{then~~\eqref{statsc} applies over ${\mD}_{i-1}^\star \cup {\mD}_{i-1}^o \cup \Gamma_{i-1}$.}  \\*[0.2mm]
&\text{If}~~  \bx_\circ \in \tilde{\mD}_{i-1}^o \text{\,\,(or\,} L \subset \tilde{\mD}_{i-1}^o \text{)} ~~\text{then~~} \\*[0.2mm]
\end{aligned}
\]
\beq\lb{Hevsc}
\bt[\bu_{i}^{o}]{|}_{\partial \mD^o_{i-1}} \neq~\exs  \bt[\bu_{i-1}^{o}]{|}_{\partial \mD^o_{i-1}}. 
\eeq
\item{} Moreover,  
\[
\begin{aligned}
&\text{If}~~ L \subset \check{\Gamma}_{i-1}  ~\text{then~\eqref{statsc} holds.}~~ \\*[0.2mm]
&\text{If}~~ L \subset {\Gamma}_{i-1} \!\setminus\nxs \overline{\check{\Gamma}_{i-1} \cup \check{\mD}_{i-1}^\star \cup \tilde{\mD}_{i}^\star}  ~~\text{then}~~ \\*[0.2mm]
\end{aligned}
\]
\beq\lb{Gevsc}
\bt[\bu_{i}^{\emph{\sf{f}}}]{|}_{{\Gamma}_{i-1}} \neq~\exs  \bt[\bu_{i-1}^{\emph{\sf{f}}}]{|}_{{\Gamma}_{i-1}}.
\eeq
\end{itemize}
\end{theorem}

\begin{proof}
Let $\bx_\circ \in \check{\mD}_{i-1}^\star$ (\emph{resp.}~$L \subset \check{\mD}_{i-1}^\star$), then observe that~(a) $\bPsi^{0}\nxs$ (\emph{resp.}~$\bPsi^{1}\nxs$) of~\eqref{RHS} satisfies $\nabla \exs\sip\exs \bC \exs \colon \! \nabla\bPsi^{o}(\bxi) + \rho \exs \omega^2 \bPsi^{o}(\bxi) = \bzero$ in $\bxi \in \mD_i \!\setminus\! \overline{\check{\mD}_{i-1}^\star}$ with $o = \lbrace 0,1 \rbrace$, and thus, the solutions to $\text{ITP}_{s}$ of \eqref{ITP*} in $\mD^\star_s \!\setminus\! \overline{\check{\mD}_{i-1}^\star}$ is given by $(\bu^\star_{s}, \bw^\star_{s}) \,=\, (-\bPsi^{o}, \bzero)$ for $s = \lbrace i, i-1 \rbrace$,~(b)~$\bt[\bu_{s}^{o}]{|}_{\partial \mD^o_{s}} = - \bt[\bPsi^{o}]{|}_{\partial \mD^o_{s}}$ in light of~(a) and~\eqref{ITPo}, and~(c)~in Corollary~\ref{GLSM-ITP}, $\llbracket \boldsymbol{\sf v}_{s} \rrbracket = \bzero$ on $\Gamma_{s} \nxs\!\setminus\! \overline{\check{\mD}^\star_{i-1}}$ which with reference to the contact law per the second of~\eqref{vk} implies $\bt[\bu_{s}^{\sf{f}}]{|}_{\Gamma_{s} \nxs\setminus \overline{\check{\mD}^\star_{i-1}}} = - \bt[\bPsi^{o}]{|}_{\Gamma_{s} \nxs\setminus \overline{\check{\mD}^\star_{i-1}}}$. Further, note from~\eqref{ITP*} and~\eqref{stsc*} that by definition $\text{ITP}_{i-1} = \exs \text{ITP}_{i}$ in $\check{\mD}_{i-1}^\star$ so that 
\[
\begin{aligned}
&(\bu^\star_{i-1} \exs {|}_{\mD^\star_{i-1}  \!\setminus  \overline{\Gamma_{i-1}}}, \nabla \bu^\star_{i-1} {|}_{\mD^\star_{i-1}  \setminus  \overline{\Gamma_{i-1}}}, \,\bt[\bu_{i-1}^{\emph{\sf{f}}}]{|}_{\Gamma_{i-1}}, \, \bt[\bu_{i-1}^{o}]{|}_{\partial \mD^o_{i-1}})~=~ \\*[0.2mm]
&\hspace{1.6cm}(\bu^\star_{i-1} \exs {|}_{\check{\mD}^\star_{i-1}  \!\setminus  \overline{\Gamma_{i-1}}} \oplus -\bPsi^{o} \exs {|}_{\tilde{\mD}^\star_{i-1}}, \nabla \bu^\star_{i-1} {|}_{\check{\mD}^\star_{i-1}  \setminus  \overline{\Gamma_{i-1}}} \oplus -\nabla \bPsi^{o} \exs {|}_{\tilde{\mD}^\star_{i-1}}, \\*[0.2mm]
&\hspace{3.9cm}\bt[\bu_{i-1}^\star]{|}_{{\Gamma_{i-1}} \cap\xxs \overline{\check{\mD}^\star_{i-1}}} \!\oplus -\bt[\bPsi^{o}]{|}_{{\Gamma_{i-1}} \!\setminus \overline{\check{\mD}^\star_{i-1}}}, \, -\bt[\bPsi^{o}]{|}_{\partial \mD^o_{i-1}}),  \\*[0.2mm]
&(\bu^\star_{i} \exs {|}_{\mD^\star_{i}  \nxs\setminus  \overline{\Gamma_{i}}}, \nabla \bu^\star_{i} {|}_{\mD^\star_{i}  \setminus  \overline{\Gamma_{i}}}, \,\bt[\bu_{i}^{\emph{\sf{f}}}]{|}_{\Gamma_{i}}, \, \bt[\bu_{i}^{o}]{|}_{\partial \mD^o_{i}})~=~ (\bu^\star_{i-1} \exs {|}_{\check{\mD}^\star_{i-1}  \!\setminus  \overline{\Gamma_{i-1}}} \oplus -\bPsi^{o} \exs {|}_{\hat{\mE}^\star_i \cup \tilde{\mD}^\star_{i} \nxs\setminus  \overline{\Gamma_{i}}}, \\*[0.2mm]
&\hspace{0.25cm} \nabla \bu^\star_{i-1} {|}_{\check{\mD}^\star_{i-1}  \!\setminus  \overline{\Gamma_{i-1}}} \oplus -\nabla \bPsi^{o} \exs {|}_{\hat{\mE}^\star_i \cup \tilde{\mD}^\star_{i} \nxs\setminus  \overline{\Gamma_{i}}}, \bt[\bu_{i-1}^\star]{|}_{{\Gamma_{i-1}} \cap\xxs \overline{\check{\mD}^\star_{i-1}}} \nxs \oplus -\bt[\bPsi^{o}]{|}_{{\Gamma_{i}} \nxs\setminus \overline{\check{\mD}^\star_{i-1}}},  -\bt[\bPsi^{o}]{|}_{\partial \mD^o_{i}}), 
\end{aligned}
\] 
establishing~\eqref{statsc}. When $\bx_\circ \in \check{\mD}_{i-1}^o$ or $L \subset \check{\mD}_{i-1}^o$, a similar argument leveraging~\eqref{ITPo} leads to
\[
\begin{aligned}
&(\bu^\star_{s} \exs {|}_{\mD^\star_{s} \!\setminus \overline{\Gamma_{s}}}, \nabla \bu^\star_{s} {|}_{\mD^\star_{s} \!\setminus \overline{\Gamma_{s}}}, \,\bt[\bu_{s}^{\emph{\sf{f}}}]{|}_{\Gamma_{s}}, \, \bt[\bu_{s}^{o}]{|}_{\partial \mD^o_{s}})~=~ ( -\bPsi^{o} \exs {|}_{\mD^\star_{s} \!\setminus \overline{\Gamma_{s}}}, \\*[0.2mm]
&\hspace{0.5cm}  -\nabla \bPsi^{o} \exs {|}_{\mD^\star_{s} \!\setminus \overline{\Gamma_{s}}}, -\bt[\bPsi^{o}]{|}_{\Gamma_{s}},  \bt[\bu_{i-1}^{o}]{|}_{\partial \check{\mD}^o_{i-1}} \!\oplus -\bt[\bPsi^{o}]{|}_{\partial \mD^o_{i} \setminus \overline{\partial \check{\mD}_{i-1}^o}}), \,\,\ s \in \lbrace i-1, i \rbrace, 
\end{aligned}
\]
which confirms~\eqref{statsc}. Same argument as above along with the proof of Theorem 4.5 in~\cite{pour2020} leads to~\eqref{statsc} when $L \subset \check{\Gamma}_{i-1}$.

Let $\bx_\circ \in \tilde{\mD}_{i-1}^\star$ (or~$L \subset \tilde{\mD}_{i-1}^\star$), then in light of the above observe that
\[
\begin{aligned}
&(\bu^\star_{s} \exs {|}_{\mD^\star_{s} \nxs\setminus\xxs \overline{\tilde{\mD}_{s}^\star \cup \xxs \Gamma_{s}}}, \nabla \bu^\star_{s} {|}_{\mD^\star_{s} \nxs\setminus\xxs \overline{\tilde{\mD}_{s}^\star \cup \xxs \Gamma_{s}}}, \,\bt[\bu_{s}^{\emph{\sf{f}}}]{|}_{\Gamma_{s} \nxs\setminus\xxs \overline{\tilde{\mD}_{s}^\star}}, \, \bt[\bu_{s}^{o}]{|}_{\partial \mD^o_{s}})~=~ \\*[0.2mm]
&\hspace{0.5cm} ( -\bPsi^{o} \exs {|}_{\mD^\star_{s} \nxs\setminus\xxs \overline{\tilde{\mD}_{s}^\star \cup \xxs \Gamma_{s}}}, -\nabla \bPsi^{o} \exs {|}_{\mD^\star_{s} \nxs\setminus\xxs \overline{\tilde{\mD}_{s}^\star \cup \xxs \Gamma_{s}}}, -\bt[\bPsi^{o}]{|}_{\Gamma_{s} \nxs\setminus\xxs \overline{\tilde{\mD}_{s}^\star}}, -\bt[\bPsi^{o}]{|}_{\partial \mD^o_{s}}), \,\,\ s \in \lbrace i-1, i \rbrace. 
\end{aligned}
\]

Next, a contradiction argument is adopted to analyze $\bu^\star_{s}$ in $\tilde{\mD}_{s}^\star$, $s \in \lbrace i-1, i \rbrace$, as the following. Suppose that~\eqref{Vevsc} does not hold i.e.,
\beq\lb{CA1}
\begin{aligned}
&(\bu^\star_{i} \exs {|}_{\tilde{\mD}^\star_{i-1}  \nxs\setminus  \overline{\Gamma_{i-1}}}, \nabla \bu^\star_{i} {|}_{\tilde{\mD}^\star_{i-1}  \setminus  \overline{\Gamma_{i-1}}}, \,\bt[\bu_{i}^{\emph{\sf{f}}}]{|}_{\Gamma_{i-1} \cap \exs \overline{\tilde{\mD}^\star_{i-1}}})~=~ \\*[0.2mm]
&\hspace{2.85cm} (\bu^\star_{i-1} \exs {|}_{\tilde{\mD}^\star_{i-1}  \!\setminus  \overline{\Gamma_{i-1}}}, \nabla \bu^\star_{i-1} {|}_{\tilde{\mD}^\star_{i-1}  \!\setminus  \overline{\Gamma_{i-1}}}, \, \bt[\bu_{i-1}^{\emph{\sf{f}}}]{|}_{{\Gamma_{i-1}} \cap\xxs \overline{\tilde{\mD}^\star_{i-1}}}),
\end{aligned}
\eeq
and consider the twelve generic configurations shown in Fig.~\ref{Evolution_sc} for microstructural evolution. The argument for similar or compound scenarios may be drawn from the following case studies. Keep in mind that the interior transmission problems of disconnected sets are independent. Based on this, $\tilde{\mD}_{i-1}^\star$, in each case, should be understood as the simply connected domain $\bx_\circ \in {\mD}_{{i-1},j}^\star \subset \tilde{\mD}_{{i-1}}^\star$ (or ${\mD}_{{i-1},j}^\star \supset L$) as defined earlier.   

\begin{figure}[bp!]
\vspace*{-4mm}
\begin{center} 
\includegraphics[width=0.85\linewidth]{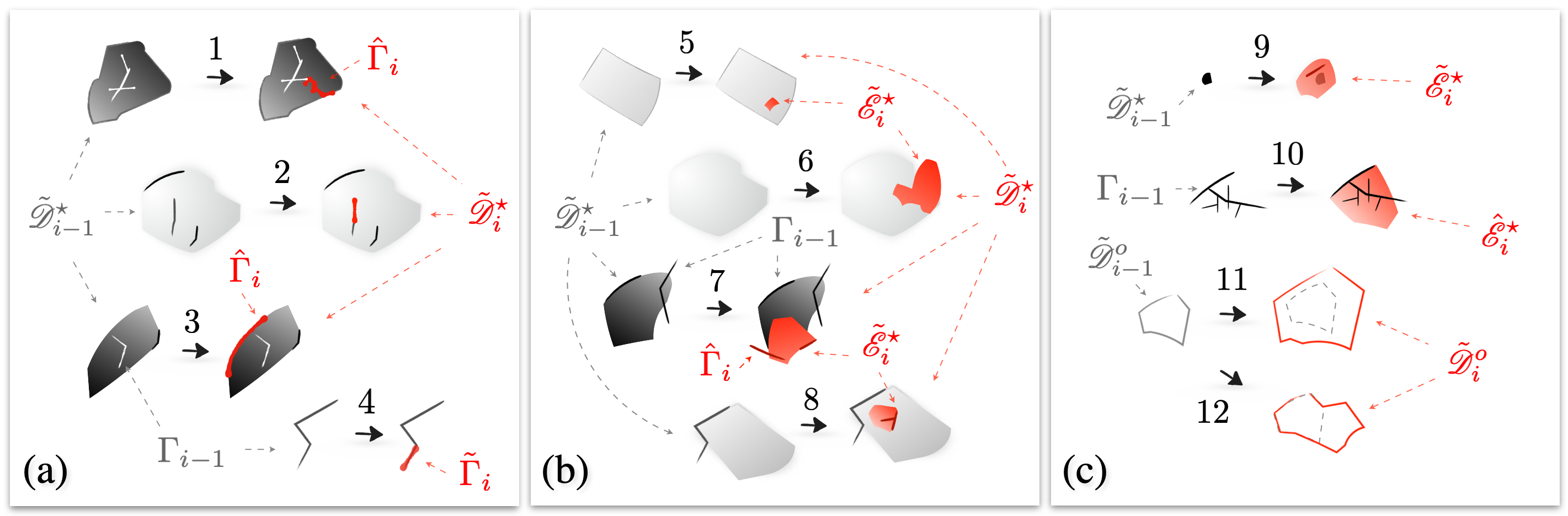}
\end{center} \vspace*{-4mm}
\caption{Twelve scenarios for microstructural transformation:~(a) geometric expansion and/or elastic modification of discontinuity surfaces in inclusions or the binder,~(b) new or modified (fractured) inclusions intersecting with previous (cracked) heterogeneities, and~(c) newborn inclusions masking the microcracked damage zones and expansion of cavities.}
\label{Evolution_sc}\vspace*{-5mm}
\end{figure}

{\emph{\bf{Case 1--3 (fracturing of inclusions)}}}.~With reference to Fig.~\ref{Evolution_sc}~(a), consider the case where $\bx_\circ$ (or $L$) is in $\tilde{\mD}_{i-1}^\star$ where evolution occurs either by new internal/boundary fractures $\hat{\Gamma}_i$ or by elastically modified interfaces $\tilde{\Gamma}_i$. Under the premise of~\eqref{CA1}, let us define $\boldsymbol{\sf w} = \bw^\star_{i}-\bw^\star_{i-1}$ in $\tilde{\mD}_{i-1}^\star \!\!\setminus\! \overline{\Gamma_i}$. On recalling~\eqref{ITPE}, observe that the Cauchy data of $\boldsymbol{\sf w}$ vanishes on $\partial \tilde{\mD}_{i-1}^\star \!\!\setminus\! \overline{\Gamma_i}$ which implies by the unique continuation principle that $\boldsymbol{\sf w} = \bzero$ in $\tilde{\mD}_{i-1}^\star$. In \underline{case 1} -- where $\tilde{\mD}_{i}^\star$ is endowed with internal $\hat{\Gamma}_i$ -- the contradiction arises from the discontinuity of $\bw^\star_{i}$ across $\hat{\Gamma}_i$ while $\bw^\star_{i-1}$ is continuous. The only exception to the latter, according to the fourth of~\eqref{ITPE}, is when $\bt[\bw^\star_{i}]=\bzero$ on $\hat{\Gamma}_i$ so that $\llbracket\bw^\star_{i} \rrbracket = \bzero$, which may not be the case per Assumption~\ref{Frac}. In \underline{case 2} -- where the contact's elasticity $\bK_s$ with $s \in \lbrace i-1, i \rbrace$ changes over $\tilde{\Gamma}_i$ i.e.,~$\bK_{i-1} \neq \bK_i$ -- vanishing $\boldsymbol{\sf w}$ in $\tilde{\mD}_{i-1}^\star$ implies $(\bK_{i}-\bK_{i-1})\llbracket\bw^\star_{i} \rrbracket = (\bK_{i}-\bK_{i-1})\llbracket\bw^\star_{i-1} \rrbracket =\bzero$ on $\tilde{\Gamma}_i$, by the fourth of~\eqref{ITPE}, which requires $\bt[\bw^\star_{i}] = \bt[\bw^\star_{i-1}]=\bzero$ over $\tilde{\Gamma}_i$ that contradicts Assumption~\ref{Frac}. In \underline{case 3} -- where $\hat{\Gamma}_i \subset \partial \tilde{\mD}_{i}^\star$ -- the contradiction may be observed from the fifth of~\eqref{ITPE} where $\boldsymbol{\sf w} = \bzero$ reads $\bt[\bw^\star_{i}] =\bzero$ on $\hat{\Gamma}_i$ with similar contradiction to Assumption~\ref{Frac}.     

{\emph{\bf{Case 4 (evolution of elastic contacts)}}}.~In this case where $L \subset \tilde{\Gamma}_i \!\setminus\! \overline{{\mD}_{i}^\star} \subset {\Gamma}_{i-1} \!\setminus\nxs \check{\Gamma}_{i-1}$, the fracture stiffness evolves within the matrix, as shown in Fig.~\ref{Evolution_sc}~(a), such that $\bK_{i}\neq\bK_{i-1}$ on $\tilde{\Gamma}_i \!\setminus\! \overline{{\mD}_{i}^\star}$. Then,~\eqref{Gevsc} is directly concluded from Theorem 4.5 and Theorem 4.7 of~\cite{pour2020}. This may also be observed from~\eqref{tufs}.

{\emph{\bf{Case 5 and 6 (volumetric growth or transformation of intact inclusions)}}}.~The premise, as depicted in~Fig.~\ref{Evolution_sc}~(b), is that $\overline{\tilde{\mD}_{i-1}^\star} \cap \Gamma_i = \overline{\tilde{\mD}_{i}^\star} \cap \Gamma_i = \emptyset$. In this case, the contradiction to~\eqref{CA1} may be argued similar to the proof of Theorem 4.2 in~\cite{AudiDLSM} which establishes~\eqref{Vevsc}. 

{\emph{\bf{Case 7--9 (elastic transformation or expansion of fractured inclusions)}}}.~Let us define $\tilde{\bw}^\star_{i}$ in $\tilde{\mD}_{i}^\star$ as the following: 
\beq\label{SWS}
\tilde{\bw}^\star_{i}(\bxi) ~\colon \!\!\! =~
\!\left\{\begin{array}{l}
\begin{aligned}
&\!\!  {\bw}^\star_{i-1}(\bxi),  & \bxi \exs \in  \tilde{\mD}_{i-1}^\star  \!\setminus\nxs  \overline{\Gamma_{i-1}}  \!\!\! \\*[0.0mm] 
& \!\! [{\bu}^\star_{i}  + \bPsi^{o}](\bxi),  & \bxi \exs \in  \tilde{\mD}_{i}^\star \!\setminus\! \overline{\tilde{\mD}_{i-1}^\star}  \!\!\! 
\end{aligned}
\end{array}\right.,
\eeq

Observe in light of~\eqref{ITP*} and~\eqref{CA1} that~$\tilde{\bw}^\star_{i}$ solves  
\[
\begin{aligned}
&\nabla \exs\sip\exs \tilde{\bC} \exs \colon \! \nabla\tilde{\bw}^\star_{i} \exs+\exs \tilde{\rho} \exs \omega^2 \tilde{\bw}^\star_{i} \,=\, \bzero \quad \text{in}\,\, \tilde{\mD}^\star_{i} \!\setminus\! \overline{\Gamma_{i-1}},
\end{aligned}
\]
wherein $(\tilde{\bC}, \tilde{\rho})$ is given by~\eqref{TCTr}. In this setting, one may show that the Cauchy data affiliated with $\tilde{\boldsymbol{\sf w}} \xxs=\xxs \tilde{\bw}^\star_{i}\nxs-\bw^\star_{i}$ vanish on $\partial \tilde{\mD}_{i}^\star \!\setminus\! \Gamma_i$ so that $(\bw^\star_{i},\tilde{\bw}^\star_{i})$ is the solution to  
$
\text{ITP}_\circ^{\xxs\text{\sf f}}(\tilde{\mD}_{i}^\star, \Gamma_{i-1}, \Gamma_i; \lbrace\tilde{\bC},\tilde{\rho}\rbrace, \lbrace\bC_i,\rho_i\rbrace, \bK_{i-1}, \bK_{i}).
$
To continue, let us consider two configurations: ({cases 7, 8}) where $\tilde{\mD}_{i-1}^\star$ is not included in any simply connected part $\tilde{\mE}_{i,j}^\star$ of $\tilde{\mE}_{i}^\star$, and ({case 9}) where $\tilde{\mD}_{i-1}^\star \subset \tilde{\mE}_{i,j}^\star$. Note that in \underline{cases 7, 8}, $\partial \tilde{\mD}_{i-1}^\star \cap \partial \tilde{\mD}_{i}^\star$ is of nonzero surface measure, then owing to the equality of Cauchy data associated with $\tilde{\bw}^\star_{i}$ and $\bw^\star_{i}$ and the fact that $(\bC_{i},\rho_{i}) = (\tilde{\bC},\tilde{\rho})=(\bC_{i-1},\rho_{i-1})$ on $\partial \tilde{\mD}_{i-1}^\star \cap \partial \tilde{\mD}_{i}^\star$ one may conclude that $\tilde{\bw}^\star_{i} = \bw^\star_{i}$ on $\tilde{\mD}_{i-1}^\star \!\setminus\!\tilde{\mE}_{i}^\star$. Consequently, $(\bw^\star_{i},\tilde{\bw}^\star_{i})|_{\tilde{\mE}_{i}^\star}$ is the solution to
$
\text{ITP}_\circ^{\xxs\text{\sf f}}(\tilde{\mE}_{i}^\star, \Gamma_{i-1}, \Gamma_i; \lbrace\tilde{\bC},\tilde{\rho}\rbrace, \lbrace\bC_i,\rho_i\rbrace, \bK_{i-1}, \bK_{i}).
$
The latter according to~Assumption~\ref{Inc*} implies that $\bw^\star_{i}=\tilde{\bw}^\star_{i}=\bzero$ in $\tilde{\mE}_{i}^\star$ which by unique continuation reads $\bw^\star_{i} = \bzero$ in $\tilde{\mD}_{i}^\star$. This requires $\bu^\star_{i} = -\bPsi^{o}$ in $\tilde{\mD}_{i-1}^\star$ which is a contradiction since $\bu^\star_{i}$ is smooth by definition while $\bPsi^{o}$ features a singularity at $\bx_\circ$. In \underline{case 9}, one may directly deduce from $\tilde{\mE}_{i}^\star \supset \tilde{\mE}_{i,j}^\star = \tilde{\mD}_{i}^\star$ and~Assumption~\ref{Inc*} that $\bw^\star_{i} = \bzero$ in $\tilde{\mD}_{i}^\star$ which leads to the same contradiction. 

{\emph{\bf{Case 10 (elastic transformation of microcracked damage zones)}}}.~With reference to~Fig.~\ref{Evolution_sc}~(c), consider the case where $L$ coincides with the binder's fractures at $t_{i-1}$ ($L \subset \Gamma_{i-1}$) within a neighborhood that undergoes elastic evolution at $t_i$ such that $L \subset \Gamma_{i-1} \cap \hat{\mE}_{i}^\star$. Contrary to~\eqref{Gevsc}, let
\[
\bt[\bu^\star_{i}]{|}_{{\Gamma}_{i-1} \cap \hat{\mE}_{i}^\star} =~\exs  \bt[\bu_{i-1}^{\emph{\sf{f}}}]{|}_{{\Gamma}_{i-1} \cap \hat{\mE}_{i}^\star}, 
\]  
wherein both free fields $\bu_{i-1}^{\emph{\sf{f}}}$ and $\bu^\star_{i}$ satisfy $\nabla \exs\sip\exs \bC \exs \colon \! \nabla(\exs\cdot\exs) + \rho \exs \omega^2 (\exs\cdot\exs) = \bzero$ in ${\hat{\mE}_{i}^\star}$. Then observe that the latter also governs ${\boldsymbol{\sf u}}^i = \bu^\star_{i} - \bu_{i-1}^{\emph{\sf{f}}}$ such that $\bt[{\boldsymbol{\sf u}}^i] = \bzero$ on ${\Gamma}_{i-1} \cap \hat{\mE}_{i}^\star$, implying per Assumption~\ref{Frac} at $t_{i-1}$ that $\bu^\star_{i} =\bu_{i-1}^{\emph{\sf{f}}}$ in ${\hat{\mE}_{i}^\star}$. Based on which, one may define $\hat{\bw}^\star_{i} = {\bu}^\star_{i}  + \bPsi^{1}$ in ${\hat{\mE}_{i}^\star}$ and similar to Cases 7--9 conclude that $(\bw^\star_{i},\hat{\bw}^\star_{i})$ is a solution to 
$
\text{ITP}_\circ^{\xxs\text{\sf f}}(\hat{\mE}_{i}^\star, \Gamma_{i-1}, \Gamma_i; \lbrace{\bC},{\rho}\rbrace, \lbrace\bC_i,\rho_i\rbrace, \bK_{i-1}, \bK_{i}),
$ 
which by Assumption~\ref{Inc*} reads $\hat{\bw}^\star_{i} = \bzero$ requiring that $\bu^\star_{i} = -\bPsi^{1}$ in ${\hat{\mE}_{i}^\star}$ which is a contradiction since $\bu^\star_{i}$ is smooth by definition while $\bPsi^{1}$ has a discontinuity across $L$.
  
{\emph{\bf{Case 11--12 (expansion of pores)}}}.~Consider the case shown in Fig.~\ref{Evolution_sc}~(c) where $\bx_\circ$ (or $L$) is in $\tilde{\mD}_{i-1}^{o}$ where the evolution of cavities occurs. In contrast to~\eqref{Hevsc}, let 

\[
\bt[\bu_{i}^{o}]{|}_{\partial \mD^o_{i-1}} =~\exs  \bt[\bu_{i-1}^{o}]{|}_{\partial \mD^o_{i-1}}, 
\] 
where $\bu_{i-1}^{o}$, $\bu^{o}_{i}$, and thus ${\boldsymbol{\sf u}}^o = \bu^o_{i} - \bu_{i-1}^0$ satisfy $\nabla \exs\sip\exs \bC \exs \colon \! \nabla(\exs\cdot\exs) + \rho \exs \omega^2 (\exs\cdot\exs) = \bzero$ in $\mD^o_{i-1}$. Note that $\bt[{\boldsymbol{\sf u}}^o] = \bzero$ on $\partial \mD^o_{i-1}$ implying per Assumption~\ref{Frac} that $\bu^o_{i} =\bu_{i-1}^o$ in $\mD^o_{i-1}$. Keep in mind that $\bu^o_{i-1}$ (\emph{resp}.~$\bu^o_{i}$) solves~\eqref{ITPo} in $\mD^o_{i-1}$ (\emph{resp}.~$\mD^o_{i}$). Then, observe that $\bu^o_{i} = -\bPsi^{o}$ solves~\eqref{ITPo} within $\mD^o_{i} \!\setminus\! \overline{\mD^o_{i-1}}$ provided that $\bu^o_{i} =\bu_{i-1}^o$ in $\mD^o_{i-1}$. In this setting, the continuity of $\bu^o_{i}$ across $\partial \bu_{i-1}^o \!\setminus\! \partial \bu_{i}^o$ along with the unique continuation principle requite that $\bu^o_{i} = -\bPsi^{o}$ in $\mD^o_{i-1}$ which is a contradiction since $\bu^o_{i}$ is smooth according to~\eqref{ITPo} while $\bPsi^{o}$ has either a singularity at $\bx_\circ$, or a discontinuity across $L$ per~\eqref{RHS}.  
\end{proof}

Theorem~\ref{Lxo} furnishes the main results required to (a) identify the invariants of scattering solutions according to Theorem~\ref{Inv1} which is directly obtained by drawing from~\cite[Theorem 4.3]{AudiDLSM} and \cite[Theorem 4.7]{pour2020}, and (b) establish the validity of differential evolution indicators in~\eqref{EIF} following~\cite[Corollary 1]{AudiDLSM}.

\begin{theorem}\lb{INV1}
Define 
\beq \lb{Inv1}
\begin{aligned}
\chi_{i}(\mathcal{G}_{i-1}, \mathcal{G}_{i}) &:= \big(\, \mathcal{G}_{i}\!-\mathcal{G}_{i-1},\, \Lambda_{{i-1}_\sharp}(\mathcal{G}_{i}\!-\mathcal{G}_{i-1})\big), \qquad \mathcal{G}_{i-1}, \mathcal{G}_{i}\in L^2(\OOd)^3,
\end{aligned}
\eeq
where $(\mathcal{G}_{i-1}, \mathcal{G}_{i})(\bPsi^{o};\gamma)$ are the constructed minimizers of $(\mathfrak{J}^{\gamma}_{i-1}, \mathfrak{J}^{\gamma}_{i})$ in~\eqref{GCf} according to~\eqref{mseq1}. Then, in light of factorization~\eqref{fact} and~\eqref{oH}, Theorem~\ref{Lxo} reads 
\[
\hspace*{-2cm} \bullet\,\,\,\,\text{If}\,\,\, \begin{array}{|ll}
\begin{aligned}
&\!\nxs  L\subset\check{\mD}_{i-1}^\star \cup \check{\mD}_{i-1}^o \cup \check{\Gamma}_{i-1}  \!\!\! \\*[0.0mm] 
& \!\nxs \bx_\circ\in\check{\mD}_{i-1}^\star \cup \check{\mD}_{i-1}^o  \!\!\! 
\end{aligned}
\end{array}.\!\!, \, \text{then} \lim\limits_{\gamma \rightarrow 0}\chi_{i}[\mathcal{G}_{i-1}, \mathcal{G}_{i}](\bPsi^{o};\gamma) = 0.
\]
\[
\hspace*{-0.45cm}  \bullet\,\,\,\, \text{If}\,\,\, \begin{array}{|ll}
\begin{aligned}
&\!\nxs  L\subset  \tilde{\mD}_{i-1}^\star \cup \tilde{\mD}_{i-1}^o \cup {\Gamma}_{i-1} \!\setminus\! \check{\Gamma}_{i-1}  \!\!\! \\*[0.0mm] 
& \!\nxs \bx_\circ\in\tilde{\mD}_{i-1}^\star \cup \tilde{\mD}_{i-1}^o  \!\!\! 
\end{aligned}
\end{array}.\!\!, \, \text{then} \,\, 0 < \lim\limits_{\gamma \rightarrow 0}\chi_{i}[\mathcal{G}_{i-1}, \mathcal{G}_{i}](\bPsi^{o};\gamma) < \infty.
\]
\[
 \bullet\,\,\,\, \text{If}\,\,\, \begin{array}{|ll}
\begin{aligned}
&\!\nxs  L\subset\mE^\star_i \!\setminus\! \overline{\tilde{\mD}_{i-1}^\star} \cup \mE^o_i  \!\setminus\!  \overline{\tilde{\mD}_{i-1}^o} \cup \hat{\Gamma}_{i} \!\setminus\! \overline{\tilde{\mD}_{i-1}^\star}  \!\!\! \\*[0.0mm] 
& \!\nxs \bx_\circ\in \mE^\star_i \!\setminus\! \overline{\tilde{\mD}_{i-1}^\star} \cup \mE^o_i  \!\setminus\!  \overline{\tilde{\mD}_{i-1}^o}  \!\!\! 
\end{aligned}
\end{array}.\!\!, \, \text{then} \, \lim\limits_{\gamma \rightarrow 0}\chi_{i}[\mathcal{G}_{i-1}, \mathcal{G}_{i}](\bPsi^{o};\gamma) = \infty.
\]
\end{theorem}

\paragraph*{Differential evolution indicators.}
Let us introduce the imaging functionals $\mathfrak{D}_i: L^2(\Omega^3)\times L^2(\Omega^3) \rightarrow \mathbb{R}$ and $\tilde{\mathfrak{D}}_i: L^2(\Omega^3)\times L^2(\Omega^3) \rightarrow \mathbb{R}$ such that given $\Upsilon_{i}(\mathcal{G}_{i}) :=  (\mathcal{G}_{i}, \Lambda_{{i}_\sharp} \mathcal{G}_{i})$, 
\beq\lb{EIF}
\begin{aligned}
& \mathfrak{D}_i(\mathcal{G}_{i-1}, \mathcal{G}_{i}) := \frac{1}{\sqrt{\Upsilon_{i}(\mathcal{G}_{i}) \big{[}1+ \Upsilon_{i}(\mathcal{G}_{i}) \chi_{i}^{-1}(\mathcal{G}_{i-1}, \mathcal{G}_{i})\big{]}}},  \\*[0.25mm]
&\tilde{\mathfrak{D}}_i(\mathcal{G}_{i-1}, \mathcal{G}_{i}) := \frac{1}{\sqrt{\Upsilon_{i-1}(\mathcal{G}_{i-1}) + \Upsilon_{i}(\mathcal{G}_{i}) \big{[}1+ \Upsilon_{i-1}(\mathcal{G}_{i-1})\chi_{i}^{-1}(\mathcal{G}_{i-1}, \mathcal{G}_{i})\big{]}}}.
\end{aligned}
\eeq
Then, it follows that
\[
\bullet\,\,\,\,\,\, \begin{array}{|ll}
\begin{aligned}
&\!\nxs  L \subset \mE^\star_i \cup \tilde{\mD}_{i-1}^\star \cup \mE^o_i \cup \tilde{\mD}_{i-1}^o \cup \hat\Gamma_i \cup {\Gamma}_{i-1} \!\!\setminus\! \check{\Gamma}_{i-1}  \!\!\! \\*[0.0mm] 
& \!\nxs \bx_\circ\in \mE^\star_i \cup \tilde{\mD}_{i-1}^\star \cup \mE^o_i \cup \tilde{\mD}_{i-1}^o  \!\!\! 
\end{aligned}
\end{array}\!\! \,\,\, \iff \,\, \lim\limits_{\gamma \rightarrow 0} \mathfrak{D}_i(\mathcal{G}_{i-1}, \mathcal{G}_{i})(\bPsi^{o};\gamma) > 0.
\]
\[
\hspace*{-2.2cm}\bullet\,\,\,\,\,\, \begin{array}{|ll}
\begin{aligned}
&\!\nxs  L\subset  \tilde{\mD}_{i-1}^\star \cup \tilde{\mD}_{i-1}^o \cup {\Gamma}_{i-1} \!\!\setminus\! \check{\Gamma}_{i-1}  \!\!\! \\*[0.0mm] 
& \!\nxs \bx_\circ\in\tilde{\mD}_{i-1}^\star \cup \tilde{\mD}_{i-1}^o  \!\!\! 
\end{aligned}
\end{array}\!\! \,\,\, \iff \,\, \lim\limits_{\gamma \rightarrow 0} \tilde{\mathfrak{D}}_i(\mathcal{G}_{i-1}, \mathcal{G}_{i})(\bPsi^{o};\gamma) > 0.
\]

In other words, ${\mathfrak{D}}_i$ (\emph{resp.}~$\tilde{\mathfrak{D}}_i$) assumes near-zero values except at the loci of $\tilde{\mD}_{i-1}^\star \cup \tilde{\mD}_{i-1}^o \cup {\Gamma}_{i-1} \!\!\setminus\! \check{\Gamma}_{i-1}$ (\emph{resp.}~$\tilde{\mD}_{i-1}^\star \cup \tilde{\mD}_{i-1}^o \cup {\Gamma}_{i-1} \!\!\setminus\! \check{\Gamma}_{i-1}$) where the indicator increases and remains finite as $\gamma \to 0$. By building on~\cite[Theorem 4.4]{AudiDLSM} and~\cite[Theorem 4.8]{pour2020}, it is quite straightforward to formulate pertinent results for noisy data which for brevity are not included in this paper. To summarize, Fig.~\ref{flowch} provides the steps for the construction of evolution indicators from numerical or laboratory test data.

\begin{figure}[tp!]
\begin{center} 
\includegraphics[width=1\linewidth]{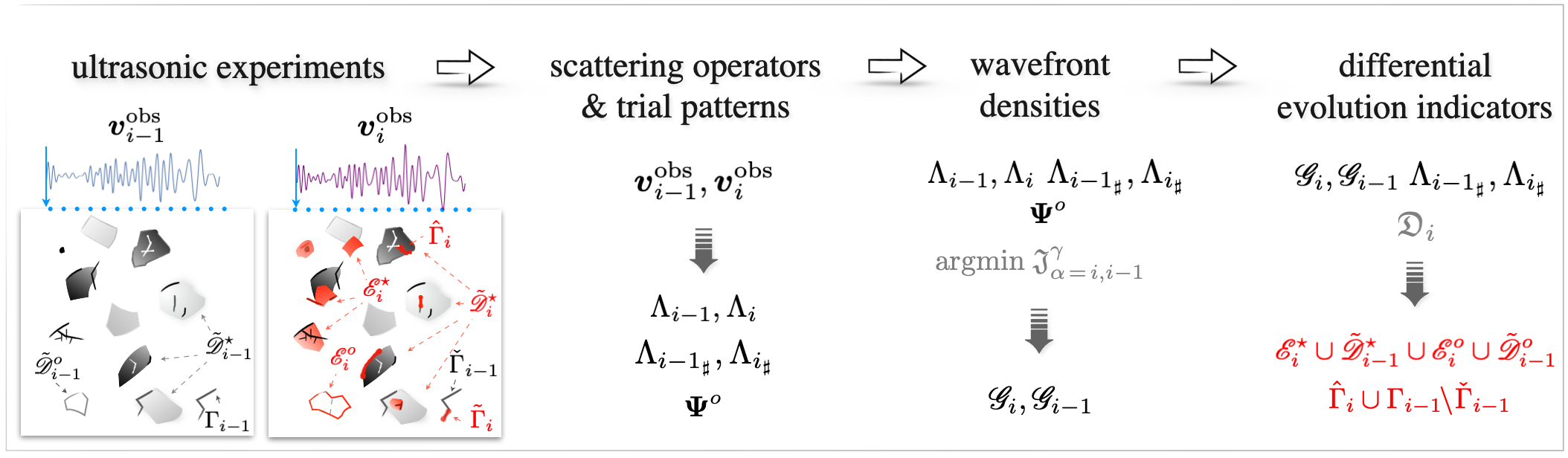}
\end{center} \vspace*{-6mm}
\caption{Differential imaging of evolution from sequential waveform data.}
\label{flowch}
\end{figure}

\section{Synthetic experiments}\label{IR}

The evolution indicators of~\eqref{EIF} are put to test in this section by a set of numerical experiments. The primary focus is on a randomly heterogeneous and discontinuous background with evolving microstructure due to elastic transformation and/or fracturing. The special cases of \emph{monolithic} solids endowed with crack or pore networks are reported in~\cite{pour2020}. In this section, the synthetic scattered fields $\bv\obs_i$, $i = \lbrace \circ,1,2,3 \rbrace$, are simulated via the boundary element method~\cite{Bon1999}, see~\cite{Fate2017} for more on the computational platform.             

With reference to Fig.~\ref{config}, the testing configuration at $t_\circ$ i.e.,~the background domain entails a composite slab of dimensions $2.5$ $\!\times\!$ $2.5$ $\!\times\!$ $0.01$ comprised of an elastic binder endowed with ellipsoidal inclusions of arbitrary distribution and size. The in-plane diameters of scatterers ranges from one to five shear wavelengths $\lambda_s = 0.04$, while their pairwise distances are greater than $2\lambda_s$. The normalized shear modulus, mass density, and Poisson's ratio of the matrix are taken as $\mu_m = 1$, $\rho_m = 1$ and $\nu_m = 0.25$, while that of the scatterers are~$\mu_s = 2$, $\rho_s = \rho_m$ and $\nu_s = \nu_m$. In this setting, the shear and compressional wave speeds in the matrix are $c^s_m = 1$ and $c^p_m = 1.73$. The specimen's microstructural evolution in the following time steps $t_1-t_3$, according to Fig.~\ref{config}, involves (a) multi-step fracturing of the binder and inclusions and their coalescence, (b) elastic transformation of pre-existing inclusions at $t_\circ$, and (c) emergence of new volumetric heterogeneities with shear modulus $\mu^{\text{new}}_s = 1.5$. Note the gradual increase in the evolution complexity, and in particular, the density of scatterers such that at $t_3$: (a) the pairwise distance between scatterers may reduce to a small fraction of $\lambda_s$, and (b) a subset of evolution support is deeply embedded within the stationary scatterers.   

Synthetic experiments are conducted at four time steps $t_\circ-t_3$ when the specimen assumes the geometric configurations shown in Fig.~\ref{config}. Every sensing step entails 2000 forward simulations where in-plane harmonic waves of frequency $\omega = 140$ rad/s are generated at a point source over the specimen's external boundary. The resulting scattered fields $\bv\obs_i$, $i = \lbrace \circ,1,2,3 \rbrace$, are then calculated on the same grid by solving the 3D elastodynamics boundary integral equations. Given that $\lambda_s$ is four times greater than the specimen thickness, the leading contributions to scattered fields are the in-plane components which are then used for data inversion.

The obtained scattered signatures are used to compute the synthetic wavefront densities $\mathcal{G}_i$, $i = \lbrace \circ,1,2,3 \rbrace$, as approximate minimizers of the cost functionals in~\eqref{GCf}. The latter follows the common three steps required for constructing any sampling-based indicator, namely:~(1) forming the discrete scattering operators $\Lambda_{{i}_\sharp}$ at every $t_i$, (2) assembling the trial signatures $\bPsi^o$ of~\eqref{RHS} as the right-hand side of the scattering equation, and (3) solving the latter by minimizing the discretized cost function~\eqref{GCf} by invoking the Morozov discrepancy principle. A detailed account of this process is provided in~\cite{pour2020}. Given $\Lambda_{{i}_\sharp}$ and $\mathcal{G}_i$, one may then evaluate the imaging functionals $\mathfrak{D}_i$ from~\eqref{EIF} which is expected to achieve its highest values at the loci of new and evolved scatterers in each sensing sequence.   

\begin{figure}[tp!]
\vspace*{-5mm}
\begin{center} 
\includegraphics[width=0.95\linewidth]{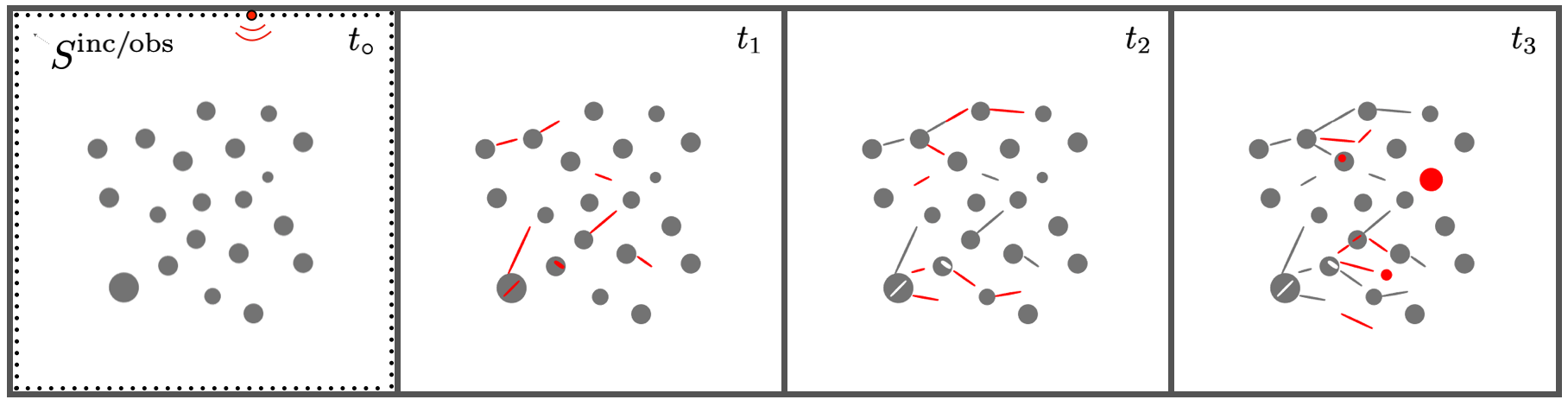}
\end{center} \vspace*{-6.5mm}
\caption{\small{Microstructural geometry of a composite slab with evolving heterogeneities and discontinuities at four sensing steps $t_\circ \!- t_3$. Elastic (in-plane) waves are periodically generated via boundary excitations on $S^\text{inc}$, and the affiliated scattered waveforms are computed over the observation surface $S^\text{obs}$. Here, $S^\text{inc} = S^\text{obs}$ is sampled at $10^3$ points around the model perimeter.}}
\label{config}\vspace*{0mm}
\end{figure}

\begin{figure}[h!]
\vspace*{-0mm}
\begin{center} 
\includegraphics[width=0.65\linewidth]{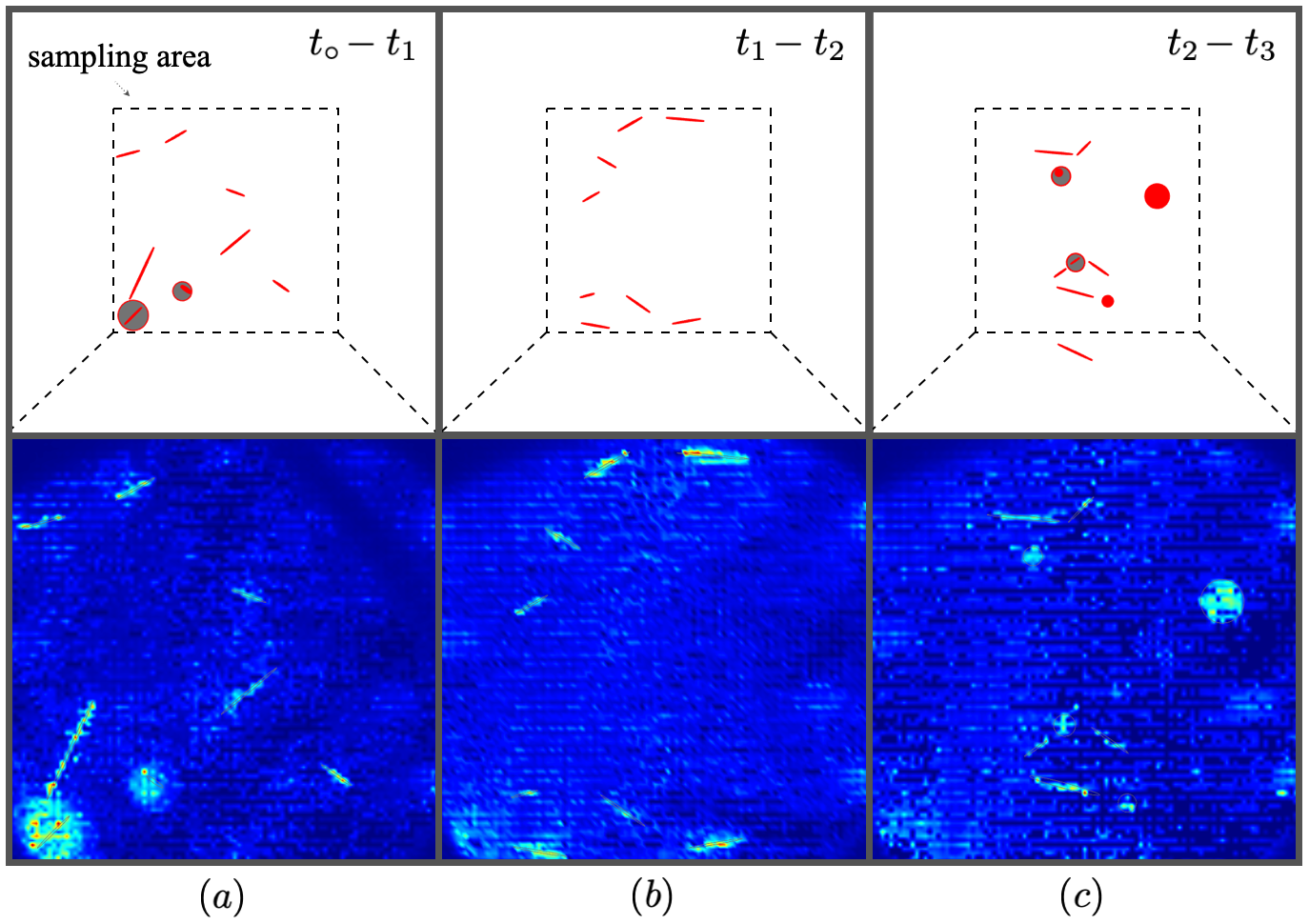}
\end{center} \vspace*{-6.5mm}
\caption{\small{Three-step reconstruction of elastic and interfacial transformations (a-c) of the initial configuration shown in Fig.~\ref{config}~(a):~(top) evolution geometry in the sensing sequence $[t_{i-1} \,\, t_i]$, $i = \lbrace 1, 2, 3 \rbrace$, and~(bottom)~the affiliated indicator map $\mathfrak{D}_i$ of~\eqref{EIF} computed from the observed scattered field data $\bv\obs_{i-1}$ and $\bv\obs_{i}$.}}
\label{recons}\vspace*{-0mm}
\end{figure}

\begin{figure}[h!]
\vspace*{-0mm}
\begin{center} 
\includegraphics[width=0.95\linewidth]{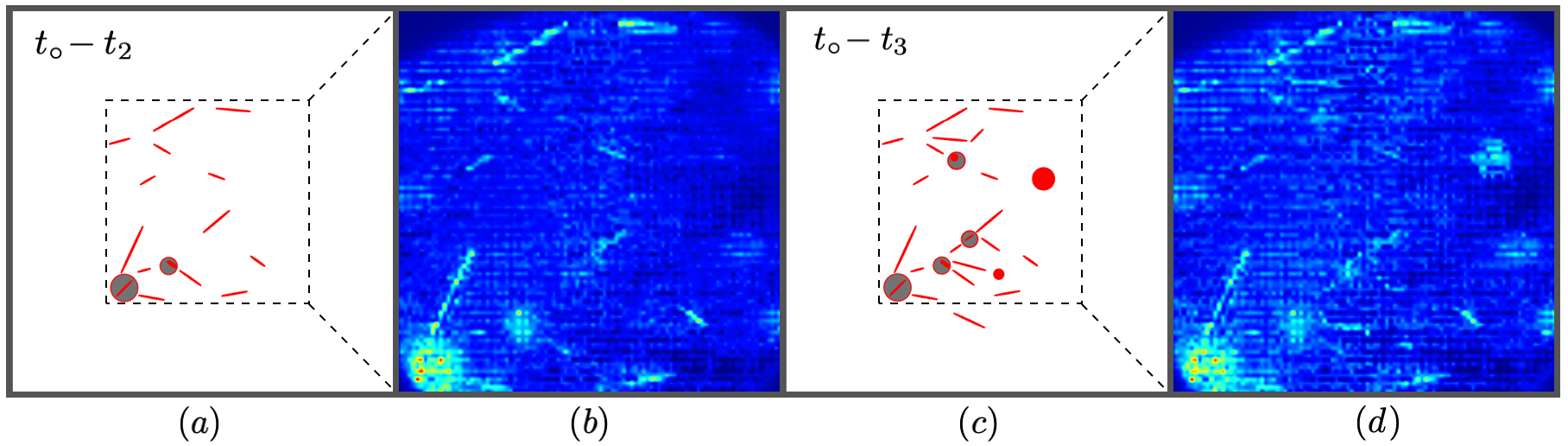}
\end{center} \vspace*{-6.5mm}
\caption{\small{Multi-step reconstruction of elastic and geometric variations:~(a), (c)~true evolution support between $[t_{\circ} \,\, t_2]$ and $[t_{\circ} \,\, t_3]$ respectively, and (b), (d) the associated (superimposed) differential maps.}}
\label{compound}\vspace*{-6mm}
\end{figure} 

Fig.~\ref{recons} illustrates the successive evolution indicators $\mathfrak{D}_j$, $j = \lbrace 1,2,3 \rbrace$, over the sampling area. For each time window $t_{j-1}\nxs- t_{j}$, the "true" support of elastic variations is provided in the top row. Keep in mind that in any sequence, $\mathfrak{D}_j$ is by design insensitive to the scatterers at $t_{j-1}$ provided that they remain unchanged by $t_{j}$. The differential maps within $t_\circ \nxs- t_2$ feature relatively sharp localizations with minimal artifacts which, given the shear wavelength, may be attributed to the rather sparse distribution of scatterers in this timeframe. Some artifacts emerge in the $\mathfrak{D}_3$ map, however, since the elastic variations between $t_2 \nxs- t_3$ occur within a densely packed network of fractures and inclusions which are assumed to be unknown. To recover the evolution support over an extended timeframe as the microstructure becomes progressively complex, one may superimpose the consecutive reconstructions (of Fig.~\ref{recons}) as in Fig.~\ref{compound}.

\section{Conclusion}

This study furnishes the theoretical foundation for differential evolution indicators for ultrasonic imaging of elastic variations within a heterogeneous and discontinuous background of random structure. In this vein, the wellposedness conditions for the forward and inverse scattering problems are established; in light of which, the pairwise relation between scattering solutions -- associated with distinct datasets, is determined. For this purpose, twelve scenarios for microstructural transitions are investigated including (a) fracturing of inclusions, (b) evolution of discontinuity surfaces within each material component or at bimaterial interfaces, (c) elastic transformation and/or expansion of fractured inclusions, (d) conversion of microcracked damage zones, and (e) expansion of pores. In all cases, it is shown that certain measures of the synthetic incident fields constructed based on the scattering solutions are, in the limit, equivalent at the loci of unknown scatterers which remain both geometrically and mechanically invariant between a given pair of time steps. This allows for exclusive reconstruction of evolved features without the knowledge (or need for recovery) of stationary components in the background. This is particularly useful in uncertain environments as showcased by the synthetic experiments -- provided that the illuminating wavelength is sufficiently smaller than the relevant microstructural length scales e.g., pairwise distance between the scatterers. Relaxing such constraints may be possible through time-domain inversion as a potential direction in future studies.               

\section{Acknowledgements}
The corresponding author kindly acknowledges the support provided by the National Science Foundation (Grant No.~1944812). This work utilized resources from the University of Colorado Boulder Research Computing Group, which is supported by the National Science Foundation (awards ACI-1532235 and ACI-1532236), the University of Colorado Boulder, and Colorado State University.

\appendix

\section{Wellposedness of the direct scattering problem}


Observe that $\forall\bv' \in H^1({\mathcal{B}}_\kappa^{-}\! \cup\nxs \mD^\star_\kappa \!\setminus\!  \overline{\Gamma_{\kappa}})^3$, the variational form of~\eqref{vk} reads  
\beq\lb{Wik-GE}
\begin{aligned}
  \int_{{\mathcal{B}}_\kappa^-} & \big[   \nabla \exs \bar{\bv}' \colon\!\xxs \bC \exs\colon\! \nabla \bv^\kappa \exs-\,  \rho \exs  \omega^2 \exs \bar{\bv}' \xxs\sip\exs \bv^\kappa  \big] \, \textrm{d}V_{\bxi} ~+\, \int_{\Gamma_\kappa} \llbracket \bar{\bv}' \rrbracket \xxs\sip\xxs \bK_{\nxs\kappa} \xxs \llbracket \bv^\kappa \rrbracket  \, \textrm{d}S_{\bxi} \,~+\, \\*[0.2mm]  
+\,\, & \int_{\mD^\star_\kappa  \setminus  \overline{\Gamma_{\kappa}}\exs} \big[  \nabla \exs \bar{\bv}' \colon\nxs \bC_\kappa \exs\colon\! \nabla \bv^\kappa \exs-\, \rho_\kappa \omega^2 \exs \bar{\bv}' \xxs\sip\exs {\bv}^\kappa  \big] \, \textrm{d}V_{\bxi}   \,~=~\,  \int_{\partial \mD^o_\kappa\!} \nxs \bar{\bv}' \xxs\sip\exs\xxs \bt^{\textrm{f}} \, \textrm{d}S_{\bxi} \,~+\,  \\*[0.2mm]  
 +\,\, & \int_{\mD^\star_\kappa  \setminus  \overline{\Gamma_{\kappa}}\exs} \big[ (\rho_\kappa - \rho)\omega^2 \exs \bar{\bv}' \xxs\sip\exs \bu^{\textrm{f}}  \,-\,  \nabla \exs \bar{\bv}' \colon\nxs (\bC_\kappa \nxs- \bC\xxs) \exs\colon\! \nabla \bu^{\textrm{f}} \big] \, \textrm{d}V_{\bxi}   ~+\, \int_{\Gamma_\kappa} \llbracket \bar{\bv}' \rrbracket \xxs\sip\exs\xxs \bt^{\textrm{f}}  \, \textrm{d}S_{\bxi} .   
\end{aligned}
\eeq

The sesquilinear form on the left-hand side of~\eqref{Wik-GE} may be decomposed as\\ $A(\bv^\kappa, \bv') + B(\bv^\kappa, \bv')$ where
\beq\lb{AB-GE}
\begin{aligned}
 & A(\bv^\kappa, \bv')  ~=~  \int_{\mD^\star_\kappa  \setminus  \overline{\Gamma_{\kappa}}\exs} \big[  \nabla \exs \bar{\bv}' \colon\nxs \bC_\kappa \exs\colon\! \nabla \bv^\kappa +\exs \bar{\bv}' \xxs\sip\exs {\bv}^\kappa  \big] \, \textrm{d}V_{\bxi}  ~+\, \int_{{\mathcal{B}}_\kappa^-}  \big[   \nabla \exs \bar{\bv}' \colon\!\xxs \bC \exs\colon\! \nabla \bv^\kappa +\, \bar{\bv}' \xxs\sip\exs \bv^\kappa  \big] \, \textrm{d}V_{\bxi}, \\*[0.2mm]  
& B(\bv^\kappa, \bv')  ~=~ -(1 + \rho_\kappa \omega^2) \int_{\mD^\star_\kappa  \setminus  \overline{\Gamma_{\kappa}}} \bar{\bv}' \xxs\sip\exs {\bv}^\kappa   \, \textrm{d}V_{\bxi} ~-\, (1 +  \rho\xxs  \omega^2) \int_{{\mathcal{B}}_\kappa^-} \bar{\bv}' \xxs\sip\exs \bv^\kappa  \, \textrm{d}V_{\bxi} \,~+\,\\*[0.2mm]  
& \hspace{33mm} \,+\,  \int_{\Gamma_\kappa} \llbracket \bar{\bv}' \rrbracket \xxs\sip\xxs \bK_{\nxs\kappa} \llbracket \bv^\kappa \rrbracket  \, \textrm{d}S_{\bxi}, \qquad \forall\bv' \in H^1({\mathcal{B}}_\kappa^{-}\! \cup\nxs \mD^\star_\kappa \!\setminus\!  \overline{\Gamma_{\kappa}})^3.
\end{aligned}
\eeq

In light of the Korn inequality~\cite{Fate2017}, observe that $A(\bv^\kappa, \bv')$ is coercive. Moreover, the antilinear form $B(\bv^\kappa, \bv')$ is compact by the application of Cauchy--Schwarz inequality to $|B(\bv^\kappa, \bv')|$, the compact embedding of $H^1({\mathcal{B}}_\kappa^{-}\! \cup\nxs \mD^\star_\kappa \!\setminus\!  \overline{\Gamma_{\kappa}})^3$ into $L^2({\mathcal{B}}_\kappa^{-}\! \cup\nxs \mD^\star_\kappa \!\setminus\!  \overline{\Gamma_{\kappa}})^3$, and the compactness of the trace operator $\bv^\kappa \to \llbracket \bv^\kappa \rrbracket$ as a map from $H^1({\mathcal{B}}_\kappa^{-}\! \cup\nxs \mD^\star_\kappa \!\setminus\!  \overline{\Gamma_{\kappa}})^3$ to $L^2(\Gamma_\kappa)^3$ owing to the compact embedding of $\tilde{H}^{{1}/{2}}(\Gamma_\kappa)^3$ into $L^2(\Gamma_\kappa)^3$. As a result,~\eqref{vk} is of Fredholm type, and thus, is wellposed as soon as the uniqueness of a solution is guaranteed. Let $(\bu^{\textrm{f}}{|}_{\mD^\star_\kappa  \setminus  \overline{\Gamma_{\kappa}}}, \nabla \bu^{\textrm{f}}{|}_{\mD^\star_\kappa  \setminus  \overline{\Gamma_{\kappa}}}, \exs\bt[\bu^{\textrm{f}}]{|}_{\Gamma_\kappa}, \bt[\bu^{\textrm{f}}]{|}_{\partial \mD^o_\kappa}) = \bzero$, then on setting ${\bv}'  = \bv^\kappa$, observe from~\eqref{Wik-GE} that  
\[
\Im(\int_{\mD^\star_\kappa  \setminus  \overline{\Gamma_{\kappa}}\exs}   \nabla \exs \bar{\bv}^\kappa \colon\nxs \bC_\kappa \exs\colon\! \nabla \bv^\kappa \, \textrm{d}V_{\bxi} \,~+\, \int_{\Gamma_\kappa} \llbracket \bar{\bv}^\kappa \rrbracket \xxs\sip\xxs \bK_{\nxs\kappa} \llbracket \bv^\kappa \rrbracket  \, \textrm{d}S_{\bxi}) ~=~  0,
\]
implying that $\llbracket \bv^\kappa \rrbracket = \bzero$ on ${\Gamma_\kappa}$, and $\bv^\kappa \nxs=\xxs \bzero$ in ${\mD^\star_\kappa  \!\setminus\!  \overline{\Gamma_{\kappa}}\exs}$ owing to Assumption~\ref{CP} and the first of~\eqref{vk}. Note that the jump in $\bv^\kappa\nxs$ vanishes not only on the intersection ${\Gamma_\kappa} \cap \partial \mD^\star_\kappa$, but also on the perfectly continuous interface $\partial \mD^\star_\kappa \!\setminus\! \overline{\Gamma_\kappa}$ according to the fourth of~\eqref{vk}. Thus, the Holmgren's theorem implies that the scattered field $\bv^\kappa\nxs$ vanishes in an open neighborhood of $\partial \mD^\star_\kappa$ which by virtue of the unique continuation theorem leads to $\bv^\kappa(\bxi) = \bzero$ in $\bxi \in {\mathcal{B}} \exs \backslash \overline{\mD^o_\kappa} $. This completes the proof for the uniqueness of a scattering solution in ${\mathcal{B}}_\kappa^{-}\! \cup\nxs \mD^\star_\kappa \!\setminus\!  \overline{\Gamma_{\kappa}}$, and thus, substantiates the wellposedness of the forward problem.

\section{Proof of Lemma~\ref{T-FI0}}\label{SFS}

Let $\hat{\boldsymbol{\sf v}}$ (\emph{resp}.~${\boldsymbol{\sf v}}'$) satisfies~\eqref{vk} for 
\[
\boldsymbol{X} \,=\, (\hat{\boldsymbol{\sf u}} {|}_{\mD^\star_\kappa  \setminus  \overline{\Gamma_{\kappa}}},\nxs \nabla \hat{\boldsymbol{\sf u}} {|}_{\mD^\star_\kappa  \setminus  \overline{\Gamma_{\kappa}}}, \bt[\hat{\boldsymbol{\sf u}}]{|}_{\Gamma_\kappa \cap \exs\overline{\mD^\star_\kappa}} \xxs\oplus\xxs {\boldsymbol{\psi}}, {\boldsymbol{\phi}}), \,\, ({\boldsymbol{\psi}}, {\boldsymbol{\phi}}) \in  H^{-1/2}(\Gamma_\kappa \nxs\!\setminus\! \overline{\mD^\star_\kappa})^3) \nxs\times\nxs {H^{-1/2}(\partial \mD^o_\kappa)^3}, 
\]
\big(\emph{resp.}
\[
\boldsymbol{X}' \,=\, ({\boldsymbol{\sf u}'} {|}_{\mD^\star_\kappa  \setminus  \overline{\Gamma_{\kappa}}},\nxs \nabla {\boldsymbol{\sf u}'} {|}_{\mD^\star_\kappa  \setminus  \overline{\Gamma_{\kappa}}}, \bt[{\boldsymbol{\sf u}'}]{|}_{\Gamma_\kappa \cap \exs\overline{\mD^\star_\kappa}} \xxs\oplus\xxs {\boldsymbol{\psi}'}, {\boldsymbol{\phi}'}), \,\, ({\boldsymbol{\psi}'}, {\boldsymbol{\phi}'}) \in  H^{-1/2}(\Gamma_\kappa \nxs\!\setminus\! \overline{\mD^\star_\kappa})^3) \nxs\times\nxs {H^{-1/2}(\partial \mD^o_\kappa)^3}\big), 
\]
wherein $\nabla \exs\sip\exs \bC \exs \colon \! \nabla \hat{\boldsymbol{\sf u}} \,+ \rho \exs \omega^2 \hat{\boldsymbol{\sf u}} ~=~ \bzero$ (\emph{resp}.~$\nabla \exs\sip\exs \bC \exs \colon \! \nabla {\boldsymbol{\sf u}'} \,+ \rho \exs \omega^2 {\boldsymbol{\sf u}'} ~=~ \bzero$) in $\mD^\star_\kappa$, then 
\beq\lb{SeQ4}
\begin{aligned}
&\langle \xxs{T_\kappa{\boldsymbol{X}},{\boldsymbol{X}'}} \rangle \,=\, -\int_{\mD^\star_\kappa  \setminus  \overline{\Gamma_{\kappa}}} \big[\xxs \nabla \bar{\boldsymbol{\sf u}}' \colon \! (\bC_{\kappa\nxs}-\bC) \xxs\colon \! \nabla (\hat{\boldsymbol{\sf u}} \xxs+\xxs\hat{\boldsymbol{\sf v}}) \xxs + \exs \omega^2(\rho -\rho_\kappa) \xxs\bar{\boldsymbol{\sf u}}' \!\cdot\nxs (\hat{\boldsymbol{\sf u}} \xxs+\xxs\hat{\boldsymbol{\sf v}})  \xxs \big]  \textrm{d}V ~+\,  \\*[0.2mm]   
&   \int_{\Gamma_\kappa \cap \exs\overline{\mD^\star_\kappa}} \bar{\bt}[\xxs{\boldsymbol{\sf u}}'] \nxs\cdot \llbracket \xxs\hat{\boldsymbol{\sf v}}\xxs \rrbracket  \, \textrm{d}S  \,+\, \int_{\Gamma_\kappa \nxs\!\setminus \overline{\mD^\star_\kappa}} \bar{\boldsymbol{\psi}}' \nxs\cdot \llbracket \xxs\hat{\boldsymbol{\sf v}}\xxs \rrbracket \, \textrm{d}S \,+\, \int_{\partial \mD^o_\kappa} \bar{\boldsymbol{\phi}}' \nxs \cdot (\hat{\boldsymbol{\sf u}}_{\boldsymbol{\phi}} \xxs+\xxs\hat{\boldsymbol{\sf v}}) \, \textrm{d}S,
\end{aligned}
\eeq 
where $\hat{\boldsymbol{\sf u}}_{\boldsymbol{\phi}}$ satisfies~\eqref{uphi} with ${\boldsymbol{\phi}}^n = {\boldsymbol{\phi}}$.
In addition, the variational form~\eqref{Wik-GE} with $\bv^\kappa=\hat{\boldsymbol{\sf v}}$ and $\bv' ={\boldsymbol{\sf v}}'$ reads
\beq\lb{Wik-GE*}
\begin{aligned}
& \int_{\mD^\star_\kappa  \setminus  \overline{\Gamma_{\kappa}}} \big[\xxs \nabla \bar{{\boldsymbol{\sf v}}}' \colon \! (\bC_{\kappa\nxs}-\bC) \xxs\colon \! \nabla (\hat{\boldsymbol{\sf u}} \xxs+\xxs\hat{\boldsymbol{\sf v}}) \xxs + \exs \omega^2(\rho -\rho_\kappa)\xxs \bar{{\boldsymbol{\sf v}}}' \!\cdot\nxs (\hat{\boldsymbol{\sf u}} \xxs+\xxs\hat{\boldsymbol{\sf v}})  \xxs \big]  \textrm{d}V   ~-\, \\*[0.2mm]  
-&   \int_{\Gamma_\kappa \cap \exs\overline{\mD^\star_\kappa}} {\bt}[\xxs \hat{\boldsymbol{\sf u}} \xxs] \cdot \llbracket \exs \bar{{\boldsymbol{\sf v}}}' \rrbracket  \, \textrm{d}S  \,-\, \int_{\Gamma_\kappa \nxs\!\setminus \overline{\mD^\star_\kappa}} \boldsymbol{\psi} \nxs\cdot \llbracket \exs\bar{{\boldsymbol{\sf v}}}' \rrbracket \, \textrm{d}S \,-\, \int_{\partial \mD^o_\kappa} \boldsymbol{\phi} \nxs \cdot \bar{{\boldsymbol{\sf v}}}' \, \textrm{d}S ~=\, \\*[0.2mm]  
 -& \int_{{\mathcal{B}}_\kappa^- \cup\xxs  \mD^\star_\kappa\nxs  \setminus  \overline{\Gamma_{\kappa}}}  \big[   \nabla \exs \bar{{\boldsymbol{\sf v}}}' \colon\!\xxs \bC \exs\colon\! \nabla \hat{\boldsymbol{\sf v}} \exs-\,  \rho \xxs  \omega^2 \exs \bar{{\boldsymbol{\sf v}}}' \sip\exs \hat{\boldsymbol{\sf v}}  \big] \, \textrm{d}V ~-\, \int_{\Gamma_\kappa} \llbracket \exs\bar{{\boldsymbol{\sf v}}}' \rrbracket \xxs\sip\xxs \bK_{\nxs\kappa} \xxs \llbracket \xxs \hat{\boldsymbol{\sf v}} \xxs \rrbracket  \, \textrm{d}S. 
\end{aligned}
\eeq

Subtracting~\eqref{SeQ4} from~\eqref{Wik-GE*}, one finds 
\beq\lb{-TXX}
\begin{aligned}
- \langle \xxs{T_\kappa{\boldsymbol{X}},{\boldsymbol{X}'}} \rangle ~=~ & \int_{\mD^\star_\kappa  \setminus  \overline{\Gamma_{\kappa}}} \xxs \nabla (\bar{{\boldsymbol{\sf u}}}'  \nxs+\xxs \bar{{\boldsymbol{\sf v}}}') \colon \! (\bC_{\kappa\nxs}-\bC) \xxs\colon \! \nabla (\hat{\boldsymbol{\sf u}} \xxs+\xxs\hat{\boldsymbol{\sf v}}) \exs \textrm{d}V ~+\, \int_{{\mathcal{B}}_\kappa^- \cup\xxs  \mD^\star_\kappa\nxs  \setminus  \overline{\Gamma_{\kappa}}}  \nabla \exs \bar{{\boldsymbol{\sf v}}}' \colon\!\xxs \bC \exs\colon\! \nabla \hat{\boldsymbol{\sf v}} \exs \textrm{d}V ~-\, \\*[0.2mm]  
-~&   \int_{\Gamma_\kappa \cap \exs\overline{\mD^\star_\kappa}} \big{[} \xxs {\bt}[\xxs \hat{\boldsymbol{\sf u}} \xxs] \cdot \llbracket \exs \bar{{\boldsymbol{\sf v}}}' \rrbracket \xxs+\xxs \bar{\bt}[\xxs{\boldsymbol{\sf u}}'] \nxs\cdot \llbracket \xxs\hat{\boldsymbol{\sf v}}\xxs \rrbracket \xxs \big{]}  \xxs \textrm{d}S  ~-\, \int_{\Gamma_\kappa \nxs\!\setminus \overline{\mD^\star_\kappa}} \big{[} \xxs \boldsymbol{\psi} \nxs\cdot \llbracket \exs\bar{{\boldsymbol{\sf v}}}' \rrbracket \xxs+\xxs\bar{\boldsymbol{\psi}}' \nxs\cdot \llbracket \xxs\hat{\boldsymbol{\sf v}}\xxs \rrbracket \xxs \big{]} \xxs \textrm{d}S ~-\, \\*[0.2mm]  
-~& \int_{\partial \mD^o_\kappa}\big{[} \xxs  \boldsymbol{\phi} \nxs \cdot \bar{{\boldsymbol{\sf v}}}'  \xxs+\xxs \bar{\boldsymbol{\phi}}' \nxs \cdot \hat{\boldsymbol{\sf v}} \xxs+\xxs \bar{\boldsymbol{\phi}}' \nxs \cdot \hat{\boldsymbol{\sf u}}_{\boldsymbol{\phi}} \xxs \big{]} \xxs \textrm{d}S ~+\, \int_{\Gamma_\kappa} \llbracket \exs\bar{{\boldsymbol{\sf v}}}' \rrbracket \xxs\sip\xxs \bK_{\nxs\kappa} \xxs \llbracket \xxs \hat{\boldsymbol{\sf v}} \xxs \rrbracket   \, \textrm{d}S ~+\,  \\*[0.2mm]  
+~& \int_{\mD^\star_\kappa  \setminus  \overline{\Gamma_{\kappa}}} \xxs \omega^2(\rho -\rho_\kappa)\xxs (\bar{{\boldsymbol{\sf u}}}'  \nxs+\xxs \bar{{\boldsymbol{\sf v}}}') \!\cdot\nxs (\hat{\boldsymbol{\sf u}} \xxs+\xxs\hat{\boldsymbol{\sf v}})  \exs   \textrm{d}V \exs-\, \int_{{\mathcal{B}}_\kappa^- \cup\xxs  \mD^\star_\kappa\nxs  \setminus  \overline{\Gamma_{\kappa}}} \rho \xxs  \omega^2 \exs \bar{{\boldsymbol{\sf v}}}' \sip\exs \hat{\boldsymbol{\sf v}} \exs \textrm{d}V, 
\end{aligned}
\eeq

On the other hand, adding~\eqref{SeQ4} to~\eqref{Wik-GE*}, the result can be recast as 
\beq\lb{+TXX}
\begin{aligned}
&\langle \xxs{T_\kappa{\boldsymbol{X}},{\boldsymbol{X}'}} \rangle \,=\, \int_{\mD^\star_\kappa  \setminus  \overline{\Gamma_{\kappa}}}  \nabla \bar{\boldsymbol{\sf u}}'  \nxs \colon \! (\bC-\bC_{\kappa\nxs}) \xxs\colon \!\nxs \nabla \hat{\boldsymbol{\sf u}} \, \textrm{d}V ~+\, \int_{\mD^\star_\kappa  \setminus  \overline{\Gamma_{\kappa}}}  \nabla \bar{{\boldsymbol{\sf v}}}' \colon \nxs \bC_{\kappa} \xxs\colon \! \nabla \hat{\boldsymbol{\sf v}} \, \textrm{d}V ~+ \exs  \\*[0.2mm]   
& \int_{\mD^\star_\kappa  \setminus  \overline{\Gamma_{\kappa}}} \xxs \big{[} \nabla \bar{\boldsymbol{\sf u}}' \colon \! (\bC-\bC_{\kappa\nxs}) \xxs\colon \! \nabla \hat{\boldsymbol{\sf v}} \, \textrm{d}V \xxs-\xxs  \nabla \bar{{\boldsymbol{\sf v}}}' \colon \! (\bC-\bC_{\kappa\nxs}) \xxs\colon \! \nabla \hat{\boldsymbol{\sf u}} \xxs \big{]} \, \textrm{d}V ~+\,  \int_{{\mathcal{B}}_\kappa^-} \nabla \exs \bar{{\boldsymbol{\sf v}}}' \colon\!\xxs \bC \exs\colon\! \nabla \hat{\boldsymbol{\sf v}} \, \textrm{d}V ~+ \exs  \\*[0.2mm]   
+&   \int_{\Gamma_\kappa \cap \exs\overline{\mD^\star_\kappa}} \big{[} \xxs \bar{\bt}[\xxs{\boldsymbol{\sf u}}'] \nxs\cdot \llbracket \xxs\hat{\boldsymbol{\sf v}}\xxs \rrbracket \xxs - \xxs {\bt}[\xxs \hat{\boldsymbol{\sf u}} \xxs] \cdot \llbracket \exs \bar{{\boldsymbol{\sf v}}}' \rrbracket \xxs \big{]}  \, \textrm{d}S  \,+\, \int_{\Gamma_\kappa \nxs\!\setminus \overline{\mD^\star_\kappa}} \big{[} \xxs \bar{\boldsymbol{\psi}}' \nxs\cdot \llbracket \xxs\hat{\boldsymbol{\sf v}}\xxs \rrbracket \xxs - \xxs \boldsymbol{\psi} \nxs\cdot \llbracket \exs\bar{{\boldsymbol{\sf v}}}' \rrbracket \xxs \big{]} \, \textrm{d}S \,+\, \\*[0.2mm] 
+&  \int_{\partial \mD^o_\kappa} \big{[} \xxs \bar{\boldsymbol{\phi}}' \nxs \cdot (\hat{\boldsymbol{\sf u}}_{\boldsymbol{\phi}} \xxs+\xxs\hat{\boldsymbol{\sf v}}) \xxs - \xxs \boldsymbol{\phi} \nxs \cdot \bar{{\boldsymbol{\sf v}}}' \xxs \big{]} \, \textrm{d}S ~+\, \int_{\Gamma_\kappa} \llbracket \exs\bar{{\boldsymbol{\sf v}}}' \rrbracket \xxs\sip\xxs \bK_{\nxs\kappa} \xxs \llbracket \xxs \hat{\boldsymbol{\sf v}} \xxs \rrbracket  \, \textrm{d}S ~-\,  \\*[0.2mm]   
-&\int_{\mD^\star_\kappa  \setminus  \overline{\Gamma_{\kappa}}} \omega^2(\rho -\rho_\kappa) \xxs(\bar{\boldsymbol{\sf u}}' \nxs-\xxs \bar{{\boldsymbol{\sf v}}}') \!\cdot\nxs (\hat{\boldsymbol{\sf u}} \xxs+\xxs\hat{\boldsymbol{\sf v}})  ~-\, \int_{{\mathcal{B}}_\kappa^- \cup\xxs  \mD^\star_\kappa\nxs  \setminus  \overline{\Gamma_{\kappa}}}   \rho \xxs  \omega^2 \exs \bar{{\boldsymbol{\sf v}}}' \sip\exs \hat{\boldsymbol{\sf v}}  \, \textrm{d}V.
\end{aligned}
\eeq 
In light of~\eqref{-TXX} and~\eqref{+TXX}, define $T_\circ^\pm\!\colon {\mathfrak{S}}(H_{\Delta}) \rightarrow \, \tilde{S}(\mD^\star_\kappa \cup \Gamma_\kappa \nxs \cup \partial \mD^o_\kappa)$ such that 
\beq\lb{T0XX}
\begin{aligned}
- \langle \xxs{T_\circ^-{\boldsymbol{X}},{\boldsymbol{X}'}} \rangle ~=~ & \int_{\mD^\star_\kappa  \setminus  \overline{\Gamma_{\kappa}}} \xxs \nabla (\bar{{\boldsymbol{\sf u}}}'  \nxs+\xxs \bar{{\boldsymbol{\sf v}}}') \colon \! (\bC_{\kappa\nxs}-\bC) \xxs\colon \! \nabla (\hat{\boldsymbol{\sf u}} \xxs+\xxs\hat{\boldsymbol{\sf v}}) \exs \textrm{d}V ~+\, \\*[0.2mm]  
+~&\int_{{\mathcal{B}}_\kappa^- \cup\xxs  \mD^\star_\kappa\nxs  \setminus  \overline{\Gamma_{\kappa}}}  \nabla \exs \bar{{\boldsymbol{\sf v}}}' \colon\!\xxs \bC \exs\colon\! \nabla \hat{\boldsymbol{\sf v}} \exs \textrm{d}V \,+\, \int_{\mD^\star_\kappa\nxs  \setminus  \overline{\Gamma_{\kappa}}}  \bar{{\boldsymbol{\sf u}}}' \sip\exs\xxs \hat{\boldsymbol{\sf u}}  \, \textrm{d}V, \\*[1.2mm]  
\langle \xxs{T_\circ^+{\boldsymbol{X}},{\boldsymbol{X}'}} \rangle ~=~ &\int_{\mD^\star_\kappa  \setminus  \overline{\Gamma_{\kappa}}}  \nabla \bar{\boldsymbol{\sf u}}'  \nxs \colon \! (\bC-\bC_{\kappa\nxs}) \xxs\colon \!\nxs \nabla \hat{\boldsymbol{\sf u}} \, \textrm{d}V ~+\, \int_{\mD^\star_\kappa  \setminus  \overline{\Gamma_{\kappa}}}  \nabla \bar{{\boldsymbol{\sf v}}}' \colon \nxs \bC_{\kappa} \xxs\colon \! \nabla \hat{\boldsymbol{\sf v}} \, \textrm{d}V ~+ \exs  \\*[0.2mm]  
\end{aligned}
\eeq 
\beq
\begin{aligned}
+~& \int_{\mD^\star_\kappa  \setminus  \overline{\Gamma_{\kappa}}} \xxs \big{[} \nabla \bar{\boldsymbol{\sf u}}' \colon \! (\bC-\bC_{\kappa\nxs}) \xxs\colon \! \nabla \hat{\boldsymbol{\sf v}} \, \textrm{d}V \xxs-\xxs  \nabla \bar{{\boldsymbol{\sf v}}}' \colon \! (\bC-\bC_{\kappa\nxs}) \xxs\colon \! \nabla \hat{\boldsymbol{\sf u}} \xxs \big{]} \, \textrm{d}V ~+\, \\*[0.2mm] 
+~& \int_{{\mathcal{B}}_\kappa^-} \nabla \exs \bar{{\boldsymbol{\sf v}}}' \colon\!\xxs \bC \exs\colon\! \nabla \hat{\boldsymbol{\sf v}} \, \textrm{d}V \,+\, \int_{\mD^\star_\kappa\nxs  \setminus  \overline{\Gamma_{\kappa}}}  \bar{{\boldsymbol{\sf u}}}' \sip\exs\xxs \hat{\boldsymbol{\sf u}}  \, \textrm{d}V.
\end{aligned}
\eeq
Given Assumption~\ref{ITPA}, it is evident that $\mathfrak{R}(e^{i\theta} T_\circ^-)$ is coercive on ${\mathfrak{S}}(H_{\Delta})$ for $\theta = [0 \,\, \pi/2)$ provided that the first of Assumption~\ref{Tass} holds. In the second case of the latter, however, one may show that $\mathfrak{R}(T_\circ^+)$ is coercive on ${\mathfrak{S}}(H_{\Delta})$ by following the argument used in the proof of~\cite[Theorem 2.47]{cako2016}. Now, by deploying the Rellich compact embeddings along with the regularity of the trace operator, one concludes that   
\[
T_\text{c}  ~\colon\!\! =~ \mathfrak{R}(T_\kappa \xxs-\xxs T_\circ^\pm) \colon {\mathfrak{S}}(H_{\Delta}) \rightarrow \, \tilde{S}(\mD^\star_\kappa \cup \Gamma_\kappa \nxs \cup \partial \mD^o_\kappa)
\]
 is compact.

\bibliography{inverse,composites}

\end{document}